\documentclass[a4paper,11pt,reqno]{amsart}

\usepackage{url}
\usepackage[colorlinks=true, linkcolor=blue, citecolor=ForestGreen]{hyperref}
\usepackage{cite}

\usepackage{latexsym,amsmath,amssymb,amsfonts,amsthm}
\usepackage{mathrsfs,mathtools}
\usepackage{dsfont}
\usepackage{cases}
\usepackage{bbm}
\usepackage[utf8]{inputenc}

\usepackage[usenames, dvipsnames]{color}
\usepackage[usenames, dvipsnames, cmyk]{xcolor}
\usepackage{graphicx}
\usepackage[scriptsize,hang,raggedright]{subfigure}
\usepackage{multirow}
\usepackage{enumerate}
\usepackage{enumitem}

\usepackage{a4wide}
\usepackage{epsfig,epstopdf}

\usepackage{tikz}
\usepackage{pgf,pgfplots}
\usetikzlibrary{calc,decorations.markings,decorations.pathmorphing,decorations.shapes,shapes,arrows,shapes.geometric,patterns,fadings}

\numberwithin{equation}{section}

\definecolor{grey}{rgb}{.7,.7,.7}
\definecolor{mygreen}{rgb}{0.1,0.75,0.2}

\newcommand{\xupref}[2]{\hspace{-0.3ex}\stackrel{\eqref{#1}}{#2}} 

\newtheorem{theorem}{Theorem}[section]
\newtheorem{proposition}[theorem]{Proposition}
\newtheorem{lemma}[theorem]{Lemma}

\theoremstyle{remark}
\newtheorem{remark}[theorem]{Remark}
\theoremstyle{definition}

\newcommand{\e}{\varepsilon}
\newcommand{\eps}{\varepsilon}
\newcommand{\vphi}{\varphi}

\newcommand{\N}{\mathbb N}

\newcommand{\R}{\mathbb R}

\newcommand{\Lu}{\mathcal{L}^{1}}

\newcommand{\ds}{\displaystyle}
\newcommand{\ts}{\textstyle}
\newcommand{\loc}{\mathrm{loc}}
\newcommand{\wto}{\rightharpoonup}

\newcommand{\dd}{\,\mathrm{d}}
\renewcommand{\setminus}{\backslash}

\newcommand{\defeq}{\coloneqq}
\renewcommand{\leq}{\leqslant}
\renewcommand{\geq}{\geqslant}

\newcommand{\BV}{\mathrm{BV}}
\newcommand{\SBV}{\mathrm{SBV}}
\newcommand{\F}{\mathcal{F}}
\newcommand{\Feps}{\mathcal{F}_\e}
\newcommand{\G}{\mathcal{G}}
\newcommand{\Gh}{\widehat{\mathcal{G}}}
\newcommand{\fbar}{\bar{f}}

\newcommand{\stress}{\sigma_{\mathrm c}}
\newcommand{\sfrac}{s_{\mathrm{frac}}}
\newcommand{\mbar}{\bar{m}}
\newcommand{\midp}{{\textstyle\frac{L}{2}}}

\title[Critical points of cohesive fracture energies]{Convergence of critical points for a phase-field approximation of 1D cohesive fracture energies}

\author{Marco Bonacini}
\address[Marco Bonacini]{Department of Mathematics, University of Trento, Italy}
\email{marco.bonacini@unitn.it}

\author{Flaviana Iurlano}
\address[Flaviana Iurlano]{Università degli Studi di Genova, Dipartimento di Matematica, Genoa, Italy}
\email{flaviana.iurlano@unige.it}

\date{\today}                                        

\subjclass[2020]{}
\keywords{}
\thanks{This is a preprint version of an article that has been accepted for publication in \textit{Calculus of Variations and Partial Differential Equations}, published by Springer.
}

\begin{document}

\begin{abstract}
Variational models for cohesive fracture are based on the idea that the fracture energy is released gradually as the crack opening grows. Recently, \cite{ConFocIur16} proposed a variational approximation via $\Gamma$-convergence of a class of cohesive fracture energies by phase-field energies of Ambrosio-Tortorelli type, which may be also used as regularization for numerical simulations. In this paper we address the question of the asymptotic behaviour of critical points of the phase-field energies in the one-dimensional setting: we show that they converge to a selected class of critical points of the limit functional. Conversely, each critical point in this class can be approximated by a family of critical points of the phase-field functionals.
\end{abstract}

\maketitle


\section{Introduction}\label{sec:intro}

Fracture models describe the evolution of surface cracks in elastic materials subjected to external loads or boundary conditions. The literature distinguishes between brittle models and cohesive models (also known as Griffith and Barenblatt models respectively).
The former treat fracture as an instantaneous phenomenon: the body deforms elastically until a crack surface appears; the crack energy is instantaneously released and there is no transmission of force across the crack surface. The latter treat fracture as a gradual phenomenon: the bonds between the lips progressively weaken; the crack energy is released with the growth of the crack opening and the force transmitted across the crack surface gradually reduces to zero. Thanks to these features, cohesive models are better suited than brittle models for describing crack nucleation. We refer the interested reader to the book \cite{BFM} and references therein for a comprehensive comparison between brittle and cohesive, and we work from now on in the cohesive setting.

The variational study of cohesive fracture started in the late 1990s and has been earning interest ever since \cite{Alm17,ACFS,BonConIur21,BCFR2022,BraDMGar99,CCF,C,CT,CLO,DMG,DMZ,DelPie13,DPT1998,DelPieTru09,LS,NS2017,NS2021,R2023,R2019,T2018,TZ2017}. The appropriate variational setting to model a cohesive fracture process was shown to be the space of functions of bounded variation or bounded deformation, allowing to describe the crack as the jump set of a discontinuous displacement, and the total energy as a competition between bulk and surface contributions \cite{FM1998}. The presence of free discontinuities, making the numerical treatment highly complex, lead to the development of regularized phase-field theories \cite{ABS,AF,BonConIur21,ConFocIur16,ConFocIur22,DMOT}, in the spirit of the classical Allen-Cahn (or Modica-Mortola) approximation for phase transitions \cite{M1987}, and the Ambrosio-Tortorelli approximation of the Mumford-Shah functional for image segmentation or brittle fracture \cite{AmbTor90,AmbTor92,G2005}. The general approach of these works is to construct sequences of purely bulk energies, whose variables are forced to engage transitions in thin concentration sets, and to show the convergence of corresponding global minimizers to a global minimizer of the given energy as the thickness of the concentration sets vanishes.

Although this kind of results usually marks decisive enhancements in the mathematical comprehension of the corresponding phenomena, its energy-based formulation and its global minimization focus may be not completely satisfactory from the mechanical point of view. Indeed, fracture evolution might realistically occur along critical states rather than following global minimizers. In addition, numerical schemes based on alternate minimization for the regularized energies typically converge to critical points of the limit energy; hence, the sole convergence of global minimizers does not provide a complete theoretical justification of the adoption of the phase-field models for numerical simulations, see for example \cite{AleMarVid14,Bourdin,BFM,FreIur,LCM,Wu,WuNgu}.

This motivated the investigation of better converging properties of the proposed regular functionals. On the one side, it lead to the study of the convergence of the corresponding gradient flows, see \cite{AlmBelNeg19,BabMill2014} for brittle fracture. On the other side, it lead to the study of the convergence of critical points, see in particular \cite{HutTon00,LeSte19,LucMod89,PadTon98,RoeTon08,Ton02,Ton05} for the Allen-Cahn functionals, and \cite{BabMilRodb,BabMilRoda,FraLeSer09,Le10} for the Ambrosio-Tortorelli functionals.

In this paper we address the latter question in the context of one-dimensional cohesive fracture: we study the asymptotic behaviour of the critical points of the regularized functionals proposed in \cite{ConFocIur16}. Thanks to its nature, the cohesive case allows for a deeper and more complete analysis with respect to the brittle case. Our approach heavily relies on one-dimensional arguments and the analysis is at the moment limited to this setting; its possible extension to the higher-dimensional case, in the spirit of the recent work \cite{BabMilRoda} in the context of brittle fracture, poses significant challenges.

The rest of this Introduction is organized as follows: in Section~\ref{subsec:critpoints} we provide notation and properties of the sharp cohesive model and of its critical points; in Section~\ref{subsec:approx} we introduce our regularized models, which are slight modifications of those proposed in \cite{BonConIur21,ConFocIur16}. A few additional technical assumptions will be needed. Our main results will be stated in Section~\ref{subsec:results}.


\subsection{The cohesive fracture energy and its critical points} \label{subsec:critpoints}
We first introduce a one-dimensional cohesive fracture energy for a bar of total length $L>0$ (at rest), and total elongation $a>0$.

The deformation of the bar is described by a function of bounded variation $u\in\BV(0,L)$, whose distributional derivative is a bounded Radon measure on $(0,L)$ that can be written as 
$$
Du = u'\dd x + D^cu + \sum_{x\in J_u}[u](x)\delta_x,
$$
where $u'\in L^1(0,L)$ denotes the density of the absolutely continuous part (with respect to the Lebesgue measure), $D^cu$ is the Cantor part, $[u](x)\defeq u(x^+)-u(x^-)$, where $u(x^+)$ and $u(x^-)$ are the approximate limits from the right and from the left of $u$ at $x$ respectively, and $J_u\defeq\{x\in(0,L) : [u](x)\neq 0\}$ is the set of essential discontinuities. Since we want to include in the energy the boundary conditions, we set $u(0^-)=0$, $u(L^+)=a$, we extend the definition of $[u](x)\defeq u(x^+)-u(x^-)$ also at the endpoints $x=0$, $x=L$, and we let $J_u^a\defeq\{x\in[0,L] : [u](x)\neq 0\}$.

The \emph{cohesive energy} of the bar is defined as
\begin{equation} \label{def:F}
\Phi(u) \defeq \int_0^L \phi(u')\dd x + \sum_{x\in J_u^a}g(|[u](x)|) + \stress|D^cu|(0,L) \qquad \text{for } u\in\BV(0,L).
\end{equation}
Here $\stress\in(0,+\infty)$ and the \emph{elastic energy density} $\phi:\R\to[0,+\infty)$ is given by
\begin{equation} \label{def:phi}
\phi(\xi)\defeq
{\begin{cases}
\xi^2
& \text{if } |\xi| \leq \frac{\stress}{2},\\
\stress|\xi| - \frac{\stress^2}{4}
& \text{if  } |\xi| > \frac{\stress}{2}.
\end{cases}}
\end{equation}
For the \emph{cohesive energy density} we assume that $g:[0,+\infty)\to[0,+\infty)$ is a nondecreasing function of class $\mathrm{C}^1$ with $g(0)=0$, $g'(0)=\stress$. We further assume that $\{g'>0\}=[0,\sfrac)$ for some $\sfrac\in(0,+\infty]$, and that $g'$ is strictly decreasing in $[0,\sfrac)$.

It is convenient to include in the energy also the non-interpenetration constraint that the singular part $D^su\defeq D^cu + \sum_{x\in J_u^a}[u](x)\delta_x$ of $Du$ is a nonnegative measure: we therefore define $\widetilde{\Phi}(u)\defeq\Phi(u)$ if $u\in\BV(0,L)$ with $D^su\geq0$ (so that in particular $u(0^+)\geq0$, $u(L^-)\leq a$), and $\widetilde{\Phi}(u)=+\infty$ otherwise.

\emph{Critical points} of the functional $\widetilde{\Phi}$ are functions $u$ such that $\widetilde{\Phi}(u)<+\infty$ and
$$
\liminf_{t\to 0^+}\frac{\widetilde{\Phi}(u+tv)-\widetilde{\Phi}(u)}{t} \geq 0 \qquad\text{for all }v\in \BV(0,L).
$$
These are studied in details in \cite{BraDMGar99} (see also \cite{DPT1998}). We stress that nonnegativity of the unilateral lower limit as $t\to0^+$ is required in place of the usual vanishing of the first variation. This is a standard way to give a meaningful notion of critical point in presence of a noninterpenetration constraint.
Mechanically, such condition provides a critical stress to nucleation, in the sense that nucleation of a crack point is only possible when the stress reaches the critical value.

By \cite[Theorem~6.3]{BraDMGar99} one has that a function $u\in\BV(0,L)$ with $D^su\geq0$ is a critical point of $\widetilde{\Phi}$ if and only if there exists a constant $\sigma\in\R$ such that
\begin{enumerate}
\item $\sigma\leq\stress$,
\item $\phi'(u')=\sigma$ a.e.\ in $(0,L)$,
\item $g'([u])=\sigma$ on $J_u^a$,
\item $D^cu=0$ if $\sigma<\stress$.
\end{enumerate}
As observed in \cite[Remark~6.6]{BraDMGar99}, in the model described by the energy $\widetilde{\Phi}$ the quantity $\phi'(u')$ represents the stress in the elastic part of the bar due to the deformation gradient $u'$; $g'([u](x))$ represents the stress transmitted through the points of $J_u^a$ of the reference configuration (where there is concentration of the strain); $D^cu$ can be seen as a singular, not concentrated strain which transmits a stress $\stress$ and can be present only if $\sigma=\stress$, whereas if $\sigma<\stress$ a critical point is necessarily in $\SBV(0,L)$. The constant $\stress$ is the maximum possible stress for an equilibrium configuration.

In view of the previous conditions, we can give an explicit description of all critical points belonging to $\SBV(0,L)$ corresponding to a prescribed elongation $a>0$ (see Figure~\ref{fig:critical-points}):
\begin{enumerate}[label = (\alph*)]
\item (\emph{elastic states}) $u(x)=\frac{a}{L}x$ in $(0,L)$;
\item (\emph{pre-fractured states}) $\sigma\in(0,\stress)$, $\# J_u^a=k$, $1\leq k <+\infty$, $u'=\frac{\sigma}{2}$ a.e.\ in $(0,L)$, $[u](x)=s_0$ for all $x\in J_u^a$ where $s_0\in(0,\sfrac)$ obeys $g'(s_0)=\sigma$, and $\frac{\sigma}{2}L+ks_0=a$;
\item (\emph{fractured states}) $\sigma=0$, $\# J_u^a=k$, $1\leq k <+\infty$, $u'=0$ a.e.\ in $(0,L)$, and $[u](x)\geq\sfrac$ for all $x\in J_u^a$, with $\sum_{x\in J_u^a}[u](x)=a$.
\end{enumerate}
Notice that in the first case $\sigma=\frac{2a}{L}\wedge\stress\in[0,\stress]$ and all the possible slopes $u'=\frac{a}{L}$ are allowed; if $\frac{2a}{L}\geq\stress$ then $\sigma=\stress$. In the second case $u$ is piecewise affine with constant slope $\frac{\sigma}{2}$ and with a finite number of jumps of the same amplitude. In the third case $\sigma=0$, $u$ is piecewise constant with a finite number of jumps with amplitude larger than $\sfrac$, and this case can only occur if $\sfrac<+\infty$ and $a\geq\sfrac$.

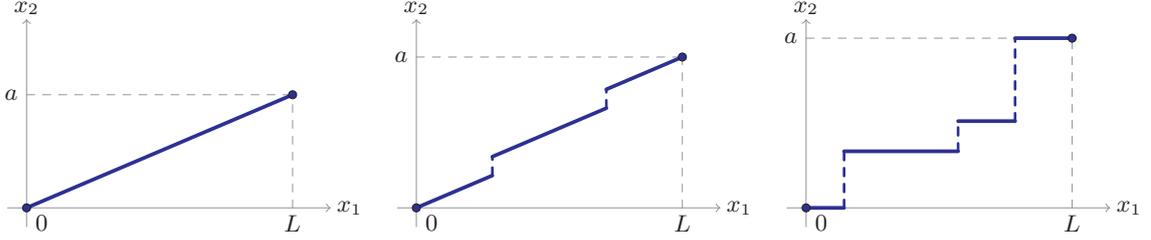
\begin{figure}
\definecolor{qqqqff}{rgb}{0,0,1}
\definecolor{qqwuqq}{rgb}{0,0.4,0}
\begin{tikzpicture}[scale=0.5,line cap=round,line join=round]
\clip(-1,-1) rectangle (9,7);
\draw [->,line width=0.3pt,color=black!50] (-0.5,0)-- (8,0);
\draw [->,line width=0.3pt,color=black!50] (0,-0.5)-- (0,5);
\draw [line width=0.3pt,color=black!50,dash pattern=on 3pt off 3pt] (7,0)-- (7,3);
\draw [line width=0.3pt,color=black!50,dash pattern=on 3pt off 3pt] (0,3)-- (7,3);
\draw [line width=1.4pt,color=qqqqff] (0,0)-- (7,3);
\begin{footnotesize}
\draw (8.5,0) node {$x_1$};
\draw (0,5.3) node {$x_2$};
\draw [fill=qqqqff] (7,3) circle (3pt);
\draw [fill=qqqqff] (0,0) circle (3pt);
\draw (0.4,-0.4) node {$0$};
\draw (7,-0.4) node {$L$};
\draw (-0.4,3) node {$a$};
\end{footnotesize}
\end{tikzpicture}
\begin{tikzpicture}[scale=0.5,line cap=round,line join=round]
\clip(-1,-1) rectangle (9,7);
\draw [->,line width=0.3pt,color=black!50] (-0.5,0)-- (8,0);
\draw [->,line width=0.3pt,color=black!50] (0,-0.5)-- (0,5);
\draw [line width=0.3pt,color=black!50,dash pattern=on 3pt off 3pt] (7,0)-- (7,4);
\draw [line width=0.3pt,color=black!50,dash pattern=on 3pt off 3pt] (0,4)-- (7,4);
\draw [line width=1.4pt,color=qqqqff] (0,0)--(2,6/7);
\draw [line width=1pt,color=qqqqff,dashed] (2,6/7)--(2,19/14);
\draw [line width=1.4pt,color=qqqqff] (2,19/14)--(5,37/14);
\draw [line width=1pt,color=qqqqff,dashed] (5,37/14)--(5,22/7);
\draw [line width=1.4pt,color=qqqqff] (5,22/7)--(7,4);
\begin{footnotesize}
\draw (8.5,0) node {$x_1$};
\draw (0,5.3) node {$x_2$};
\draw [fill=qqqqff] (7,4) circle (3pt);
\draw [fill=qqqqff] (0,0) circle (3pt);
\draw (0.4,-0.4) node {$0$};
\draw (7,-0.4) node {$L$};
\draw (-0.4,4) node {$a$};
\end{footnotesize}
\end{tikzpicture}
\begin{tikzpicture}[scale=0.5,line cap=round,line join=round]
\clip(-1,-1) rectangle (9,7);
\draw [->,line width=0.3pt,color=black!50] (-0.5,0)-- (8,0);
\draw [->,line width=0.3pt,color=black!50] (0,-0.5)-- (0,5);
\draw [line width=0.3pt,color=black!50,dash pattern=on 3pt off 3pt] (7,0)-- (7,4.5);
\draw [line width=0.3pt,color=black!50,dash pattern=on 3pt off 3pt] (0,4.5)-- (7,4.5);
\draw [line width=1.4pt,color=qqqqff] (0,0)--(1,0);
\draw [line width=1pt,color=qqqqff,dashed] (1,0)--(1,1.5);
\draw [line width=1.4pt,color=qqqqff] (1,1.5)--(4,1.5);
\draw [line width=1pt,color=qqqqff,dashed] (4,1.5)--(4,2.3);
\draw [line width=1.4pt,color=qqqqff] (4,2.3)--(5.5,2.3);
\draw [line width=1pt,color=qqqqff,dashed] (5.5,2.3)--(5.5,4.5);
\draw [line width=1.4pt,color=qqqqff] (5.5,4.5)--(7,4.5);
\begin{footnotesize}
\draw (8.5,0) node {$x_1$};
\draw (0,5.3) node {$x_2$};
\draw [fill=qqqqff] (7,4.5) circle (3pt);
\draw [fill=qqqqff] (0,0) circle (3pt);
\draw (0.4,-0.4) node {$0$};
\draw (7,-0.4) node {$L$};
\draw (-0.4,4.5) node {$a$};
\end{footnotesize}
\end{tikzpicture}
	\caption{Critical points in $\SBV(0,L)$ of the cohesive energy \eqref{def:F}. Left: elastic state $u(x)=\frac{a}{L}x$. Center: pre-fractured state with $\sigma\in(0,\stress)$. Right: fractured state with $\sigma=0$.}
	\label{fig:critical-points}
\end{figure}

\begin{remark}\label{rem:relaxation}
By the results in \cite{BouBraBut95} the functional $\widetilde{\Phi}$ is the lower semicontinuous envelope with respect to the $L^1$-convergence of the energy
\begin{equation} \label{def:energy-not-relaxed}
u \mapsto \int_0^L |u'|^2\dd x + \sum_{x\in J_u^a}g(|[u](x)|) \qquad\text{for }u\in\SBV(0,L)\text{ with }[u]\geq0.
\end{equation}
As it is observed in \cite[Remark~6.5]{BraDMGar99}, every critical point of this functional is also a critical point of $\widetilde{\Phi}$, and every critical point of $\widetilde{\Phi}$ with $\sigma<\stress$ is a critical point of \eqref{def:energy-not-relaxed}. For $\sigma=\stress$, there are critical points of $\widetilde{\Phi}$ which are not critical for \eqref{def:energy-not-relaxed}: in particular, elastic states $u(x)=\frac{a}{L}x$ with $\frac{2a}{L}>\stress$, and critical points with $D^cu\neq0$. Such configurations, however, are not \emph{local minimizers} of $\widetilde{\Phi}$: in particular, $\widetilde{\Phi}$ and \eqref{def:energy-not-relaxed} have the same (local) minimizers, see \cite[Theorem~7.2]{BraDMGar99}.
\end{remark}  

\begin{remark}\label{rmk:critical-points}
It is instructive to consider a quasistatic evolution for the cohesive {model \eqref{def:energy-not-relaxed}}, corresponding to a time-dependent elongation $a\geq0$, monotone increasing in time: at each time we assume that the deformation $u_a$ of the bar is a critical state satisfying the boundary conditions $u_a(0)=0$, $u_a(L)=a$ (in the weak sense, as specified above). Initially, for small values of $a$, the response of the bar is purely elastic and the evolution follows the elastic critical points $u_a(x)=\frac{a}{L}x$, until the critical stress is reached ($\frac{a}{L}=\frac{\stress}{2}$). At this point, that is for $a=\frac12\stress L$, it is expected that the state of the bar switches to a pre-fractured state, a pre-fracture point appears, and the amplitude of the crack continuously increases from $0$ to $\sfrac$ as {{$a$ increases from $\frac12\stress L$ to $\sfrac$}}. In this case the displacement $u_a$ has constant slope $\frac{\sigma}{2}$, with $\sigma\in(0,\stress)$, and jump {amplitude} $s_0\in(0,\sfrac)$, which should satisfy the compatibility conditions
\begin{equation} \label{eq:rmk-intro}
g'(s_0)=\sigma, \qquad \frac{\sigma}{2}L + s_0 = a.
\end{equation}
{The limit case $a=\sfrac$ corresponds to the complete fracture state, characterized by jump amplitude exactly equal to the boundary datum $s_0=\sfrac=a$ and vanishing stress $\sigma=0$.}

{Concerning the compatibility conditions above, we remark that the existence of a solution $(\sigma,s_0)\in(0,\stress)\times(0,\sfrac)$ to \eqref{eq:rmk-intro} for values $a>\frac12 \stress L$, with the property that $s_0\to 0^+$ and $\sigma\to\stress^-$ as $a\to(\frac12 \stress L)^+$, is not a priori guaranteed.} In case of nonexistence, then, once the critical stress is reached, the evolution {creates} instantaneously a jump of strictly positive amplitude, {possibly even} brittle without cohesive effects (see \cite[Section~9]{C} for an explicit example in two dimensions).

The behaviour of the function $g$ at the origin determines whether for elongations $a\to(\frac12\stress L)^+$ there are solutions of \eqref{eq:rmk-intro} such that $s_0\to0^+$. Assuming that $g$ satisfies the expansion
$$
g(s)=\stress s - \tilde{\ell}s^p + o(s^p) \qquad\text{as }s\to0
$$
for some $\tilde{\ell}>0$ and $p>1$, then it can be checked by elementary arguments that existence of solutions as above is guaranteed if $p>2$, whereas it fails if $p<2$. The case $p=2$ is critical, the existence of solutions depends on the length $L$ of the bar and fails for sufficiently large $L$; this is an instance of the well-known size effects in fracture. A suitable choice of our regularized models will produce in the limit a density $g$ with the previous asymptotic expansion, see Proposition~\ref{prop:g4}. We refer to \cite{DelPie13} for a more detailed discussion of the content of this remark.
\end{remark}


\subsection{Phase-field approximation} \label{subsec:approx}
Following \cite{ConFocIur16}, we now introduce a family of functionals $\Feps : L^1(0,L)\times L^1(0,L)\to[0,+\infty]$, depending on a real parameter $\e>0$, which approximate the cohesive energy density \eqref{def:F} in the sense of $\Gamma$-convergence. We let $f$ satisfy the following conditions, that will be assumed throughout the paper unless explicitly stated:
\begin{enumerate}[label = ($f$\arabic*)]
	\addtolength{\itemsep}{6pt}
\item\label{ass:f1} $f\in C^2([0,1); [0,+\infty))$, $f^{-1}(0)=\{0\}$;
\item\label{ass:f2} $\ds\lim_{s\to1^-}(1-s)f(s)=\stress$, with $\stress\in(0,+\infty)$;
\item\label{ass:f3} $\ds\frac{\dd}{\dd s}\bigl[(1-s)f(s)\bigr]>0$ for all $s\in(0,1)$;
\item\label{ass:f4} $\ds\frac{\dd}{\dd s}\biggl[\frac{(1-s)f'(s)}{f(s)}\biggr]<0$ for all $s\in(0,1)$;
\item\label{ass:f5} $\ds\lim_{s\to1^-}\frac{1}{(1-s)^3}\frac{\dd}{\dd s}\bigl[(1-s)f(s)\bigr]=+\infty$;
\item\label{ass:f6} the map $s\mapsto\sqrt{s}f(1-\sqrt{s})$ is convex.
\end{enumerate}
The previous conditions look very technical and a few comments are required: firstly, they are satisfied by a large class of functions, as Remark~\ref{rmk:examplef}  below shows. Secondly, the $\Gamma$-convergence result in \cite{ConFocIur16} holds under weaker assumptions, but the analysis of the model was further improved in \cite{BonConIur21} where conditions \ref{ass:f1}, \ref{ass:f2}, \ref{ass:f3}, \ref{ass:f6} were assumed in order to obtain more detailed properties of the limit functional, that we will recall in Section~\ref{sect:g}. Here, we also include the conditions \ref{ass:f4}--\ref{ass:f5} whose role is crucial for the analysis in Section~\ref{sect:ODE}. We believe that the optimal set of assumptions is \ref{ass:f1}--\ref{ass:f5}, but dropping the convexity condition \ref{ass:f6} would require a much finer analysis that goes beyond the scopes of this paper (see Remark~\ref{rmk:uniqueness}).

\begin{remark} \label{rmk:examplef}
Prototype functions with the properties above are the maps
\begin{equation} \label{eq:examplefq}
f_{(q)}(s)\defeq \frac{\stress\bigl[1-(1-s)^q\bigr]}{1-s}
\qquad\text{and}\qquad
f^{(p)}(s)\defeq \frac{(\stress+p(1-s))s^2}{1-s}
\end{equation}
for $\stress>0$, $q\in(0,2]$, and $p\in(-\stress,2\stress)$. (Notice that for $q\in(2,4)$ the function $f_{(q)}$ satisfies \ref{ass:f1}--\ref{ass:f5} but not \ref{ass:f6}.)
\end{remark}

We next need to truncate the function $f$ in a neighbourhood of $1$ in a smooth way. To this aim we fix points
\begin{equation} \label{def:seps2}
s_\e\in(0,1) \quad\text{such that}\quad s_\e\uparrow 1 \quad\text{and}\quad \sqrt{\e}f(s_\e)\to1
\end{equation}
as $\e\to0$. By \ref{ass:f2} and $\sqrt{\e}f(s_\e)\to1$ we easily deduce that 
\begin{equation} \label{def:seps}
\frac{1-s_\e}{\sqrt{\e}}\to\stress \qquad\text{as } \e\to0.
\end{equation}
We then define the function $f_\e:[0,1]\to\R$ by
\begin{equation} \label{def:feps}
f_\e(s) \defeq 
\begin{cases}
\sqrt{\e}f(s) & \text{if }0\leq s \leq s_\e,\\
\psi_\e(s) & \text{if }s_\e<s\leq 1,
\end{cases}
\end{equation}
where $\psi_\e:[0,1]\to[0,1]$ is any function satisfying the following conditions:
\begin{enumerate}[label = ($\psi$\arabic*)]
	\addtolength{\itemsep}{6pt}
\item\label{ass:psi1} $\psi_\e\in C^2([0,1])$ is monotone nondecreasing, $\psi_\e(1)=1$;
\item\label{ass:psi2} $\psi_\e(s_\e)=\sqrt{\e}f(s_\e)$, $\psi_\e'(s_\e)=\sqrt{\e}f'(s_\e)$;
\item\label{ass:psi3} the map $s\mapsto\frac{\psi_\e'(s)}{1-s}$ is monotone nonincreasing.
\end{enumerate}
 Notice that the condition \ref{ass:psi3} forces $\psi_\e'(1)=0$. The function $f_\e$ in \eqref{def:feps} is of class $C^1(0,1)$, and can be easily extended to a globally $C^1$-function on $\R$ by setting $f_\e(s)=\sqrt{\e}f'(0)s$ for $s<0$, $f_\e(s)=1$ for $s>1$. 

\begin{remark} \label{rmk:examplepsi}
An explicit family of functions with the properties above are the exponentials
\begin{equation*}
\psi_\e(s) \defeq 1 - \alpha_{\e}e^{-\frac{\beta_\e}{1-s}},
\qquad
\alpha_{\e}\defeq \bigl(1-\sqrt{\e}f(s_\e)\bigr)e^{\frac{\beta_\e}{1-s_\e}},
\quad
\beta_\e \defeq \frac{\sqrt{\e}f'(s_\e)(1-s_\e)^2}{1-\sqrt{\e}f(s_\e)}\,.
\end{equation*}
The choice of the coefficients $\alpha_{\e}$ and $\beta_\e$ guarantees that \ref{ass:psi2} holds. Moreover, it can be checked by an elementary computation that also \ref{ass:psi3} holds for $\e$ small. See Figure~\ref{fig:feps} for a numerical plot of the resulting function $f_\eps$.
\begin{figure}[ht]
	\begin{center}
		\includegraphics[width=6.5cm]{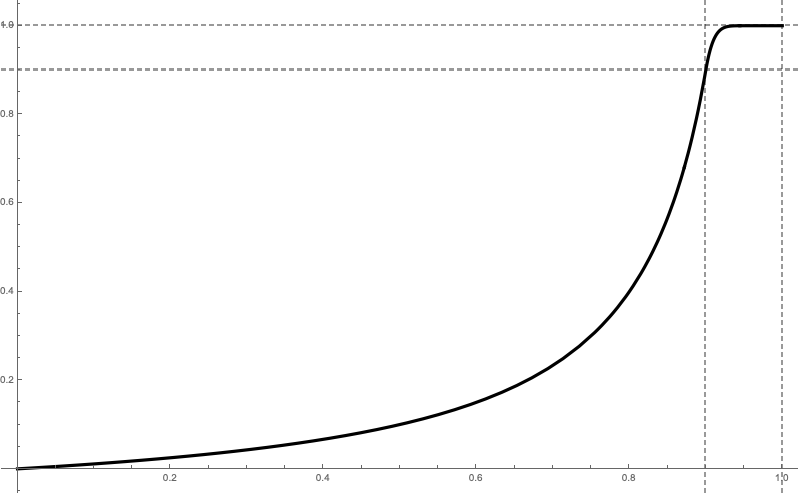}
	\end{center}
	\caption{Numerical plot of the function $f_\eps$ defined in \eqref{def:feps}, obtained by choosing the prototype function $f_{(q)}$ (with $q=1$) as in Remark~\ref{rmk:examplef} and the exponential junction $\psi_\e$ as in Remark~\ref{rmk:examplepsi}.}
	\label{fig:feps}
\end{figure}
\end{remark}

We next introduce a cohesive energy density $g:[0,+\infty)\to[0,+\infty)$, depending on $f$, by solving the following auxiliary optimal profile problem:
\begin{equation} \label{def:g}
g(s) \defeq \inf_{(\alpha,\beta)\in\mathcal{U}_s} \G(\alpha,\beta),
\end{equation}
where 
\begin{equation} \label{def:G}
\G(\alpha,\beta) \defeq \int_{-\infty}^{+\infty} \biggl( f^2(\beta)|\alpha'|^2 + \frac{(1-\beta)^2}{4} + |\beta'|^2 \biggr) \dd t,
\end{equation}
\begin{multline} \label{def:U}
\mathcal{U}_s \defeq \biggl\{ (\alpha,\beta)\in H^1_{\loc}(\R)\times H^1_{\loc}(\R) \,:\, \alpha'\in L^1(\R),\, \bigg|\int_{-\infty}^{+\infty} \alpha'(t)\dd t\bigg| = s, \\
0\leq\beta\leq1, \, \lim_{|t|\to+\infty}\beta(t)=1 \biggr\}
\end{multline}
(where $f^2(\beta)\defeq(f(\beta))^2$ and we adopt the convention $f(\beta)|\alpha'|=0$ if $\beta=1$ and $\alpha'=0$).
The properties of the function $g$ are discussed in Section~\ref{sect:g}. Here we just note that $g$ satisfies the conditions we required in Subsection~\ref{subsec:critpoints} for a cohesive surface energy density, see in particular Proposition~\ref{prop:g1}. A more operational formula for $g$ is provided in Proposition~\ref{prop:g2b}. This may be used to obtain the explicit expression of $g$ in the spirit of Remark~\ref{rmk:sm-explicit}, which treats the prototype case $f(t)=t/(1-t)$, $t\in[0,1)$. The threshold of complete fracture $\sfrac$ (which is possibly infinite) is explicitly determined in terms of $f$, see \eqref{def:sfrac} and Proposition~\ref{prop:sfrac}.

With the previous definitions, we are now ready to introduce the approximating functionals and to state the main $\Gamma$-convergence result from \cite{ConFocIur16}. We let $\Feps: L^1(0,L)\times L^1(0,L)\to[0,+\infty]$ be defined by 
\begin{equation} \label{def:Feps}
\Feps(u,v) \defeq
\begin{cases}
\displaystyle\int_0^L \Bigl( f_\e^2(v) |u'|^2 + \frac{(1-v)^2}{4\e} + \e|v'|^2 \Bigr) \dd x
&
\begin{array}{l}
\text{if } u,v\in H^1(0,L), \\ 
0\leq v\leq 1 \text{ } \Lu\text{-a.e. in }(0,L),
\end{array}\\
+\infty
&\text{ otherwise.}
\end{cases}
\end{equation}

The following $\Gamma$-convergence result is proved in \cite{ConFocIur16}, see in particular Theorem~3.1, Remark~3.2 and Theorem~3.3.1. Although we state the theorem in dimension one, the result in \cite{ConFocIur16} is actually proved in any dimension and under more general assumptions on $f$.

\begin{theorem}[Conti-Focardi-Iurlano] \label{thm:cfi}
	The functionals $\Feps$ defined in \eqref{def:Feps} $\Gamma$-converge as $\eps\to0^+$ in $L^1(0,L)\times L^1(0,L)$ to the functional
	\begin{equation} \label{def:F2}
	\F(u,v)\defeq
	\begin{cases}
	\Phi(u) & \text{if } u\in \BV(0,L),\, v=1 \text{ } \Lu\text{-a.e. in }(0,L), \\
	+\infty &\text{ otherwise,}
	\end{cases}
	\end{equation}
	where $\Phi$ is the cohesive fracture energy given by \eqref{def:F} (with elastic energy density $\phi$ as in \eqref{def:phi} and surface energy density $g$ as in \eqref{def:g}). Moreover, if $(u_\e,v_\e)$ satisfies the uniform bound
	\begin{equation} \label{eq:compactness}
	\sup_\e \Bigl( \Feps(u_\e,v_\e) + \|u_\e\|_{L^1(0,L)} \Bigr) < +\infty,
	\end{equation}
	then there exists a subsequence $(u_{\e_k},v_{\e_k})_{k\in\N}$ and a function $u\in\BV(0,L)$ such that $u_{\e_k}\to u$ almost everywhere in $(0,L)$ and $v_{\e_k}\to1$ in $L^1(0,L)$.
\end{theorem}

\begin{remark} \label{rmk:assf}
We collect here for later use a few properties that follow immediately from the assumptions \ref{ass:f1}-\ref{ass:f5}. Firstly, we observe that the function $f$ is strictly increasing with $f'(s)>0$ for $s\in(0,1)$, by \ref{ass:f3}. By de l'Hôpital's theorem, using \ref{ass:f2} and \ref{ass:f4}, one also has that
\begin{equation} \label{eq:f6}
\lim_{s\to1^-}(1-s)^2f'(s)=\stress\,.
\end{equation}
Finally, by the monotonicity property \ref{ass:f4} one has that the function $\frac{f'(s)}{f(s)}$ has a limit as $s\to0^+$, and since $f(0)=0$ it is easily seen that it must be
\begin{equation} \label{eq:f7}
\lim_{s\to0^+}\frac{f'(s)}{f(s)}=+\infty.
\end{equation}
\end{remark}


\subsection{Main results} \label{subsec:results}
We consider a family of \textit{critical points} $(u_\e,v_\e)$ of the approximating functionals \eqref{def:Feps}, \textit{i.e.}\ $u_\e,v_\e\in H^1(0,L)$ satisfy the Euler-Lagrange equations (in the weak sense)
\begin{subnumcases}{}\smallskip
-\e v_\e'' + f_\e(v_\e)f_\e'(v_\e)(u_\e')^2 + \ts\frac{v_\e-1}{4\e}=0 & \text{in }(0,L), \label{def:cp1}\\ \smallskip
\bigl(f_\e^2(v_\e)u_\e'\bigr)'=0 & \text{in }(0,L), \label{def:cp2}\\ \smallskip
u_\e(0)=0, u_\e(L)=a_\e, & \label{def:cp3} \\ \smallskip
v_\e(0)=v_\e(L)=1, & \label{def:cp4}
\end{subnumcases}
where for the Dirichlet boundary condition for $u_\e$ we also require that
\begin{equation} \label{def:bc}
a_\e \to a>0 \qquad\text{as }\e\to0.
\end{equation}
In our first main result we show that any such a family of critical points with equibounded energies is precompact and that any limit point is necessarily a critical point of the cohesive energy \eqref{def:F}. Moreover the $\Gamma$-convergence acts as a selection criterion, since the limit critical point has at most one crack point, located at the midpoint of the bar.

\begin{theorem} \label{thm:main1}
Assume that $(u_\e,v_\e)$ are critical points of the functionals $\Feps$ satisfying \eqref{def:bc} and
\begin{equation} \label{eq:bound-energy}
\sup_{\e}\Feps(u_\eps,v_\eps)<+\infty.
\end{equation}
Then there exists a subsequence $\e_k\to0$ such that $(u_{\e_k},v_{\e_k})\to(u,1)$ in $(L^1(0,L))^2$, for some $u\in\SBV(0,L)$ such that $|Du|(0,L)=a$. Moreover, letting $\ds m_0 \defeq \lim_{k\to\infty}\min_{[0,L]}v_{\e_k}$, we have that:
\begin{enumerate}
\item\label{thm:main-case1} If $m_0=1$, then $u(x)=\frac{a}{L}x$.
\item\label{thm:main-case2} If $m_0\in(0,1)$, then $u(x)=c_0x + (a-c_0L)\chi_{(\frac{L}{2},L)}(x)$ with $c_0=\frac12f(m_0)(1-m_0)\in(0,\frac{\stress}{2})$, $a-c_0L=[u](\midp)\in(0,\sfrac)$, and $g'([u](\midp))=2c_0$.
\item\label{thm:main-case3} If $m_0=0$, then $\sfrac\in\R$ and $u(x)=a\chi_{(\frac{L}{2},L)}(x)$ with $a=[u](\midp)=\sfrac$.
\end{enumerate}
Finally if $m_0<1$, or if $m_0=1$ and $\frac{a}{L}\leq\frac{\stress}{2}$, the convergence of the energies holds:
\begin{equation} \label{eq:conv-energy}
\Feps(u_\e,v_\e) \to \F (u,1) \qquad\text{as }\e\to0.
\end{equation}
If instead $m_0=1$ and $\frac{a}{L}>\frac{\stress}{2}$, the convergence of the energies \eqref{eq:conv-energy} does not hold.
\end{theorem}

Our second result is an existence counterpart of Theorem~\ref{thm:main1}: we show that, for any critical point $\bar{u}$ of the cohesive energy \eqref{def:F} that might appear as limit of critical points of the functionals $\Feps$ (that is, those that can be obtained as limits in Theorem~\ref{thm:main1}), we can actually \emph{construct} a family $(u_\e,v_\e)$ of critical points of $\Feps$, with equibounded energy, approximating $\bar{u}$ as $\e\to0$.

\begin{theorem} \label{thm:main2}
Let $\bar{u}\in\SBV(0,L)$ be either:
\begin{enumerate}
\item $\bar{u}(x)=\frac{a}{L}x$ for some $a>0$, or
\item $\bar{u}(x)=c_0x + (a-c_0L)\chi_{(\frac{L}{2},L)}(x)$ with $c_0\in(0,\frac{\stress}{2})$ and $g'(a-c_0L)=2c_0$, or
\item $\bar{u}(x)=a\chi_{(\frac{L}{2},L)}(x)$ with $a=\sfrac$ (if $\sfrac$ is finite).
\end{enumerate}
Then for every $\e>0$ sufficiently small there exists a critical point $(u_\e,v_\e)$ solution to \eqref{def:cp1}--\eqref{def:cp4} with $a_\e\to a$ and uniformly bounded energy, such that $u_\e\to\bar{u}$ in $L^1(0,L)$ as $\e\to0$.
\end{theorem}

Let us discuss briefly our strategy of proof. Compared with the brittle case \cite{FraLeSer09,Le10,BabMilRoda,BabMilRodb}, the main difficulty in our problem is that the behaviours of $u_\e$ and $v_\e$ are deeply related, meaning that their transitions happen in the same regions with the same scale. With this idea in mind, we start the proof of Theorem~\ref{thm:main1} and the study of system \eqref{def:cp1}-\eqref{def:cp4} by computing $u'_\e$ from \eqref{def:cp2} and inserting it into \eqref{def:cp1}, so obtaining a second order ODE for $v_\e$ (equation \eqref{eq:cp1}). Analysis of the ODE \eqref{eq:cp1} (performed in Section~\ref{sect:ODE}) shows that $v_\e$ is a symmetric well with minimum $m_\e=v_\e(\frac{L}{2})\in(0,1]$ and that the interval $\{f_\e(v_\e)=\eps^{1/2}f(v_\e)\}$ shrinks to $0$ as $\e\to0$ (Section~\ref{sect:proof-prelim}). Also, by definition of such interval, we find $v_\e\to1$ uniformly on compact sets not containing $\frac{L}{2}$. Now, we argue differently depending on the value of $m_0:=\lim_{\e\to0}m_\e$, that is, on the fact we are in the elastic ($m_0=1$), pre-fractured ($m_0\in(0,1)$) or fractured ($m_0=0$) regime. The richest regime is the pre-fractured one, addressed in Section~\ref{subsec:cohesive}.
Defining $c_\e:=f^2_\e(v_\e)u'_\e$ and $c_0:=\lim_{\e\to0}c_\e$, we get $u'_\e\to c_0$ uniformly on compact sets not containing $\frac{L}{2}$. This in particular implies $u\in\mathrm{SBV}(0,L)$, $J_u\subset\{\frac{L}{2}\}$, $u'=c_0$ a.e. We consider the blow-up of $u_\e$ and $v_\e$ around the minimum point $\frac{L}{2}$, setting $\tilde u_\e(t)=u_\e(\frac{L}{2}+\e t)$ and $\tilde v_\e(t)=v_\e(\frac{L}{2}+\e t)$. Passing to the limit, we get $(\tilde u_\e,\tilde v_\e)\to(\alpha_{\bar s},\beta _{\bar s})$ optimal profile for $g(\bar s)$, for some $\bar s\geq0$. The most delicate point is to prove that $[u](\frac{L}{2})=\bar s$. This is obtained by studying the bijective continuous map $s\in(0,\sfrac)\mapsto m_s\in(0,1)$, where $m_s$ is the minimum of $\beta_s$ with $(\alpha_s,\beta_s)$ optimal profile for $g(s)$. Its inverse can be written in integral form, see \eqref{eq:ms-3} in Proposition~\ref{prop:g2b}. Proofs in the elastic and fractured regimes, respectively performed in Sections~\ref{subsec:elastic} and \ref{subsec:fracture}, are not based on blow-up arguments, but rather on energy estimates. However, they also require fine ad hoc computations involving formula \eqref{eq:ms-3}: indeed, in the elastic case we need to show that $[u](\frac{L}{2})=0$ and in the fractured case that $[u](\frac{L}{2})= \sfrac$.

The proof of the second main Theorem~\ref{thm:main2} is addressed in Section~\ref{sect:proof2}. The elastic case is trivial, since we can define $u_\e(x)= \frac{a}{L}x$, $v_\e(x)= 1$ in $(0,L)$. The pre-fractured and fractured cases are again ODE-based. Take $c_\e\in (0,\frac{\sigma_c}{2})$ and set $c_0:=\lim_{\e\to0}c_\e$. Then, we show that for all $\e>0$ we can choose $m_\e>0$ such that the unique solution $v_\e$ to the second order ODE \eqref{eq:cp1} with initial conditions $v_\e(\frac{L}{2})=m_\e$, $v'_\e(\frac{L}{2})=0$, satisfies in addition $v_\e(0)=v_\e(L)=1$. This strongly uses the analysis of the ODE in Section~\ref{sect:ODE} and a continuous dependence argument on the initial value. Finally, $u_\e$ is easily computed from \eqref{def:cp2}.

\bigskip\noindent
\textbf{Structure of the paper.}
Section~\ref{sect:ODE} contains a detailed analysis of the ODE that is solved by a critical profile $v_\e$. In Section~\ref{sect:g} we discuss the properties of the cohesive energy density $g$ appearing in the $\Gamma$-limit of the functionals $\Feps$. In Sections~\ref{sect:proof-prelim} and \ref{sect:proof1} we give the proof of Theorem~\ref{thm:main1}, whereas Section~\ref{sect:proof2} contains the proof of Theorem~\ref{thm:main2}.


\section{Analysis of the ODE satisfied by critical points} \label{sect:ODE}

In this section we discuss the properties of the solution to the Cauchy problem
\begin{subnumcases}{}\smallskip
	y''=\frac{1-y}{4}\biggl[\frac{(2\alpha)^2f'(y)}{(1-y)f^3(y)}-1\biggr] \label{eq:ODE1}\\ \smallskip
	y(0)=m \label{eq:ODE2}\\ \smallskip
	y'(0)=0 \label{eq:ODE3}
\end{subnumcases}
for fixed parameters $\alpha\in(0,\frac{\stress}{2})$ and $m\in(0,1)$. As it will be clear in the following of the paper, this equation is satisfied by an optimal profile $\beta_s$ for the minimum problem defining $g$ (see \eqref{def:g} and Section \ref{sect:g}) and also corresponds to the blow-up of the critical points equation \eqref{def:cp1} around a stationary point of $v_\e$.

In the whole of this section we always assume that the function $f$ appearing in \eqref{eq:ODE1} satisfies the conditions \ref{ass:f1}--\ref{ass:f5}; in particular, the analysis of the Cauchy problem does not make use of \ref{ass:f6}.

It is convenient to introduce the functions
\begin{equation} \label{def:fbar}
	\fbar(s)\defeq \frac{f'(s)}{(1-s)f^3(s)}
\end{equation}
and
\begin{equation} \label{def:h}
	h(s)\defeq\frac{1-s}{4}\biggl[\frac{(2\alpha)^2f'(s)}{(1-s)f^3(s)}-1\biggr] = \frac{1-s}{4}\Bigl[(2\alpha)^2\fbar(s)-1\Bigr]\qquad \text{for }s\in(0,1),
\end{equation}
extended by continuity by setting $h(1)=0$. With this notation, the equation \eqref{eq:ODE1} takes the form $y''=h(y)$. The constant function $\bar{y}\equiv1$ is always a stationary solution of the equation \eqref{eq:ODE1}; in the next proposition we show in particular that the function $h$ has a unique zero $z_\alpha\in(0,1)$, which corresponds to a second stationary solution $\bar{y}\equiv z_\alpha$.

\begin{proposition}[Properties of $\fbar$] \label{prop:salpha}
	The function $\fbar$ defined in \eqref{def:fbar} is of class $C^1(0,1)$ and strictly decreasing, with $\fbar'(s)<0$ for all $s\in(0,1)$. Moreover
	\begin{equation} \label{eq:limits-fbar}
		\lim_{s\to0^+}\fbar(s)=+\infty, \quad
		\lim_{s\to1^-}\fbar(s)=\frac{1}{\stress^2}, \quad
		\lim_{s\to0^+}\frac{|\fbar'(s)|}{\fbar(s)}=+\infty, \quad
		\lim_{s\to1^-}\frac{|\fbar'(s)|}{(1-s)^3}=+\infty.
	\end{equation}
	In particular, it follows that for every $\alpha\in(0,\frac{\stress}{2})$ there exists a unique $z_\alpha\in(0,1)$ such that
	\begin{equation} \label{eq:salpha}
		\fbar(z_\alpha)=\frac{1}{(2\alpha)^2}\,,
	\end{equation}
	and moreover $(1-z_\alpha)f(z_\alpha)>2\alpha$.
\end{proposition}

\begin{proof}
	The regularity of $\fbar$ follows by \ref{ass:f1}. By computing explicitly the derivative appearing in assumption \ref{ass:f4}, one easily obtains the inequality
	$$
	(1-s)f(s)f''(s)-f(s)f'(s)-(1-s)(f'(s))^2 < 0.
	$$
	Then using this inequality we get for all $s\in(0,1)$
	\begin{align*}
		\fbar'(s)
		&= \frac{(1-s)f(s)f''(s)+f(s)f'(s)-3(1-s)(f'(s))^2}{(1-s)^2f^4(s)} \\
		& < \frac{2f'(s)\bigl[f(s)-(1-s)f'(s)\bigr]}{(1-s)^2f^4(s)}
		= -\frac{2f'(s)[(1-s)f(s)]'}{(1-s)^2f^4(s)} < 0,
	\end{align*}
	where the last inequality follows by \ref{ass:f3} (see also Remark~\ref{rmk:assf}); this proves the strict monotonicity of $\fbar$. The first limit in \eqref{eq:limits-fbar} follows from \eqref{eq:f7}; the second limit in \eqref{eq:limits-fbar} is a consequence of \ref{ass:f2} and \eqref{eq:f6}; for the third limit, one has, arguing as before and using \eqref{eq:f7},
	\begin{align*}
		\frac{|\fbar'(s)|}{\fbar(s)}
		& >\frac{2f'(s)\bigl[(1-s)f'(s)-f(s)\bigr]}{(1-s)^2f^4(s)}\cdot\frac{(1-s)f^3(s)}{f'(s)}
		= \frac{2f'(s)}{f(s)}-\frac{2}{1-s} \to +\infty \quad\text{as } s\to0^+.
	\end{align*}
	Similarly for the fourth limit in \eqref{eq:limits-fbar} we use assumption \ref{ass:f5}, together with \ref{ass:f2} and \eqref{eq:f6}:
	\begin{align*}
		\frac{|\fbar'(s)|}{(1-s)^3}
		& >\frac{2f'(s)\bigl[(1-s)f(s)\bigr]'}{(1-s)^5f^4(s)} 
		= \frac{\bigl[(1-s)f(s)\bigr]'}{(1-s)^3}\cdot\frac{2(1-s)^2f'(s)}{(1-s)^4f^4(s)} \to +\infty \quad\text{as } s\to1^-.
	\end{align*}
	
	Existence and uniqueness of $z_\alpha$ are immediate consequences of the strict monotonicity of $\fbar$ and of \eqref{eq:limits-fbar}. To show that $(1-z_\alpha)f(z_\alpha)>2\alpha$, we observe that \eqref{eq:salpha} and the monotonicity assumption \ref{ass:f4} give
	\begin{equation*}
		(1-z_\alpha)^2f^2(z_\alpha)
		= \frac{(1-z_\alpha)f'(z_\alpha)}{f(z_\alpha)}\cdot(2\alpha)^2
		> \lim_{s\to1^-} \frac{(1-s)f'(s)}{f(s)}\cdot(2\alpha)^2 = (2\alpha)^2,
	\end{equation*}
	where we used \ref{ass:f2} and \eqref{eq:f6} in the last inequality.
\end{proof}

By Cauchy-Lipschitz theorem, for every value of the initial datum $m\in(0,1)$ the Cauchy problem \eqref{eq:ODE1}--\eqref{eq:ODE3} has a unique solution $\bar{y}$ of class $C^2$, which can be continued as long as $y(t)\in(0,1)$ and is defined in a maximal interval $(-T,T)$ for some $T\in(0,+\infty]$. By uniqueness the solution is symmetric with respect to the origin, that is $\bar{y}(t)=\bar{y}(-t)$, and therefore we study the equation only on the positive real axis.
In the following theorem we characterize the behaviour of the solution to \eqref{eq:ODE1}--\eqref{eq:ODE3} in terms on the relation between the parameters $m$ and $\alpha$. We focus on the case $m\in(0,z_\alpha)$; for the case $m\in(z_\alpha,1)$, see Remark~\ref{rmk:mlarger}.

\begin{theorem} \label{thm:ODE}
	Let $\alpha\in(0,\frac{\stress}{2})$ and $m\in(0,z_\alpha)$ be given, let $f$ satisfy \ref{ass:f1}--\ref{ass:f5}, and let $\bar{y}$ be the solution to the Cauchy problem \eqref{eq:ODE1}--\eqref{eq:ODE3}. Then there exists $t_0>0$ such that $\bar{y}(t_0)=z_\alpha$ and $m<\bar{y}(t)<z_\alpha$ for $t\in(0,t_0)$. Moreover:
	\begin{enumerate}
		\item\label{item:ODE1} if $(1-m)f(m)<2\alpha$, then there exists $t_1\in(t_0,+\infty)$ such that $\bar{y}(t_1)=1$, and $\bar{y}$ is strictly increasing in $(0,t_1)$;
		\item\label{item:ODE2} if $(1-m)f(m)=2\alpha$, then $\bar{y}(t)<1$ for all $t\in(0,+\infty)$, $\bar{y}$ is stricly increasing, and $\bar{y}(t)\to1$ as $t\to+\infty$;
		\item\label{item:ODE3} if $(1-m)f(m)>2\alpha$, then $\bar{y}$ oscillates periodically between its minimum $m$ and a maximum $M\in(z_\alpha,1)$: more precisely, there exists $t_2\in(t_0,+\infty)$ such that $\bar{y}(t_2)=M=\max\bar{y}$ and $\bar{y}$ is strictly increasing in $(0,t_2)$. For $t\in(t_2,2t_2)$ one has $\bar{y}(t)=\bar{y}(2t_2-t)$, and $\bar{y}$ is periodic with period $2t_2$.
	\end{enumerate}
	The values of $t_0$, $t_1$, $t_2$, $M$ in the previous statements depend on $m$ and $\alpha$.
\end{theorem}

\begin{figure}
	\begin{center}
		\includegraphics[width=7.2cm]{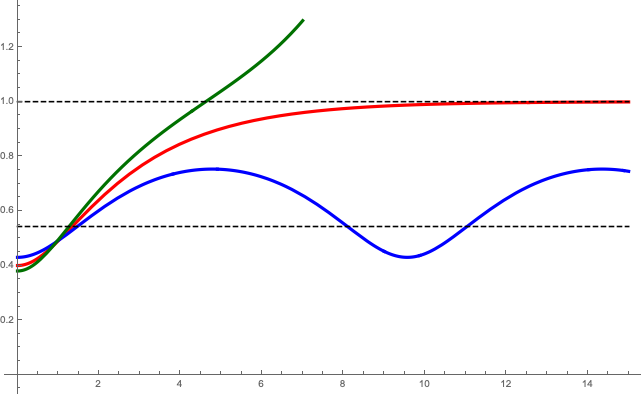} \qquad
		\includegraphics[width=7.2cm]{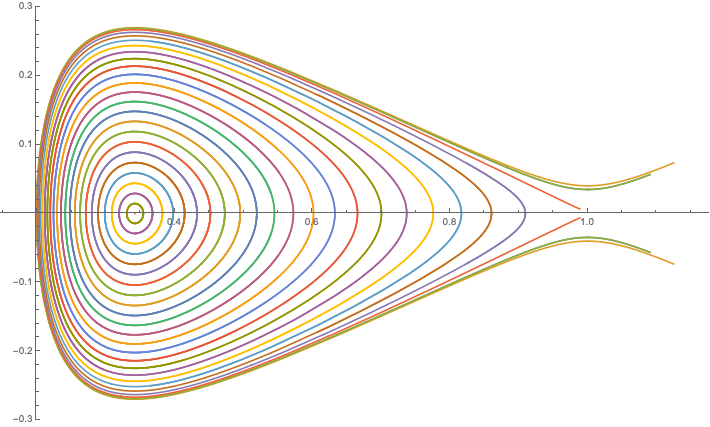}
	\end{center}
	\caption{Left: numerical plots of the solution to the Cauchy problem \eqref{eq:ODE1}--\eqref{eq:ODE3} for different initial values $m$, for the prototype function $f(s)=\frac{s}{1-s}$ (see \eqref{eq:examplefq} with $\stress=q=1$) and $\alpha=0.2$. The two dashed lines correspond to the stationary solutions $\bar{y}\equiv1$ and $\bar{y}\equiv z_\alpha$. The green, red and blue curves correspond to cases \ref{item:ODE1}, \ref{item:ODE2} and \ref{item:ODE3} in Theorem~\ref{thm:ODE} respectively. Right: phase diagram of the solutions to \eqref{eq:ODE1}--\eqref{eq:ODE3} in the plane $(y,y')$.}
	\label{fig:ODE}
\end{figure}

\begin{proof}
	Notice that, since $(1-z_\alpha)f(z_\alpha)>2\alpha$ by Proposition~\ref{prop:salpha} and the map $s\mapsto(1-s)f(s)$ is strictly increasing by assumption \ref{ass:f3},  for $m\in(0,z_\alpha)$ all the three cases may occur. We also observe that, in view of Proposition~\ref{prop:salpha}, one has that $h(s)>0$ for $s\in(0,z_\alpha)$ and $h(s)<0$ for $s\in(z_\alpha,1)$.
	
	The solution $\bar{y}$ satisfies $\bar{y}'(0)=0$ and $\bar{y}''(0)=h(m)>0$, therefore $\bar{y}'>0$ and $\bar{y}$ is strictly increasing and convex in a right neighbourhood of the origin. Since $\bar{y}$ remains convex as long as $\bar{y}(t)<z_\alpha$, there exists $t_0>0$ such that $\bar{y}(t_0)=z_\alpha$, and $\bar{y}$ is convex and strictly increasing in $(0,t_0)$.
	
	At the point $t_0$ we have $\bar{y}'(t_0)>0$ (otherwise $\bar{y}\equiv z_\alpha$) and $\bar{y}''(t_0)=0$, and therefore $\bar{y}(t)>z_\alpha$ and $\bar{y}''(t)=h(\bar{y}(t))<0$ for $t$ in a right neighbourhood of $t_0$, so that $\bar{y}$ becomes concave after $t_0$. We let
	$$
	t_1\defeq\sup\bigl\{t>0\,:\, \bar{y}(s)<1 \text{ for all }s\in(0,t)\bigr\}
	$$
	and we observe that $t_1>t_0$. We can then identify three possible types of solutions.
	
	\smallskip\noindent\textit{Case I.}
	If $t_1\in\R$, then $\bar{y}$ reaches the value $1$ in finite time: indeed it must be $\bar{y}(t_1)=1$, and $\bar{y}'(t_1)>0$ (it cannot be $\bar{y}'(t_1)=0$, or else the solution would coincide with the stationary solution constantly equal to $1$).
	
	\smallskip
	If $t_1=+\infty$, then the solution is defined in the whole positive real line and $\bar{y}(t)<1$ for all $t>0$. We let
	$$
	t_2\defeq\sup\bigl\{t>0\,:\, \bar{y}'(s)>0 \text{ for all }s\in(0,t)\bigr\}
	$$
	and we observe that $t_2>t_0$, as $\bar{y}'(t)>0$ for all $t\in(0,t_0)$.
	
	\smallskip\noindent\textit{Case II.}
	If $t_2=+\infty$, then $\bar{y}'(t)>0$ for all $t>0$. Then the solution is strictly increasing in $(0,+\infty)$, is convex in $(0,t_0)$, concave in $(t_0,\infty)$, and $\bar{y}(t)\to1$ as $t\to+\infty$.
	
	\smallskip\noindent\textit{Case III.}
	If $t_2\in\R$, then $\bar{y}'(t_2)=0$, $M\defeq \bar{y}(t_2)\in(z_\alpha,1)$ and $\bar{y}''(t_2)=h(M)<0$. Then $t_2$ is a local maximum of $\bar{y}$. The function $z(t)\defeq \bar{y}(2t_2-t)$, for $t\in(t_2,2t_2)$, solves the equation \eqref{eq:ODE1} in $(t_2,2t_2)$, with $z(t_2)=M=\bar{y}(t_2)$ and $z'(t_2)=0=\bar{y}'(t_2)$, and therefore by uniqueness of the solution of the Cauchy problem it must be $\bar{y}(t)=z(t)=\bar{y}(2t_2-t)$, for $t\in(t_2,2t_2)$. At the point $2t_2$ we have $\bar{y}(2t_2)=m$, $\bar{y}'(2t_2)=0$, and again by uniqueness we can conclude that $\bar{y}(t)=\bar{y}(t-2t_2)$ for all $t>2t_2$, that is, $\bar{y}$ is periodic with period $2t_2$.
	
	\smallskip
	These are the only three possible behaviours of solutions to \eqref{eq:ODE1}--\eqref{eq:ODE3}, for $m\in(0,z_\alpha)$. See Figure~\ref{fig:ODE} for numerical plots of the three types of solutions. To conclude the proof, we only need to determine the form of the solution in terms of the value of the initial datum. We let
	$$
	\tilde{t}\defeq
	\begin{cases}
	t_1 & \text{in Case I,}\\
	+\infty & \text{in Case II, }\\
	t_2 & \text{in Case III.}
	\end{cases}
	$$
	Since $\bar{y}'>0$ in $(0,\tilde{t})$ in all cases, we can multiply the equation \eqref{eq:ODE1} by $\bar{y}'$ and integrate in $(0,t)$, for $t<\tilde{t}$: we obtain by a change of variables
	\begin{equation*}
		\int_0^t \bar{y}''(\tau)\bar{y}'(\tau)\dd\tau = \int_0^t h(\bar{y}(\tau))\bar{y}'(\tau)\dd\tau = \int_{m}^{\bar{y}(t)}h(s)\dd s,
	\end{equation*}
	and, recalling that $\bar{y}'(0)=0$,
	\begin{align*}
		\frac12\bigl(\bar{y}'(t)\bigr)^2
		& = \int_{m}^{\bar{y}(t)}h(s)\dd s = \int_{m}^{\bar{y}(t)} \frac{1-s}{4}\biggl[\frac{(2\alpha)^2f'(s)}{(1-s)f^3(s)}-1\biggr] \dd s \\
		& = \frac14 \int_m^{\bar{y}(t)} \biggl( (2\alpha)^2\frac{f'(s)}{f^3(s)} - 1 + s \biggr)\dd s \\
		& = \frac14\biggl( -\frac{(2\alpha)^2}{2f^2(\bar{y}(t))} + \frac{(2\alpha)^2}{2f^2(m)} -\bar{y}(t) + \frac12\bar{y}^2(t) + m -\frac12m^2 \biggr).
	\end{align*}
	Therefore for all $t\in(0,\tilde{t})$
	\begin{equation} \label{eq:ODEfirstintegral}
		(\bar{y}')^2 = \frac14\biggl[ (2\alpha)^2\biggl( \frac{1}{f^2(m)} - \frac{1}{f^2(\bar{y})} \biggr) - (1-m)^2 + (1-\bar{y})^2 \biggr].
	\end{equation}
	
	If we are in Case I, then as $t\to\tilde{t}=t_1$ we have $\bar{y}(t)\to1$ and $\bar{y}'(t)\to\bar{y}'(t_1)>0$. Therefore, by letting $t\to t_1$ in \eqref{eq:ODEfirstintegral} we obtain
	\begin{equation*} 
		0<(\bar{y}'(t_1))^2 = \frac14\biggl[ \frac{(2\alpha)^2}{f^2(m)} - (1-m)^2 \biggr] = \frac{(1-m)^2}{4}\biggl[ \biggl(\frac{2\alpha}{(1-m)f(m)}\biggr)^2 - 1 \biggr] 
	\end{equation*}
	and hence $(1-m)f(m)<2\alpha$.
	
	If we are in Case II, by letting $t\to+\infty$ in \eqref{eq:ODEfirstintegral} we obtain, since $\bar{y}(t)\to1$ and $\bar{y}'(t)\to0$,
	\begin{equation*} 
		0=\lim_{t\to+\infty}(\bar{y}'(t))^2 = \frac{(1-m)^2}{4}\biggl[ \biggl(\frac{2\alpha}{(1-m)f(m)}\biggr)^2 - 1 \biggr] 
	\end{equation*}
	and hence $(1-m)f(m)=2\alpha$.
	
	Finally, if we are in Case III, by letting $t\to\tilde{t}=t_2$ in \eqref{eq:ODEfirstintegral} we have, since $\bar{y}'(t_2)=0$ and $\bar{y}(t_2)=M\in(z_\alpha,1)$,
	\begin{align*} 
		0 = \frac14\biggl[ (2\alpha)^2\biggl( \frac{1}{f^2(m)} - \frac{1}{f^2(M)} \biggr) - (1-m)^2 + (1-M)^2 \biggr],
	\end{align*}
	from which we get
	\begin{align*} 
		(2\alpha)^2
		&= \biggl(\frac{(1-m)^2-(1-M)^2}{f^2(M)-f^2(m)}\biggr) f^2(m)f^2(M) \\
		&=\biggl(\frac{f^2(M)(1-m)^2-f^2(M)(1-M)^2}{(1-m)^2 \bigl(f^2(M)-f^2(m)\bigr)}\biggr) \bigl(f(m)(1-m)\bigr)^2 \\
		&<\biggl(\frac{f^2(M)(1-m)^2-f^2(m)(1-m)^2}{(1-m)^2 \bigl(f^2(M)-f^2(m)\bigr)}\biggr) \bigl(f(m)(1-m)\bigr)^2
		=  \bigl(f(m)(1-m)\bigr)^2,
	\end{align*}
	where we used the fact that $f^2(M)(1-M)^2>f^2(m)(1-m)^2$ since $M>m$ and the function $s\mapsto f(s)(1-s)$ is strictly increasing by \ref{ass:f3}. Therefore $(1-m)f(m)>2\alpha$.
	
	Since these are the only possible cases, the characterization of $\bar{y}$ in the statement holds.
\end{proof}

\begin{remark} \label{rmk:M}
	Let $m_\alpha\in(0,z_\alpha)$ be the unique value such that $(1-m_\alpha)f(m_\alpha)=2\alpha$. The three cases \ref{item:ODE1}, \ref{item:ODE2} and \ref{item:ODE3} in Theorem~\ref{thm:ODE} correspond to $m\in(0,m_\alpha)$, $m=m_\alpha$, and $m\in(m_\alpha, z_\alpha)$ respectively. If $m\in(m_\alpha,z_\alpha)$ (case \ref{item:ODE3}) the solution has a maximum value $M\in(z_\alpha,1)$ which is uniquely determined by $\alpha$ and $m$ by the equation
	\begin{equation} \label{eq:M}
		(1-M)^2 - \frac{(2\alpha)^2}{f^2(M)} = (1-m)^2 - \frac{(2\alpha)}{f^2(m)},
	\end{equation}
	which is obtained by evaluating \eqref{eq:ODEfirstintegral} at the maximum point $t_2$. We observe that the function
	$$
	x\mapsto(1-x)^2-\frac{(2\alpha)^2}{f^2(x)}
	$$
	vanishes at $x=m_\alpha$ and at $x=1$, has a maximum point at $x=z_\alpha$, is increasing in $(m_\alpha,z_\alpha)$ and decreasing in $(z_\alpha,1)$, as can be easily checked by noting that its derivative is given by $2(1-x)((2\alpha)^2\fbar(x)-1)$.
	
	It follows that $M$ takes all the values in the interval $(z_\alpha,1)$ when $m\in(m_\alpha,z_\alpha)$, with $M\to1$ as $m\to m_\alpha$, and $M\to z_\alpha$ as $m\to z_\alpha$.
\end{remark}

\begin{remark} \label{rmk:mlarger}
	In the case $m\in(z_\alpha,1)$, the solution $\bar{y}$ is a translation of one of the periodic orbits obtained in Theorem~\ref{thm:ODE}\ref{item:ODE3}. Indeed, by Remark~\ref{rmk:M} every value of $m\in(z_\alpha,1)$ corresponds to the maximum value of one of the solutions with initial datum in $(m_\alpha,z_\alpha)$, so that by uniqueness $\bar{y}$ is a translation of that solution.
\end{remark}

In the next proposition we obtain uniform estimates on the point $t_\eta$ at which the solution to the Cauchy problem \eqref{eq:ODE1}--\eqref{eq:ODE3} reaches a value $\eta$ arbitrarily close to 1 (or the maximum $M$ for the periodic solutions). In the following, we will denote by $\omega(\cdot)$ a generic modulus of continuity (that is, $\omega:[0,1]\to\R$ is a bounded, monotone increasing function such that $\omega(t)\to0$ as $t\to0^+$), which is independent of $\alpha\in(0,\frac{\stress}{2})$ and $m\in(0,z_\alpha)$, but depends ultimately only on $f$. The function $\omega$ might change from line to line.

\begin{proposition} \label{prop:ODE}
	Under the assumptions of Theorem~\ref{thm:ODE}, let $y$ be the solution to \eqref{eq:ODE1}--\eqref{eq:ODE3}. Let $\eta\in(z_\alpha,1)$ be fixed; in the case $(1-m)f(m)>2\alpha$, assume further that $\eta\in(z_\alpha,M)$, where $M\in(z_\alpha,1)$ is the maximum of $y$. Denote by $t_\eta>0$ the first point such that $y(t_\eta)=\eta$.
	
	Then there exists a constant $C_\alpha>0$, depending only on $\alpha\in(0,\frac{\stress}{2})$, and a modulus of continuity $\omega$ independent of $\alpha$ and $m$, such that the following estimates hold:
	\begin{equation}\label{eq:est-teta-1} 
		t_\eta \leq \frac{C_{\alpha}}{\sqrt{1-\eta}} \qquad\qquad\text{with}\quad \sup_{\alpha\in(0,\frac{\stress}{2}-\delta)}C_\alpha<+\infty \quad\text{for all }\delta>0,
	\end{equation}
	\begin{equation}\label{eq:est-teta-2}
		t_\eta \leq \frac{\omega(1-m)}{(1-\eta)^2}\,.
	\end{equation}
	Furthermore, we have the estimate from below
	\begin{equation}\label{eq:est-teta-3}
		t_\eta \geq \log\biggl(\frac{1-z_\alpha + C_{\alpha,m}}{1-\eta + C_{\alpha,m}}\biggr), \qquad\text{where}\quad C_{\alpha,m}\defeq \bigg|1-\biggl(\frac{2\alpha}{(1-m)f(m)}\biggr)^2\bigg|^\frac{1}{2}.
	\end{equation}
\end{proposition}

\begin{proof}
	We first obtain a uniform estimate of the point $t_0$ where there is a change of convexity of the solution $y$, as in the statement of Theorem~\ref{thm:ODE}: recall that for $t\in(0,t_0)$ the solution is strictly increasing with $m<y(t)<z_\alpha$.
	
	From the proof of Theorem~\ref{thm:ODE}, see in particular \eqref{eq:ODEfirstintegral}, we have that
	\begin{equation}\label{proof:est-1}
		(y')^2 = \frac{(1-y)^2}{4}\biggl[ 1 - \biggl(\frac{2\alpha}{(1-y)f(y)}\biggr)^2 \biggr] - \frac{(1-m)^2}{4}\biggl[ 1 - \biggl(\frac{2\alpha}{(1-m)f(m)}\biggr)^2 \biggr] \eqqcolon \Psi(y).
	\end{equation}
	By a straightforward computation and recalling the definition \eqref{def:fbar} of $\fbar$ and \eqref{eq:salpha} we find
	\begin{equation} \label{proof:est-2}
		\Psi'(s)=\frac12(2\alpha)^2(1-s)(\fbar(s)-\fbar(z_\alpha))
	\end{equation}
	and in particular $\Psi'(s)>0$ for $s\in(m,z_\alpha)$ since $\fbar$ is strictly decreasing (see Proposition~\ref{prop:salpha}).
	By repeatedly applying the mean value theorem we have for all $s\in(m,z_\alpha)$ (using also $\Psi(m)=0$)
	\begin{align} \label{proof:est-3}
		\Psi(s) &=  \Psi(m) + \Psi'(\xi)(s-m) & & \text{for some }\xi\in(m,s) \nonumber\\
		&= \frac12(2\alpha)^2(1-\xi)(\fbar(\xi)-\fbar(z_\alpha))(s-m) \nonumber\\
		&= \frac{1}{2\fbar(z_\alpha)}(1-\xi)(-\fbar'(\zeta))(z_\alpha-\xi)(s-m) & & \text{for some }\zeta\in(\xi,z_\alpha) \nonumber\\
		& \geq \frac12\Bigl(-\frac{\fbar'(\zeta)}{\fbar(\zeta)}\Bigr)(1-z_\alpha)(z_\alpha-s)(s-m),
	\end{align}
	where we used in particular \eqref{eq:salpha} in the third equality, and the monotonicity of the function $\fbar$ in the last inequality. Observe now that, being $\fbar'<0$ and continuous in $(0,1)$, and by \eqref{eq:limits-fbar}, the ratio $|\fbar'|/\fbar$ is uniformly bounded from below by a positive constant in $(0,\frac12)$, whereas if $\zeta\in(\frac12,1)$ one has, by the fourth limit in \eqref{eq:limits-fbar}, that $|\fbar'(\zeta)|/\fbar(\zeta)\geq |\fbar'(\zeta)|/\fbar(\frac12) \geq (1-\zeta)^3/\omega(1-m)$, for a uniform modulus of continuity $\omega$. Therefore we can write
	\begin{equation} \label{proof:est-4}
		\inf_{\zeta\in(m,z_\alpha)}\frac{|\fbar'(\zeta)|}{\fbar(\zeta)} \geq \frac{(1-z_\alpha)^3}{\omega(1-m)}\,.
	\end{equation}
	Hence combining \eqref{proof:est-1}, \eqref{proof:est-3}, and \eqref{proof:est-4} we find
	\begin{align}\label{stimat0}
		t_0
		& = \int_0^{t_0} \frac{y'(t)}{\sqrt{\Psi(y(t))}} \dd t
		= \int_{m}^{z_\alpha} \frac{\dd s}{\sqrt{\Psi(s)}}
		\leq \frac{\omega(1-m)}{(1-z_\alpha)^2}\int_m^{z_\alpha}\frac{\dd s}{\sqrt{z_\alpha-s}\sqrt{s-m}}
	\end{align}
	which eventually yields
	\begin{equation} \label{proof:est-t0}
		t_0 \leq \pi \frac{\omega(1-m)}{(1-z_\alpha)^2}\,.
	\end{equation}
	
	We next fix $\eta>z_\alpha$ as in the statement. By the properties of the solution, there exists a point $t_\eta>t_0$ such that $y(t_\eta)=\eta$; in the interval $(t_0,t_\eta)$ the solution satisfies $z_\alpha<y(t)<\eta$, $y'(t)>0$. As before we have from \eqref{proof:est-1}
	\begin{equation} \label{proof:est-5}
		t_\eta - t_0
		= \int_{t_0}^{t_\eta} \frac{y'(t)}{\sqrt{\Psi(y(t))}} \dd t
		= \int_{z_\alpha}^{\eta} \frac{\dd s}{\sqrt{\Psi(s)}}\,,
	\end{equation}
	and therefore we need to estimate from below $\Psi(s)$ for $s\in(z_\alpha,\eta)$. Notice that by \eqref{proof:est-2} and monotonicity of $\fbar$, $\Psi$ is decreasing in $(z_\alpha,1)$. For all $s\in(z_\alpha,\eta)$ we have by the mean value theorem, using also $\Psi(\eta)\geq0$, \eqref{eq:salpha}, and \eqref{proof:est-2},
	\begin{align} \label{proof:est-6}
		\Psi(s) &=  \Psi(\eta) + \Psi'(\xi)(s-\eta) \nonumber\\
		&\geq \frac12(2\alpha)^2(1-\xi)(\fbar(z_\alpha)-\fbar(\xi))(\eta-s) \nonumber\\
		& \geq \frac12(1-\eta)(\eta-s)\frac{\fbar(z_\alpha)-\fbar(\xi)}{\fbar(z_\alpha)}
	\end{align}
	for some $\xi\in(s,\eta)$.
	We next bound from below the quotient on the right-hand side of \eqref{proof:est-6}, and we consider first the case $\alpha$ uniformly bounded away from $\frac{\stress}{2}$ (which will prove \eqref{eq:est-teta-1}), and then the case $\alpha$ in a small neighbourhood of $\frac{\stress}{2}$ (which will prove \eqref{eq:est-teta-2}). It is important to recall from Proposition~\ref{prop:salpha} that $z_\alpha$ depends continuously and monotonically on $\alpha$, and that $z_\alpha\to1$ as $\alpha\to\frac{\stress}{2}$.
	
	Notice first that
	\begin{equation} \label{proof:est-7}
		\text{if }\quad \frac{\fbar(\xi)}{\fbar(z_\alpha)} \leq \frac12  \qquad\text{then}\qquad \frac{\fbar(z_\alpha)-\fbar(\xi)}{\fbar(z_\alpha)} = 1 - \frac{\fbar(\xi)}{\fbar(z_\alpha)} \geq \frac12 \geq \frac12(\xi-z_\alpha).
	\end{equation}
	Let $\delta>0$ be such that $\alpha<\frac{\stress}{2}-\delta$. Then $z_\alpha<1-c_\delta$ for some $c_\delta>0$ depending only on $\delta$. We have that
	\begin{equation} \label{proof:est-8a}
		\text{if }\quad \xi \geq 1-\frac{c_\delta}{2} \qquad\text{then}\qquad \frac{\fbar(z_\alpha)-\fbar(\xi)}{\fbar(z_\alpha)} \geq 1 - \frac{\fbar(1-\frac{c_\delta}{2})}{\fbar(1-c_\delta)} \eqqcolon C_\delta^1 \geq C_\delta^1(\xi-z_\alpha),
	\end{equation}
	with $C_\delta^1>0$ depending only on $\delta$, by the properties of $\fbar$.
	If instead $\frac{\fbar(\xi)}{\fbar(z_\alpha)} \geq \frac12$ and $\xi \leq 1-\frac{c_\delta}{2}$ we obtain
	\begin{align*}
		\frac{\fbar(z_\alpha)-\fbar(\xi)}{\fbar(z_\alpha)}
		& \geq \frac{\fbar(z_\alpha)-\fbar(\xi)}{\xi-z_\alpha}\cdot\frac{\xi-z_\alpha}{2\fbar(\xi)}
		= -\fbar'(\zeta)\cdot\frac{\xi-z_\alpha}{2\fbar(\xi)} \qquad\qquad\text{for some }\zeta\in(z_\alpha,\xi) \\
		& \geq \frac{|\fbar'(\zeta)|}{2\fbar(\zeta)} \cdot (\xi-z_\alpha)
		\geq \biggl(\inf_{(0,1-\frac{c_\delta}{2})}\frac{|\fbar'|}{2\fbar}\biggr) (\xi-z_\alpha) \eqqcolon C_\delta^2(\xi-z_\alpha)
	\end{align*}
	where again $C_\delta^2>0$ by the properties of $\fbar$ in Proposition~\ref{prop:salpha}. Hence
	\begin{equation} \label{proof:est-8b}
		\text{if }\quad \frac{\fbar(\xi)}{\fbar(z_\alpha)} \geq \frac12 \quad\text{and}\quad \xi \leq 1-\frac{c_\delta}{2}
		\qquad\qquad\text{then}\qquad \frac{\fbar(z_\alpha)-\fbar(\xi)}{\fbar(z_\alpha)} \geq C_\delta^2(\xi-z_\alpha).
	\end{equation}
	Therefore by inserting \eqref{proof:est-7}, \eqref{proof:est-8a}, and \eqref{proof:est-8b} into \eqref{proof:est-6} we find that for all $\alpha<\frac{\stress}{2}-\delta$
	\begin{equation*}
		\Psi(s) \geq C_\delta(1-\eta)(\eta-s)(\xi-z_\alpha) \geq C_\delta(1-\eta)(\eta-s)(s-z_\alpha) \qquad\text{for all }s\in(z_\alpha,\eta)
	\end{equation*}
	for a uniform constant $C_\delta>0$ depending only on $\delta$. In turn by \eqref{proof:est-5} we conclude that 
	\begin{equation*}
		t_\eta - t_0 \leq \frac{1}{\sqrt{C_\delta}\sqrt{1-\eta}}\int_{z_\alpha}^\eta \frac{\dd s}{\sqrt{\eta-s}\sqrt{s-z_\alpha}} = \frac{\pi}{\sqrt{C_\delta}\sqrt{1-\eta}}
	\end{equation*}
	which combined with \eqref{proof:est-t0} proves the estimate \eqref{eq:est-teta-1} in the statement.
	
	We next show \eqref{eq:est-teta-2}. If $\frac{\fbar(\xi)}{\fbar(z_\alpha)} \geq \frac12$ then by using the last condition in \eqref{eq:limits-fbar} and arguing similarly to \eqref{proof:est-4} we have
	\begin{align*}
		\frac{\fbar(z_\alpha)-\fbar(\xi)}{\fbar(z_\alpha)}
		& \geq \frac{\fbar(z_\alpha)-\fbar(\xi)}{\xi-z_\alpha}\cdot\frac{\xi-z_\alpha}{2\fbar(\xi)}
		= -\fbar'(\zeta)\cdot\frac{\xi-z_\alpha}{2\fbar(\xi)} \qquad\qquad\text{for some }\zeta\in(z_\alpha,\xi) \\
		& \geq \frac{|\fbar'(\zeta)|}{2\fbar(\zeta)} \cdot (\xi-z_\alpha)
		\geq \frac{(1-\eta)^3}{\omega(1-z_\alpha)}\cdot(\xi-z_\alpha)
	\end{align*}
	for a uniform modulus of continuity $\omega$. Therefore combining this estimate with \eqref{proof:est-7} and inserting them into \eqref{proof:est-6} we can write
	\begin{equation*}
		\Psi(s) \geq \frac{(1-\eta)^4}{\omega(1-z_\alpha)}(\eta-s)(\xi-z_\alpha) \geq \frac{(1-\eta)^4}{\omega(1-z_\alpha)}(\eta-s)(s-z_\alpha) \qquad\text{for all }s\in(z_\alpha,\eta)
	\end{equation*}
	and in turn by \eqref{proof:est-5} we conclude that 
	\begin{equation*}
		t_\eta - t_0 \leq \frac{\omega(1-z_\alpha)}{(1-\eta)^2}\int_{z_\alpha}^\eta \frac{\dd s}{\sqrt{\eta-s}\sqrt{s-z_\alpha}} = \frac{\pi\,\omega(1-z_\alpha)}{(1-\eta)^2}\,.
	\end{equation*}
	By combining this estimate with \eqref{proof:est-t0} we obtain \eqref{eq:est-teta-2}.
	
	We eventually prove the estimate from below \eqref{eq:est-teta-3}. By the assumptions \ref{ass:f2}--\ref{ass:f3} we have that $(1-s)f(s)\leq \stress$ for all $s\in(0,1)$, so that we can bound the function $\Psi$ in \eqref{proof:est-1} by
	\begin{equation*}
		\Psi(s) \leq \frac{(1-s)^2}{4}\biggl[ 1 - \biggl(\frac{2\alpha}{\stress}\biggr)^2 \biggr] + \frac{(1-m)^2}{4}\bigg| 1 - \biggl(\frac{2\alpha}{(1-m)f(m)}\biggr)^2 \bigg|
	\end{equation*}
	and in turn we obtain the rough estimate $\sqrt{\Psi(s)}\leq (1-s) + C_{\alpha,m}$, where $C_{\alpha,m}$ is the constant defined in \eqref{eq:est-teta-3}. By inserting this inequality into \eqref{proof:est-5} and integrating we deduce \eqref{eq:est-teta-3} by an elementary computation.
\end{proof}

We conclude this section by discussing the behaviour of solutions to \eqref{eq:ODE1}--\eqref{eq:ODE3} in the case $\alpha\geq\frac{\stress}{2}$.

\begin{proposition} \label{prop:ODE2}
	Let $\alpha\geq\frac{\stress}{2}$ and $m\in(0,1)$ be given, let $f$ satisfy \ref{ass:f1}--\ref{ass:f5}, and let $\bar{y}$ be the solution to the Cauchy problem \eqref{eq:ODE1}--\eqref{eq:ODE3}. Then there exists $t_1>0$ such that $\bar{y}(t_1)=1$ and $\bar{y}(t)$ is monotone increasing in $(0,t_1)$. Moreover, denoting for $\eta\in(m,1)$ by $t_\eta\in(0,t_1)$ the point such that $\bar{y}(t_\eta)=\eta$, the estimate \eqref{eq:est-teta-2} holds.
\end{proposition}

\begin{proof}
	By the properties of $\fbar$ in Proposition~\ref{prop:salpha} and since $\alpha\geq\frac{\stress}{2}$, the function $h$ in \eqref{def:h} is strictly positive in $(0,1)$. Therefore, since initially we have $\bar{y}'(0)=0$ and $\bar{y}''(0)=h(m)>0$, the solution is strictly increasing and convex for $t>0$ as long as $\bar{y}\in(0,1)$. The existence of a point $t_1>0$ such that $\bar{y}(t_1)=1$ follows easily.
	
	Fix now $\tilde{\alpha}\in(0,\frac{\stress}{2})$ such that $m\in(0,z_{\tilde{\alpha}})$ and $(1-m)f(m)<2\tilde{\alpha}$, and let $\tilde{y}$ be the solution to the Cauchy problem \eqref{eq:ODE1}--\eqref{eq:ODE3} with $\alpha$ replaced by $\tilde{\alpha}$. Since $\alpha\geq\frac{\stress}{2}>\tilde{\alpha}$, by comparison we have $\tilde{y}\leq \bar{y}$.
	
	We can apply to $\tilde{y}$ the qualitative analysis in Theorem~\ref{thm:ODE} and Proposition~\ref{prop:ODE} to deduce, in particular, that for all $\eta\in(m,1)$ there exists $\tilde{t}_\eta$ such that $\tilde{y}(\tilde{t}_\eta)=\eta$, and $\tilde{t}_\eta$ obeys the estimate \eqref{eq:est-teta-2}. Then by comparison $t_\eta\leq\tilde{t}_\eta$ obeys the same estimate, as desired.
\end{proof}


\section{Properties of the cohesive energy density, old and new} \label{sect:g}

We discuss in this section the main properties of the cohesive energy density $g$ defined by \eqref{def:g}, appearing in the $\Gamma$-limit of the phase-field functionals \eqref{def:Feps}. Most of the properties of $g$ have been studied in \cite{ConFocIur16,BonConIur21}, but we also include here some new results: in Proposition~\ref{prop:g3} we prove that $g\in C^1([0,+\infty))$ and a characterization of the derivative of $g$, and in Proposition~\ref{prop:sfrac} we determine explicitly the value of the threshold of complete fracture $\sfrac$ (see \eqref{def:sfrac}). Eventually we determine the asymptotic expansion of $g$ near the origin in terms of the behaviour of $f(s)$ as $s\to1$, see Proposition~\ref{prop:g4}: this is only relevant in connection with the existence of critical points of pre-fractured type, see Remark~\ref{rmk:critical-points}.

The following properties are proved in \cite[Proposition~4.1]{ConFocIur16}.

\begin{proposition}\label{prop:g1}
Assume that $f$ satisfies the assumptions \ref{ass:f1}, \ref{ass:f2}, \ref{ass:f3}. The function $g$ defined in \eqref{def:g} enjoys the following properties:
\begin{enumerate}
\item $g(0)=0$, and $g$ is subadditive;
\item g is nondecreasing, $g(s) \leq \min\{1,\stress s\}$, and $g$ is Lipschitz continuous with Lipschitz constant $\stress$;
\item $\lim_{s\to+\infty}g(s)=1$;
\item  $\lim_{s\to0^+}\frac{g(s)}{s}=\stress$.
\end{enumerate}
\end{proposition}

We studied in \cite{BonConIur21} the existence and properties of optimal pairs for the minimum problem \eqref{def:g} which defines $g$, which we collect in the following two propositions. Notice that the first part of the statement of Proposition~\ref{prop:g2a} does not make use of the convexity assumption \ref{ass:f6}, whereas this condition is crucial to deduce uniqueness of the optimal pair and the properties of the function $\mbar_{s}$ in Proposition~\ref{prop:g2b} (compare with Remark~\ref{rmk:uniqueness}).

\begin{proposition}\label{prop:g2a}
Assume that $f$ satisfies the assumptions \ref{ass:f1}, \ref{ass:f2}, \ref{ass:f3}. Let $g$ be defined by \eqref{def:g} and let
\begin{equation} \label{def:sfrac}
\sfrac\defeq \sup\{s \,:\, g(s)<1 \} \in(0,+\infty].
\end{equation}
Let $s\in(0,\sfrac)$, so that $g(s)<1$. Then:
\begin{enumerate}
\item There exists an optimal pair $(\alpha_s,\beta_s)\in\mathcal{U}_s$ such that $g(s)=\G(\alpha_s,\beta_s)$.
\item If $(\alpha_s,\beta_s)$ is an optimal pair for $g(s)$, then there exist $-\infty\leq T_-<T_*<T_+\leq+\infty$ such that $\{\beta_s<1\}=(T_-,T_+)$, $\beta_s\in\mathrm{C}^1(T_-,T_+)$, $\beta_s$ is symmetric with respect to $T_*$ and nonincreasing in $(-\infty,T_*)$, $\alpha_s\in\mathrm{C}^1(\R)$ is nondecreasing.
\item Any optimal pair $(\alpha_s,\beta_s)$ solves the Euler-Lagrange equations
\begin{equation} \label{eq:opt-prof-1}
\beta_s'' = f(\beta_s)f'(\beta_s)(\alpha_s')^2 - \frac{1-\beta_s}{4}  \qquad\text{weakly in }\{\beta_s<1\},
\end{equation}
\begin{equation} \label{eq:opt-prof-2}
f^2(\beta_s)\alpha_s' = \sigma_s \qquad\text{pointwise on }\{\beta_s<1\}, \text{ for a constant }\sigma_s\in\R.
\end{equation}
\item Any optimal pair $(\alpha_s,\beta_s)$ satisfies pointwise on $\{\beta_s<1\}$
\begin{equation} \label{eq:opt-prof-4}
f^2(\beta_s)|\alpha_s'|^2 + |\beta_s'|^2 - \frac{(1-\beta_s)^2}{4} =0.
\end{equation}
\item Assuming in addition that $f$ satisfies \ref{ass:f6}, the optimal pair is unique up to translations, in the sense that if $(\alpha_s,\beta_s)$ and $(\hat\alpha_s,\hat\beta_s)$ are minimizers then there are $a_1, t_1\in\R$ such that $\alpha_s(t)=a_1+\hat\alpha_s(t-t_1)$, $\beta_s(t)=\hat\beta_s(t-t_1)$.
\end{enumerate}
If instead $s\geq s_{\mathrm{frac}}$, one has $g(s)=1$. In particular $g(s)=\G(0,\beta_s)$, where $\beta_s(t)=1-e^{-\frac{|t|}{2}}$, that is, one can interpret the value $g(s)=1$ as the energy of a configuration $(\alpha_s,\beta_s)\notin\mathcal{U}_s$ where $\alpha_s$ is piecewise constant.
\end{proposition}

\begin{proof}
See the proof of \cite[Proposition~8.1]{BonConIur21}, which only uses \ref{ass:f1}-\ref{ass:f3}, and \cite[Proposition~8.3]{BonConIur21} for the uniqueness part (v). We only sketch here the proof of (i), since some details were missing in \cite[Proposition~8.1(ii)]{BonConIur21}. The existence of an optimal pair $(\alpha_s,\beta_s)$ can be obtained in two steps: in the first, we fix $\beta$ with $(1-\beta)\in H^1(\R)$, $0\leq\beta\leq1$, and $\lim_{|t|\to+\infty}\beta(t)=1$, and we minimize $\G(\cdot,\beta)$ with respect to the first variable; in the second step, we minimize $\G(\alpha,\beta)$ with respect to $\beta$, where $\alpha$ has been obtained in the first step and depends on $\beta$.

To prove the existence of a minimizing profile $\alpha$ (for fixed $\beta$) in the first step, we can assume that $\inf_{t\in\R}\beta(t)>0$, since otherwise $\G(\alpha,\beta)\geq1$ for all $\alpha$. Then, we check that the crucial property $1/f^2(\beta)\in L^1(\R)$ holds. Indeed, by \ref{ass:f2} there exists $K>0$ such that $f(\beta)(1-\beta)\geq\frac{\stress}{2}$ in $\R\setminus[-K,K]$, where we adopt the convention $f(\beta)(1-\beta)=\sigma_c$ if $\beta=1$. Then
$$
\int_{\R}\frac{\dd t}{f^2(\beta(t))} \leq \int_{[-K,K]}\frac{\dd t}{f^2(\beta(t))} + \frac{4}{\stress^2}\int_{\R\setminus[-K,K]}(1-\beta(t))^2\dd t < +\infty.
$$
We define
$$
\alpha(t) \defeq c\int_0^t \frac{\dd\tau}{f^2(\beta(\tau))} \qquad\text{where } c\defeq s\biggl(\int_{\R}\frac{1}{f^2(\beta)}\dd \tau\biggr)^{-1}.
$$
In this way $(\alpha,\beta)\in\mathcal{U}_s$ and $\alpha$ minimizes $\G(\cdot,\beta)$ for fixed $\beta$, by convexity of the energy.

Next, as a second step we minimize in $\beta$. Notice that, in the minimization problem \eqref{def:g}, we can restrict to profiles such that $\beta$ is nonincreasing in $(-\infty,0)$ and nondecreasing in $(0,+\infty)$, by exploiting the same argument as in \cite[Proposition~8.1~(v)]{BonConIur21}. We define
\begin{equation} \label{def:Gh}
\Gh_s(\beta) \defeq s^2\biggl(\int_{-\infty}^{+\infty}\frac{1}{f^2(\beta(t))}\dd t\biggr)^{-1} + \int_{-\infty}^{+\infty} \biggl( \frac{(1-\beta)^2}{4} + |\beta'|^2 \biggr) \dd t,
\end{equation}
so  that $\Gh_s(\beta)=\G(\alpha_\beta,\beta)$, where $\alpha_\beta$ is the optimal profile associated to a given function $\beta$, as constructed in the previous step. In this way
\begin{equation} \label{new-proof0}
g(s) = \inf_{\beta}\Gh_s(\beta)
\end{equation}
where the infimum is taken over all profiles $\beta$ such that $1-\beta\in H^1(\R)$ and $0\leq\beta\leq1$. Let $(\beta_k)_k$ be a minimizing sequence for $\Gh_s$, that is $g(s)=\lim_{k}\Gh_s(\beta_k)$.  The sequence $1-\beta_k$ is bounded in $H^1(\R)$ and has therefore a subsequence (not relabeled) converging weakly and uniformly on compact sets to some function $1-\beta\in H^1(\R)$, with $\min_{\R}\beta=\beta(0)\in(0,1)$ (or else $\beta(0)=0$ would imply $g(s)\geq 1$, while $\beta(0)=1$ would give $g(s)=0$, hence $s=0$).

We let
\begin{equation*}
    \mu \defeq \limsup_{k\to+\infty}\int_{-\infty}^{\infty}(1-\beta_k)^2\dd t - \int_{-\infty}^{\infty}(1-\beta)^2\dd t \geq0,
\end{equation*}
and we also observe that
\begin{equation} \label{new-proof2}
    \limsup_{k\to\infty}\int_{\R}\frac{\dd t}{f^2(\beta_k)} \leq \int_{\R}\frac{\dd t}{f^2(\beta)} + \frac{\mu}{\stress^2}.
\end{equation}
Indeed, given $\e>0$ we can choose $T>0$ so large that $(1-\beta(\pm T))f(\beta(\pm T))>\stress -\e$ (using assumption \ref{ass:f2}).
{Then also $(1-\beta_k(\pm T))f(\beta_k(\pm T))\geq \stress - \e$ for all $k$ sufficiently large; in turn, by monotonicity of $\beta_k$ and of the map $(1-s)f(s)$, we also have $(1-\beta_k(t))f(\beta_k(t))\geq \stress - \e$ for all $|t|\geq T$ and for all $k$ sufficiently large.} By using the uniform convergence $\beta_k\to\beta$ on $[-T,T]$ we have
\begin{align*}
    \limsup_{k\to\infty}\int_{\R}\frac{\dd t}{f^2(\beta_k)} 
    & \leq \int_{-T}^{T}\frac{\dd t}{f^2(\beta)} + \limsup_{k\to\infty}\int_{\R\setminus[-T,T]}\frac{(1-\beta_k)^2}{(1-\beta_k)^2f^2(\beta_k)}\dd t \\
    & \leq \int_{-T}^{T}\frac{\dd t}{f^2(\beta)} + \limsup_{k\to\infty} \frac{1}{(\stress-\e)^2}\int_{\R\setminus[-T,T]}(1-\beta_k)^2\dd t \\
    & = \int_{-T}^{T}\frac{\dd t}{f^2(\beta)} + \frac{1}{(\stress-\e)^2}\int_{\R\setminus[-T,T]}(1-\beta)^2\dd t + \frac{\mu}{(\stress-\e)^2}\,,
\end{align*}
and by passing to the limit first as $T\to+\infty$ and then as $\e\to0$ we obtain \eqref{new-proof2}.

Then by definition of $\mu$ and by \eqref{new-proof2}
\begin{align} \label{new-proof1}
    g(s) & = \lim_{k\to+\infty}\Gh_s(\beta_k)
    \geq s^2\biggl(\frac{\mu}{\stress^2} + \int_{\R}\frac{\dd t}{f^2(\beta)}\biggr)^{-1} + \frac{\mu}{4} + \int_{-\infty}^{+\infty} \biggl( \frac{(1-\beta)^2}{4} + |\beta'|^2 \biggr) \dd t.
\end{align}
If $\mu=0$, then the previous inequality shows that $g(s)\geq\Gh_s(\beta)$, and therefore $\beta$ is a minimizer in \eqref{new-proof0} and the proof is concluded. If instead $\mu>0$, we modify the profile $\beta$ as follows: define
$$
\hat{\beta}(t)\defeq
\begin{cases}
    \beta(0) & \text{if }|t|\leq L, \\
    \beta(t-L) & \text{if }t> L, \\
    \beta(t+L) & \text{if }t<-L, \\
\end{cases}
$$
where $L>0$ is chosen so that $2L(1-\beta(0))^2=\mu$. With this choice
\begin{align*}
    \Gh_s(\hat{\beta})
    &= s^2\biggl( \frac{2L}{f^2(\beta(0))} + \int_{\R}\frac{\dd t}{f^2(\beta)} \biggr)^{-1} + 2L\,\frac{(1-\beta(0))^2}{4} + \int_{-\infty}^{+\infty} \biggl( \frac{(1-\beta)^2}{4} + |\beta'|^2 \biggr) \dd t \\
    & = s^2\biggl( \frac{\mu}{(1-\beta(0))^2f^2(\beta(0))} + \int_{\R}\frac{\dd t}{f^2(\beta)} \biggr)^{-1} + \frac{\mu}{4} + \int_{-\infty}^{+\infty} \biggl( \frac{(1-\beta)^2}{4} + |\beta'|^2 \biggr) \dd t \\
    & \leq s^2\biggl( \frac{\mu}{\stress^2} + \int_{\R}\frac{\dd t}{f^2(\beta)} \biggr)^{-1} + \frac{\mu}{4} + \int_{-\infty}^{+\infty} \biggl( \frac{(1-\beta)^2}{4} + |\beta'|^2 \biggr) \dd t
    \xupref{new-proof1}{\leq} g(s)
\end{align*}
where in the first inequality we used the fact that $(1-\beta(0))f(\beta(0))\leq \stress$ by assumption \ref{ass:f3}. This shows that $\hat{\beta}$ is a minimizer in \eqref{new-proof0} and concludes the proof also in the case $\mu>0$.
\end{proof}

\begin{proposition}\label{prop:g2b}
Assume that $f$ satisfies the assumptions \ref{ass:f1}--\ref{ass:f6}. For $s\in(0,\sfrac)$ let $(\alpha_s,\beta_s)$ be the optimal pair for $g(s)$ given by Proposition~\ref{prop:g2a}, translated so that $\beta_s$ attains its minimum at the origin; for $s\geq\sfrac$ let $\beta_s(t)=1-e^{-\frac{|t|}{2}}$. Define 
\begin{equation} \label{minbeta1}
\mbar_s \defeq \min_{t\in\R}\beta_s(t) = \beta_s(0)
\end{equation}
(uniquely determined by $s$). The map $s\mapsto\mbar_s$ is continuous, strictly decreasing in $[0,s_{\mathrm{frac}})$, with $\mbar_0=1$ and $\mbar_s=0$ for $s\geq s_{\mathrm{frac}}$. 

Moreover, one has that $\{\beta_s<1\}=\R$, and the constant $\sigma_s$ in \eqref{eq:opt-prof-2} is given by
\begin{equation} \label{eq:ms-1}
\sigma_s = \frac12 (1-\mbar_s)f(\mbar_s).
\end{equation}
Denoting by $m\mapsto s(m)$ the inverse of the map $s\mapsto \mbar_s$, one has for all $m\in(0,1)$ 
\begin{equation} \label{eq:ms-3}
	s(m) = 2(1-m)f(m)\int_{m}^1 \frac{1}{f(t)\bigl((1-t)^2f^2(t) - (1-m)^2f^2(m)\bigr)^{1/2}} \dd t,
\end{equation}	
\begin{equation} \label{eq:ms-2}
g(s(m)) = 2\int_{m}^1 \frac{(1-t)^2f(t)}{\bigl((1-t)^2f^2(t) - (1-m)^2f^2(m)\bigr)^{1/2}} \dd t.
\end{equation}
\end{proposition}

\begin{proof}
The first part of the statement concerning the properties of the map $\mbar_s$ is proved in \cite[Proposition~8.3]{BonConIur21} without the hypotheses \ref{ass:f4}-\ref{ass:f5}. By substituting \eqref{eq:opt-prof-2} into \eqref{eq:opt-prof-4} and evaluating the resulting equation at the origin, recalling that $\beta_s(0)=\mbar_s$ and $\beta_s'(0)=0$, we find \eqref{eq:ms-1}. Inserting \eqref{eq:opt-prof-2} into \eqref{eq:opt-prof-1} we find that $\beta_s$ solves in the set $\{\beta_s<1\}$ the Cauchy problem \eqref{eq:ODE1}--\eqref{eq:ODE3} for the values of the parameters $\alpha=\sigma_s$ and $m=\mbar_{s}$. By \eqref{eq:ms-1} we can apply case \ref{item:ODE2} in Theorem~\ref{thm:ODE} to deduce that $\beta_s<1$ in $\R$ and $\beta_s'>0$ in $(0,+\infty)$.

Finally, we prove \eqref{eq:ms-3} and \eqref{eq:ms-2}. We have by a change of variable
\begin{equation*}
	\begin{split}
		s
		& = \int_{\R} \alpha_s'\dd t
		\xupref{eq:opt-prof-2}{=} 2\int_{0}^{+\infty} \frac{\sigma_s}{f^2(\beta_s)}\dd t
		\xupref{eq:opt-prof-4}{=} 2\sigma_s \int_0^{+\infty} \frac{\beta_s'}{f^2(\beta_s)} \biggl(\frac{(1-\beta_s)^2}{4}-\frac{\sigma_s^2}{f^2(\beta_s)}\biggr)^{-\frac12} \dd t\\
		& = 2\sigma_s \int_{\mbar_s}^{1} \frac{1}{f^2(t)} \biggl(\frac{(1-t)^2}{4}-\frac{\sigma_s^2}{f^2(t)}\biggr)^{-\frac12} \dd t 
		= 4\sigma_s \int_{\mbar_s}^{1} \frac{\dd t}{f(t)\bigl[(1-t)^2f^2(t)-4\sigma_s^2\bigr]^{1/2}}
	\end{split}
\end{equation*}
from which we obtain \eqref{eq:ms-3} thanks to \eqref{eq:ms-1}. Similarly
\begin{equation*}
	\begin{split}
		g(s)
		& = \G(\alpha_s,\beta_s)
		\xupref{eq:opt-prof-4}{=} 2\int_{-\infty}^{+\infty}\frac{(1-\beta_s)^2}{4} \dd t
		= \int_{0}^{+\infty}(1-\beta_s)^2\dd t \\
		& \xupref{eq:opt-prof-4}{=} \int_0^{+\infty} (1-\beta_s)^2 \beta_s' \biggl(\frac{(1-\beta_s)^2}{4}-\frac{\sigma_s^2}{f^2(\beta_s)}\biggr)^{-\frac12} \dd t \\
		& = \int_{\mbar_s}^{1} (1-t)^2\biggl(\frac{(1-t)^2}{4}-\frac{\sigma_s^2}{f^2(t)}\biggr)^{-\frac12} \dd t
		= \int_{\mbar_s}^{1} \frac{2(1-t)^2 f(t)}{\bigl[(1-t)^2f^2(t)-4\sigma_s^2\bigr]^{1/2}}\dd t
	\end{split}
\end{equation*}
from which we obtain \eqref{eq:ms-2} thanks to \eqref{eq:ms-1}.
\end{proof}

\begin{remark} \label{rmk:sm-explicit}
In the prototype case $f(t)=t/(1-t)$, $t\in[0,1)$, equations \eqref{eq:ms-3}-\eqref{eq:ms-2} give explicit formulas for the maps $m\in(0,1)\mapsto s(m)\in(0,\sfrac)$ and $m\in(0,1)\mapsto g(s(m))\in(0,1)$ (notice that in this case $\sfrac=\pi$ by Proposition~\ref{prop:g4} below):
\[s(m)=2\arctan\Big(\frac{\sqrt{1-m^2}}{m}\Big)-2m\log\Big(\frac{1+\sqrt{1-m^2}}{m}\Big),\]
\[g(s(m))=m^2\log\Big(\frac{m}{\sqrt{1-m^2}+1}\Big)+\sqrt{1-m^2}.\]
See also Figure~\ref{fig:ms} for numeric plots obtained using these expressions.
\begin{figure}
	\begin{center}
		\includegraphics[width=7cm]{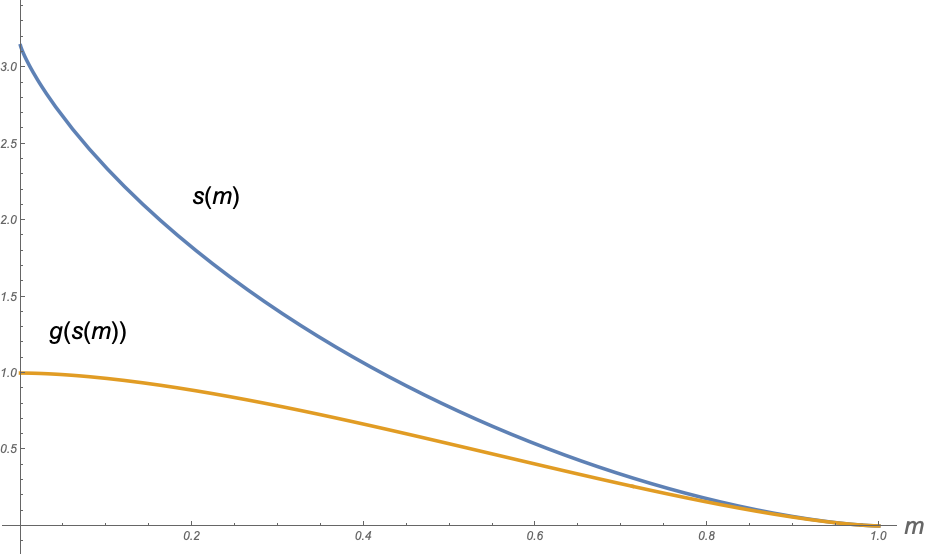} \qquad
		\includegraphics[width=6cm]{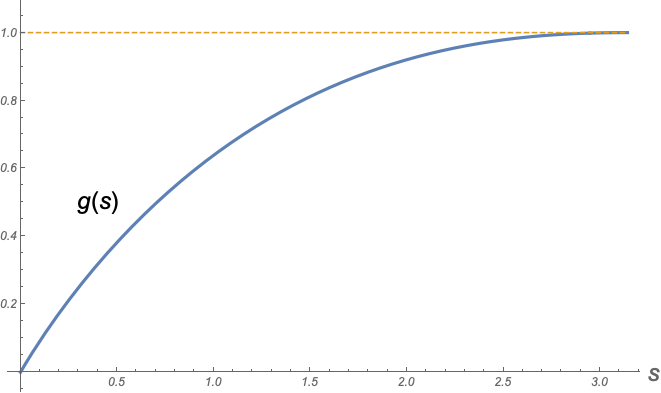}
	\end{center}
	\caption{Left: plots of the functions $s(m)$ and $g(s(m))$ in \eqref{eq:ms-3}--\eqref{eq:ms-2}, in the prototype case $f(t)=t/(1-t)$, see Remark~\ref{rmk:sm-explicit}. Right: plot of the function $g(s)$ obtained by using the expressions in Remark~\ref{rmk:sm-explicit} and inverting numerically the function $s(m)$.}
	\label{fig:ms}
\end{figure}
\end{remark}

\begin{remark} \label{rmk:uniqueness}
If we drop the assumption \ref{ass:f6}, there is numerical evidence (using the function $f_{(q)}$ in \eqref{eq:examplefq} for $q>2$) that the conclusions of Proposition~\ref{prop:g2b} are no longer true in general. However, it seems that it is still possible to define a one-to-one correspondence $s\in (0,\sfrac)\mapsto \mbar_s\in (m_{\text{frac}},1)$, for some $m_{\text{frac}}>0$.
\end{remark}

We next show that $g$ is differentiable and characterize its derivative in terms of the minimum value $\mbar_s$ (see \eqref{minbeta1}) of the optimal profile $\beta_s$.

\begin{proposition} \label{prop:g3}
Under the assumptions \ref{ass:f1}--\ref{ass:f6}, the function $g$ defined in \eqref{def:g} satisfies
\begin{equation} \label{eq:gprime}
g'(s) = (1-\mbar_s)f(\mbar_s) \qquad\text{for all }s\geq 0.
\end{equation}
\end{proposition} 

\begin{proof}
We first observe that the identity \eqref{eq:gprime} is valid for $s=0$ since $g'(0)=\stress$ by Proposition~\ref{prop:g1}, $\mbar_0=1$ by Proposition~\ref{prop:g2b}, and $\lim_{t\to1}(1-t)f(t)=\stress$ by assumption~\ref{ass:f2}. It is also valid for $s>\sfrac$, since $g(s)\equiv1$ and $\mbar_s=0$ for $s\in[\sfrac,+\infty)$.

We then consider the case $s\in(0,\sfrac)$. Let $(\alpha_s,\beta_s)\in\mathcal{U}_s$ be an optimal pair for $g(s)$, according to Proposition~\ref{prop:g2a}. Let $s,s'\in(0,\sfrac)$ be two arbitrary points. Since $(\frac{s}{s'}\alpha_{s'},\beta_{s'})\in\mathcal{U}_s$ is an admissible competitor for the minimum problem \eqref{def:g} defining $g(s)$, we have
\begin{align*}
g(s)
& \leq \G\bigl({\ts\frac{s}{s'}}\,\alpha_{s'},\beta_{s'}\bigr)
= \int_{-\infty}^{+\infty} \biggl( \Bigl(\frac{s}{s'}\Bigr)^2f^2(\beta_{s'})|\alpha'_{s'}|^2 + \frac{(1-\beta_{s'})^2}{4} + |\beta_{s'}'|^2 \biggr) \dd t \\
& = \G(\alpha_{s'},\beta_{s'}) + \biggl(\Bigl(\frac{s}{s'}\Bigr)^2-1\biggr)\int_{-\infty}^{+\infty}f^2(\beta_{s'})|\alpha'_{s'}|^2\dd t \\
& \xupref{eq:opt-prof-2}{=} g(s') + (s-s')\cdot\frac{s+s'}{(s')^2}\cdot\sigma_{s'}\int_{-\infty}^{+\infty}\alpha_{s'}'\dd t \\
& \xupref{eq:ms-1}{=} g(s') + (s-s') \cdot \frac{s+s'}{2s'} \cdot (1-\mbar_{s'})f(\mbar_{s'})
\end{align*}
whence
\begin{equation}\label{proof:gprime-4}
g(s) \leq g(s') + (s-s')\cdot\frac{s+s'}{2s'}\cdot (1-\mbar_{s'})f(\mbar_{s'}) \qquad\text{for all }s,s'\in(0,\sfrac).
\end{equation}
If $s'<s$, by dividing both sides in \eqref{proof:gprime-4} by $(s-s')$ we get
\begin{equation*}
\frac{g(s)-g(s')}{s-s'} \leq \frac{s+s'}{2s'}\cdot (1-\mbar_{s'})f(\mbar_{s'}),
\end{equation*}
and letting $s\to(s')^+$ or $s'\to(s)^-$, using the continuity of $s\mapsto \mbar_s$, we obtain the following inequalities for the right and left derivatives of $g$:
\begin{equation*}
g'_+(s') \leq (1-\mbar_{s'})f(\mbar_{s'}), \qquad g'_-(s) \leq (1-\mbar_s)f(\mbar_s).
\end{equation*}
By arguing similarly for $s'>s$ in \eqref{proof:gprime-4}, we obtain the opposite inequalities, so that \eqref{eq:gprime} follows for all $s\in(0,\sfrac)$.

To conclude the proof, it only remains to show that $g$ is differentiable also at $\sfrac$ and $g'(\sfrac)=0$. This is easily obtained since $\lim_{s\to\sfrac^-}g'(s)=0$ by the first part of the proof and $g\equiv 1$ for $s>s_{\mathrm{frac}}$.
\end{proof}

A condition for the finiteness of the threshold $\sfrac$ is given in \cite[Proposition~6.3]{BonConIur21}. In the following proposition we improve that result by determining explicitly the value of $\sfrac$ in terms of the derivative of $f$ at the origin.

\begin{proposition} \label{prop:sfrac}
Under the assumptions \ref{ass:f1}--\ref{ass:f6}, the threshold $\sfrac$ defined in \eqref{def:sfrac} is given by
\begin{equation} \label{eq:sfrac}
\sfrac=\frac{\pi}{f'(0)}
\end{equation}
(with $\sfrac=+\infty$ if $f'(0)=0$).
\end{proposition} 

\begin{proof}
For any fixed $\delta\in(0,1)$ we have by \eqref{eq:ms-3}
\begin{equation*} \label{proof:sfrac1}
\begin{split}
\sfrac
& = \lim_{s\to\sfrac^-} 2(1-\mbar_{s})f(\mbar_{s})\int_{\mbar_{s}}^1 \frac{\dd t}{f(t)\bigl[(1-t)^2f^2(t) - (1-\mbar_{s})^2f^2(\mbar_{s})\bigr]^{1/2}} \\
& = \lim_{s\to\sfrac^-} 2(1-\mbar_s)f(\mbar_s) \int_{\mbar_s}^{\delta} \frac{\dd t}{f(t)\bigl[(1-t)^2f^2(t)-(1-\mbar_s)^2f^2(\mbar_s)\bigr]^{1/2}}\,,
\end{split}
\end{equation*}
where in the last equality we have used that the denominator of the integrand is uniformly far from $0$ in $(\delta,1)$ and that $\mbar_{s}\to0$.
Denoting by $\phi(t)=(1-t)f(t)$, which is strictly increasing by assumption~\ref{ass:f3}, we can write the previous identity as
\begin{align*}
\sfrac
& = \lim_{s\to\sfrac^-} 2\phi(\mbar_s) \int_{\mbar_s}^{\delta} \frac{\phi'(t)}{\phi(t)\sqrt{\phi^2(t)-\phi^2(\mbar_s)}}\cdot\frac{(1-t)}{\phi'(t)}\dd t \\
& \leq \lim_{s\to\sfrac^-} 2\phi(\mbar_s) \biggl(\sup_{t\in(\mbar_s,\delta)}\frac{(1-t)}{\phi'(t)}\biggr) \int_{\mbar_s}^{\delta} \frac{\phi'(t)}{\phi(t)\sqrt{\phi^2(t)-\phi^2(\mbar_s)}}\dd t \\
& = \lim_{s\to\sfrac^-} 2\biggl(\sup_{t\in(\mbar_s,\delta)}\frac{(1-t)}{\phi'(t)}\biggr) \arctan\biggl(\frac{\sqrt{\phi^2(\delta)-\phi^2(\mbar_s)}}{\phi(\mbar_s)}\biggr)
= \pi\sup_{t\in(0,\delta)}\frac{(1-t)}{\phi'(t)}
\end{align*}
since $\mbar_s\to0$ as $s\to\sfrac^-$. Arguing similarly we find
\begin{equation*}
\pi\inf_{t\in(0,\delta)}\frac{(1-t)}{\phi'(t)} \leq \sfrac \leq \pi \sup_{t\in(0,\delta)}\frac{(1-t)}{\phi'(t)}.
\end{equation*}
The conclusion follows by letting $\delta\to0^+$.
\end{proof}

\begin{remark}
For the prototype examples in Remark~\ref{rmk:examplef} one has $\sfrac=\frac{\pi}{q\stress}\in\R$ for the functions $f_{(q)}$, and $\sfrac=+\infty$ for the functions $f^{(p)}$.
\end{remark}

We conclude this section by determining the asymptotic expansion of the cohesive energy density $g$ at the origin.

\begin{proposition} \label{prop:g4}
Assume that $f$ satisfies the assumptions \ref{ass:f1}--\ref{ass:f6}, and assume further that
\begin{equation} \label{eq:asympf}
(1-s)f(s)=\stress -\ell (1-s)^q + o((1-s)^q) \qquad\text{as }s\to1^-
\end{equation}
for some $\ell>0$ and $q\in(0,2]$, where $o(t)$ denotes any quantity such that $\lim_{t\to0}\frac{o(t)}{t}=0$.
Then the function $g$ defined in \eqref{def:g} satisfies, for $p\defeq\frac{4+q}{4-q}$ and for some $\tilde\ell>0$,
\begin{equation} \label{eq:asympg}
g(s) = \stress s - \tilde\ell s^p + o(s^p) \qquad\text{as }s\to0^+.
\end{equation}
\end{proposition}

\begin{proof}
The proof is a straightforward adaptation of \cite[Proposition~8.5]{BonConIur21}, which deals with the case $q=1$. The details are left to the reader.
\end{proof}

\begin{remark} \label{rmk:asymp}
The function $f_{(q)}$ for $q\in(0,4)$ in Remark~\ref{rmk:examplef} satisfies the condition \eqref{eq:asympf}.
\end{remark}


\section{Preliminary properties of critical points} \label{sect:proof-prelim}

We assume along this section that $(u_\e,v_\e)\in (H^1(0,L))^2$ is a family of critical points of the energies $\Feps$, \textit{i.e.}\ they are weak solutions to the system of equations \eqref{def:cp1}--\eqref{def:cp4}. We also suppose that the Dirichlet boundary condition satisfies \eqref{def:bc}, and that the equiboundedness of the energy \eqref{eq:bound-energy} holds. We first remark that by \eqref{def:cp2} there exist constants $c_\e\in\R$ such that
\begin{equation} \label{eq:ceps}
f_\e^2(v_\e)u_\e' = c_\e \qquad\text{a.e.\ in }(0,L).
\end{equation}
Notice that from \eqref{eq:ceps} it follows that $u_\e'$ has constant sign a.e.\ in $(0,L)$ and therefore, in view of the boundary conditions, it must be $c_\e\geq0$, so that $u_\e$ is monotone nondecreasing in $(0,L)$. Moreover it cannot be $c_\e=0$, or otherwise the second term in \eqref{def:cp1} would vanish and $v_\e$ would be a weak solution to $-\e v_\e'' + \frac{v_\e-1}{4\e}=0$, with $v_\e(0)=v_\e(L)=1$; however, this would imply that $v_\e\equiv 1$ and in turn, by \eqref{eq:ceps}, $u_\e'=0$ almost everywhere in $(0,L)$, which is not possible in view of the boundary conditions \eqref{def:cp3}. Therefore $c_\e>0$ for all $\e$.

We also have
\begin{equation*}
\Feps(u_\e,v_\e) \geq \int_0^L f_\e^2(v_\e)|u_\e'|^2\dd x = c_\e \int_0^L u_\e'\dd x = c_\e a_\e,
\end{equation*}
so that by \eqref{def:bc} and \eqref{eq:bound-energy} we obtain that $\sup_{\e}c_\e<+\infty$ and, up to subsequences,
\begin{equation} \label{eq:c0}
c_\e \to c_0 \geq0 \qquad\text{as }\e\to0.
\end{equation}

Similarly to \cite[Lemma~3.1]{FraLeSer09} and to \cite[Lemma~2.3]{BabMilRodb}, we show that $v_\e$ obeys a maximum principle.

\begin{lemma} \label{lemma:maxprin}
We have that $0\leq v_\e\leq 1$ in $[0,L]$.
\end{lemma}

\begin{proof}
By testing \eqref{def:cp1} with the function $\vphi_\e\defeq\max\{-v_\e,0\}$, which is admissible since $\vphi_\e\in H^1_0(0,L)$ in view of the boundary conditions \eqref{def:cp4}, we have
\begin{align*}
-\int_0^L \e(\vphi_\e')^2\dd x +\int_0^L f_\e(v_\e)f_\e'(v_\e)(u_\e')^2\vphi_\e \dd x - \int_0^L \frac{\vphi_\e+1}{4\e}\,\vphi_\e\dd x = 0,
\end{align*}
and since all the terms are nonpositive (recall that $f_\e(s)\defeq\sqrt{\e}f'(0)s$ for $s<0$) we deduce that $\int_0^L \frac{\vphi_\e+1}{4\e}\vphi_\e\dd x=0$, and therefore $v_\e\geq0$. Similarly, by testing the equation with the function $\vphi_\e\defeq\max\{0,v_\e-1\}\in H^1_0(0,L)$ we prove that $v_\e\leq 1$.
\end{proof}

We next show that the solutions to the equations \eqref{def:cp1}--\eqref{def:cp4} satisfy a conservation law, which can also be seen as a consequence of the vanishing of the first variation of the functional $\Feps$ with respect to \textit{inner variations}.

\begin{proposition} \label{prop:firstintegral}
There exist constants $d_\e\in\R$ such that
\begin{equation} \label{eq:first-int}
\frac{(1-v_\e)^2}{4\e}-f_\e^2(v_\e)(u_\e')^2 - \e (v_\e')^2 = d_\e \qquad\text{in }(0,L),
\end{equation}
with $\sup_{\e}|d_\e|<+\infty$ and, up to subsequences, $d_\e\to d_0$ as $\e\to0$.

Furthermore, $u_\e,v_\e\in \mathrm{C}^2(0,L)$ with $v_\e>0$ in $[0,L]$, and the equations \eqref{def:cp1}--\eqref{def:cp4} hold in the classical sense.
\end{proposition}

\begin{proof}
We first remark that, as $v_\e$ is a weak solution to \eqref{def:cp1}, we have $v_\e''\in L^1(0,L)$ and therefore $v_\e\in C^1([0,L])$. By \eqref{eq:ceps} we have $u_\e' = \frac{c_\e}{f_\e^2(v_\e)}$ almost everywhere in the open set $\{v_\e>0\}$, that is, $u_\e'$ is (almost everywhere equal to) a $\mathrm{C}^1$-function in $\{v_\e>0\}$. In particular $u_\e$ is of class $\mathrm{C}^2$ in $\{v_\e>0\}$. In turn, by \eqref{def:cp1} the same holds for $v_\e$, and equation \eqref{def:cp1} holds in the classical sense in $\{v_\e>0\}$.

We can thus differentiate the left-hand side of \eqref{eq:first-int} in $\{ v_\e>0 \}$:
\begin{align*}
\frac12\biggl( \frac{(1-v_\e)^2}{4\e} &- f_\e^2(v_\e)(u_\e')^2 - \e (v_\e')^2 \biggr)' \\
& = v_\e' \biggl(-\e v_\e'' -f_\e(v_\e)f_\e'(v_\e)(u_\e')^2+\frac{v_\e-1}{4\e}\biggr)-f_\e^2(v_\e)u_\e'u_\e'' \\
& \xupref{def:cp1}{=} -2v_\e'f_\e(v_\e)f_\e'(v_\e)(u_\e')^2-f_\e^2(v_\e)u_\e'u_\e'' \\
& = -u_\e' \Bigl(f_\e^2(v_\e)u_\e'\Bigr)'
\xupref{def:cp2}{=}0,
\end{align*}
hence \eqref{eq:first-int} holds in $\{v_\e>0\}$, for a constant $d_\e$ possibly changing with the connected components of $\{v_\e>0\}$.

We next show that $v_\e>0$ everywhere in $[0,L]$. Consider any connected component $(a,b)$ of the open set $\{v_\e>0\}$. By combining \eqref{eq:first-int} and \eqref{eq:ceps} we have
\begin{equation} \label{proof:first-int}
\frac{(1-v_\e)^2}{4\e}-\frac{c_\e^2}{f_\e^2(v_\e)} - \e (v_\e')^2 = d_\e \qquad\text{in } (a,b).
\end{equation}
Assume by contradiction that $v_\e$ vanishes at one of the endpoints, say $v_\e(a)=0$. The point $a$ must be in the interior $(0,L)$ by \eqref{def:cp4}; since $a$ is a minimum point of $v_\e$ by Lemma~\ref{lemma:maxprin} and $v_\e$ is of class $C^1([0,L])$, we have $v_\e'(a)=0$. We can pass to the limit in \eqref{proof:first-int} as $x\to a$ from the interior of $(a,b)$:
$$
\frac{1}{4\e} - d_\e = \lim_{x\to a}\frac{c_\e^2}{f_\e^2(v_\e(x))}
$$
and since $f_\e(v_\e(x))\to f_\e(0)=0$, we conclude that it must be $c_\e=0$. However, we already observed that $c_\e>0$ (see the discussion after \eqref{eq:ceps}), which is a contradiction proving that $(a,b)=(0,L)$, and since $v_\e$ cannot vanish at the endpoints we obtain $\{v_\e>0\}=[0,L]$.

Finally, by integrating \eqref{eq:first-int} on $(0,L)$ we also have
\begin{equation*}
|d_\e| = \frac{1}{L} \int_0^L \biggl| \frac{(1-v_\e)^2}{4\e} - f_\e^2(v_\e) (u_\e')^2 - \e(v_\e')^2 \biggr| \dd x \leq \frac{1}{L}\Feps(u_\e,v_\e),
\end{equation*}
and therefore $\sup_\e|d_\e|<+\infty$ by \eqref{eq:bound-energy}.
\end{proof}

Notice that, by using \eqref{eq:ceps}, we can rewrite \eqref{eq:first-int} in the form
\begin{equation} \label{eq:first-integral2}
\frac{(1-v_\e)^2}{4\e}-\frac{c_\e^2}{f_\e^2(v_\e)} - \e (v_\e')^2 = d_\e \qquad\text{in }(0,L).
\end{equation}
Similarly, we can rewrite \eqref{def:cp1} as an equation for the function $v_\e$ alone:
\begin{equation} \label{eq:cp1}
-\e v_\e'' + \frac{c_\e^2f_\e'(v_\e)}{f_\e^3(v_\e)}+\frac{v_\e-1}{4\e}=0 \qquad\text{in }(0,L).
\end{equation}
From this equation we can deduce the symmetry properties of the function $v_\e$, similarly to \cite[Lemma~4.1 and Proposition~4.2]{FraLeSer09} and to \cite[Proposition~2.1]{BabMilRodb}.

\begin{lemma} \label{lemma:symmetry}
The graph of $v_\e$ in $[0,L]$ is a symmetric ``well'': it is symmetric with respect to the point $\midp$, which is a global minimum, and $v_\e$ is decreasing in $(0,\midp)$.
\end{lemma}

\begin{proof}
By \eqref{eq:cp1}, $v_\e$ is a solution to an equation of the form $v_\e'' = h_\e(v_\e)$, where the function $h_\e$ is Lipschitz continuous in $(0,+\infty)$ by definition \eqref{def:feps} of $f_\e$ and by assumption \ref{ass:f1}. Notice that $v_\e$ takes values in $(0,1]$ by Lemma~\ref{lemma:maxprin}, and therefore Cauchy-Lipschitz theorem guarantees uniqueness.

If $v_\e$ is not identically equal to 1, by Rolle's Theorem there exists at least one critical point in $(0,L)$. Given any critical point $x_0\in(0,L)$ of $v_\e$, we can symmetrize the graph of $v_\e$ about the vertical line through $x_0$: more precisely, if $x_0\in(0,\frac{L}{2}]$ then we define $\tilde{v}_\e:(0,2x_0)\to\R$ by $\tilde{v}_\e(x)=v_\e(x)$ for $x\in(0,x_0)$, $\tilde{v}_\e(x)=v_\e(2x_0-x)$ for $x\in(x_0,2x_0)$. Then $\tilde{v}_\e$ is also a solution of $\tilde{v}_\e'' = h_\e(\tilde{v}_\e)$ in $(0,2x_0)$, with $\tilde{v}_\e(x_0)=v_\e(x_0)$, $\tilde{v}_\e'(x_0)=v_\e'(x_0)=0$, and Cauchy-Lipschitz theorem yields that $v_\e=\tilde{v}_\e$ in $(x_0,2x_0)$. In particular the critical point $x_0$ is either a maximum or a minimum point. A symmetric argument can be repeated in the case of a critical point in the interval $(\frac{L}{2},L)$.

Therefore the graph of $v_\e$ is symmetric with respect to all the vertical lines passing through its critical points, which are either absolute maximum or absolute minimum points. If there is an interior maximum point at $x_0\in(0,L)$, since $v_\e(0)=v_\e(L)=1$ it must be $v_\e(x_0)=1$, $v_\e'(x_0)=0$. Then by uniqueness we conclude that $v_\e\equiv1$, since the constant function 1 is also a solution of \eqref{eq:cp1}. Hence, if $v_\e$ is not identically equal to 1, there are no interior maximum points and therefore there is a unique interior critical (minimum) point, located ad $\midp$. The symmetric structure described in the statement follows.
\end{proof}

\begin{remark} \label{rmk:symmetry}
The symmetry property of $v_\e$ proved in the previous lemma is in accordance with the result in \cite{BabMilRodb} for the Ambrosio-Tortorelli functional, where it is shown that, imposing the Dirichlet boundary conditions on $v_\e$, a much stronger symmetry property is obtained than in the Neumann case considered in \cite{FraLeSer09}, namely that $v_\e$ has \emph{a unique critical point} located at the midpoint $\frac{L}{2}$. We expect that, also in our setting, imposing Neumann conditions $v_\e'(0)=v_\e'(L)=0$ we would obtain a weaker symmetry property, namely that there exists $n_\e\in\N$ such that the graph of $v_\e$ in $(0,L)$ is made of $n_\e$ repeated identical subgraphs, each of which is a symmetric ``well'' (with a unique interior critical point, which is a global minimum, and two maxima at the endpoints), or a symmetric ``bell'' (with a unique interior critical point, which is a global maximum, and two minima at the endpoints).
\end{remark}

We conclude this section by collecting in the following lemma the compactness properties of the family $(u_\e,v_\e)$, together with a uniform bound of $v_\e'$.

\begin{lemma} \label{lemma:compactness}
We have that $v_\e\to1$ in $L^1(0,L)$ and, up to extracting a subsequence $\e_k\to0$, $u_\e\to u$ in $L^1(0,L)$ for some $u\in\BV(0,L)$ with $|Du|(0,L)\leq a$. Moreover $\e\|v_\e'\|_\infty\leq 1$ for all $\e$ sufficiently small.
\end{lemma}

\begin{proof}
In view of the bound \eqref{eq:bound-energy} we have that $\sup_{\e}\frac{1}{\e}\int_0^L(1-v_\e)^2\dd x<+\infty$, which immediately implies the convergence of $v_\e$. Since $u_\e$ is monotone increasing by \eqref{eq:ceps}, we have $|Du_\e|(0,L)=a_\e$, $\|u_\e\|_{L^\infty(0,L)}=a_\e$. Hence $(u_\e)_\e$ is bounded in $\BV$ and the second convergence follows by the compact embedding of $\BV$ into $L^1$. By semicontinuity of the total variation $|Du|(0,L)\leq\liminf_{\e}|Du_\e|(0,L)=\liminf_\e a_\e = a$. Finally by \eqref{eq:first-integral2}
$$
\e^2 (v_\e')^2 \leq \frac{(1-v_\e)^2}{4} + \e|d_\e| \leq \frac14 + \e\,\sup_{\e}|d_\e|,
$$
with $\sup_{\e}|d_\e|<+\infty$ by Proposition~\ref{prop:firstintegral}.
\end{proof}

As a consequence of Lemma~\ref{lemma:compactness} we remark for later use that the constant $c_0$ (see \eqref{eq:c0}) satisfies the bound
\begin{equation} \label{eq:bound-c0}
c_0 = \frac{1}{L}\int_0^L \lim_{\e\to0}c_\e\dd x
\xupref{eq:ceps}{\leq} \lim_{\e\to0}\frac{1}{L}\int_0^L f_\e^2(v_\e)u_\e'\dd x
\leq \lim_{\e\to0}\frac{1}{L}\int_0^L u_\e'\dd x
= \frac{a}{L}\,.
\end{equation}


\section{Proof of the convergence of critical points} \label{sect:proof1}

This section is entirely devoted to the proof of Theorem~\ref{thm:main1}. We assume along all this section that $(u_\e,v_\e)$ is a family of critical points of $\F_\e$ satisfying the assumptions of Theorem~\ref{thm:main1}. We recall that $(u_\e,v_\e)$ enjoys the regularity properties discussed in the previous section and that $v_\e(x)\in(0,1]$, see in particular Lemma~\ref{lemma:maxprin} and Proposition~\ref{prop:firstintegral}.

We denote by
\begin{equation} \label{def:min-max}
m_\e \defeq \min_{[0,L]}v_\e \in (0,1].
\end{equation}
In view of the symmetry properties of the critical points observed in Lemma~\ref{lemma:symmetry}, we have that $v_\e $ has a \emph{single-well shape}, that is, its global minimum $m_\e$ is achieved at the midpoint $\frac{L}{2}$, the graph of $v_\e$ is symmetric with respect to $\frac{L}{2}$, and $v_\e$ is decreasing in $(0,\frac{L}{2})$ and increasing in $(\frac{L}{2},L)$, achieving its maximum at the endpoints $v_\e(0)=v_\e(L)=1$.

We further assume that we have extracted a subsequence (not relabeled) such that $u_\e\to u$ and $v_\e\to1$ in $L^1(0,L)$ as in Lemma~\ref{lemma:compactness}, $c_\e\to c_0$ (see \eqref{eq:c0}), $d_\e\to d_0$ (see Proposition~\ref{prop:firstintegral}), and also
\begin{equation} \label{def:m0}
m_\e\to m_0\in[0,1] \qquad \text{as }\e\to0.
\end{equation}

For later use it is convenient to introduce the discrepancy
\begin{equation} \label{def:discrepancy}
\xi_\e(x) \defeq \frac{(1-v_\e(x))^2}{4\e} - \e(v_\e'(x))^2 \xupref{eq:first-integral2}{=} \frac{c_\e^2}{f_\e^2(v_\e(x))} + d_\e.
\end{equation}
By the second expression of $\xi_\e$ in \eqref{def:discrepancy} and monotonicity of $v_\e$, the minimum of the function $\xi_\e$ is attained at the maximum point of $v_\e$, that is $\min\xi_\e = \xi_\e(0) = - \e(v_\e'(0))^2 \leq0$, and similarly the maximum of $\xi_\e$ is attained at the midpoint, that is $\max\xi_\e = \xi_\e(\midp)=\frac{(1-m_\e)^2}{4\e}\geq0$. Hence there exists $y_\e\in[0,\midp]$ such that $\xi_\e(y_\e)=0$, $\xi_\e\leq0$ in $[0,y_\e]$ and $\xi_\e\geq0$ in $[y_\e,\midp]$. Up to subsequences we can assume $y_\e\to y_0\in[0,\midp]$. Notice that by evaluating \eqref{def:discrepancy} at the point $y_\e$ we find
\begin{equation} \label{eq:deps-negative}
d_\e = -\frac{c_\e^2}{f_\e^2(v_\e(y_\e))} < 0.
\end{equation}

In the following lemma we show an explicit relation between the limit values $c_0$ and $m_0$.

\begin{lemma} \label{lemma:c0}
Assume that $f_\e(m_\e)\to0$. Then
\begin{equation} \label{eq:c0-2}
c_0 = \frac12(1-m_0)f(m_0),
\end{equation}
where $c_0$ and $m_0$ are the limits in \eqref{eq:c0} and \eqref{def:m0} respectively, and the right-hand side of \eqref{eq:c0-2} must be interpreted as $\frac{\stress}{2}$ if $m_0=1$, in view of \ref{ass:f2}.
\end{lemma}

\begin{proof}
By evaluating \eqref{eq:first-integral2} at a minimum point of $v_\e$ we have
\begin{equation} \label{eq:first-int-min}
\frac{(1-m_\e)^2}{4\e}-\frac{c_\e^2}{f_\e^2(m_\e)} = d_\e.
\end{equation}
Since by assumption $f_\e(m_\e)\to0$ and recalling the definition \eqref{def:feps} of $f_\e$, it must be $f_\e(m_\e)=\sqrt{\e}f(m_\e)$ for $\e$ small. Then
\begin{equation*}
c_0^2 = \lim_{\e\to0} c_\e^2 = \lim_{\e\to0} \frac{1}{4\e}(1-m_\e)^2f_\e^2(m_\e) - d_\e f_\e^2(m_\e) = \frac14(1-m_0)^2f^2(m_0),
\end{equation*}
where we used the uniform bound on $d_\e$ in Proposition~\ref{prop:firstintegral}.
\end{proof}

In the following lemma, which is a consequence of the qualitative study of the equation \eqref{eq:cp1} for $v_\e$ contained in Section~\ref{sect:ODE}, we deduce some general properties of the sequence $(u_\e,v_\e)$.

\begin{lemma} \label{lemma:stima-tempi}
Assume that $m_\e < s_\e$, where $s_\e$ is as in \eqref{def:feps}. Let
\begin{equation*}
A_\e\defeq \{x\in(0,L)\,:\, v_\e(x)<s_\e\}.
\end{equation*}
Then $A_\e=(\frac{L}{2}-x_\e,\frac{L}{2} + x_\e)$ with $\lim_{\e\to0}x_\e=0$. If $m_0<1$, then $\lim_{\e\to0}\frac{x_\e}{\e}=+\infty$. 

Finally, $u\in\SBV(0,L)$ with $J_u\subset\{\frac{L}{2}\}$ and $u'=c_0$ almost everywhere in $(0,L)$.
\end{lemma}

\begin{proof}
Since $v_\e$ has a single-well shape, $v_\e(\frac{L}{2})=m_\e<s_\e$, and $v_\e=1$ at the endpoints of $(0,L)$, we have $A_\e=(\frac{L}{2}-x_\e,\frac{L}{2} + x_\e)$ for some $x_\e\in(0,\frac{L}{2})$. Notice that for $x\in A_\e$ we have $f_\e(v_\e(x))=\sqrt{\e}f(v_\e(x))$. The rescaled function $\tilde{v}_\e(t)=v_\e(\frac{L}{2}+\e t)$ satisfies
\begin{equation} \label{eq:blowup}
\begin{split}
& \tilde{v}_\e'' =\frac{1-\tilde{v}_\e}{4}\biggl[\frac{(2c_\e)^2f'(\tilde{v}_\e)}{(1-\tilde{v}_\e)f^3(\tilde{v}_\e)}-1\biggr] \qquad\qquad\text{for }t\in\Bigl(-\frac{x_\e}{\e},\frac{x_\e}{\e}\Bigr)\eqqcolon(-t_\e,t_\e), \\[5pt]
& \tilde{v}_\e(0) = m_\e, \quad \tilde{v}_\e'(0) = 0,
\end{split}
\end{equation}
that is, $\tilde{v}_\e$ is a solution in $(-t_\e,t_\e)$ of the Cauchy problem \eqref{eq:ODE1}--\eqref{eq:ODE3} studied in Section~\ref{sect:ODE}, for the values of the parameters $m\defeq m_\e\in(0,1)$ and $\alpha \defeq c_\e>0$.

We first consider the case $c_\e<\frac{\stress}{2}$, which is the assumption in Theorem~\ref{thm:ODE} and Proposition~\ref{prop:ODE}. Notice that necessarily $m_\e<z_{c_\e}$, where $z_{c_\e}\in(0,1)$ is defined by the equation \eqref{eq:salpha}: indeed, if it were $m_\e>z_{c_\e}$ then by the qualitative analysis of the ODE \eqref{eq:blowup} (see in particular Remark~\ref{rmk:mlarger}) the solution $\tilde{v}_\e$ would have a local maximum at the origin, which is incompatible with the single-well structure; if $m_\e=z_{c_\e}$ then $\tilde{v}_\e$ would be constant, which is again not possible.

Hence in the case $c_\e<\frac{\stress}{2}$ we have $m_\e\in(0,z_{c_\e})$ and we are in position to apply Proposition~\ref{prop:ODE} with $\eta\defeq s_\e$ in order to estimate the time $t_\e$ such that $\tilde{v}_\e(t_\e)=s_\e$. Notice that, in the case $(1-m_\e)f(m_\e)> 2c_\e$, the additional assumption $\eta<M$ (where $M$ is the maximum of the solution and then depends on $\eps$) is certainly satisfied, or else the function $\tilde{v}_\e$ would reach a maximum point before $t_\e$ and then decrease, which is incompatible with its single-well shape. 
Hence $t_\e$ obeys the bounds \eqref{eq:est-teta-1}--\eqref{eq:est-teta-2}. In particular, if $m_\e\to m_0\in[0,1)$ then $c_\e\to c_0\in[0,\frac{\stress}{2})$ by Lemma~\ref{lemma:c0} and therefore the estimate \eqref{eq:est-teta-1} holds with a constant which is uniformly bounded with respect to $\e$, namely
\begin{equation} \label{proof:est-teps-1}
t_\e \leq \frac{C_0}{\sqrt{1-s_\e}} \qquad\text{if }m_0\in [0,1).
\end{equation}
If, instead, $m_\e\to m_0=1$, then we can apply \eqref{eq:est-teta-2} to deduce
\begin{equation} \label{proof:est-teps-2}
t_\e \leq \frac{\omega(1-m_\e)}{(1-s_\e)^2} \qquad\text{if }m_0=1,
\end{equation}
where $\omega(\cdot)$ is a modulus of continuity independent of $\e$. Hence if $m_0\in[0,1)$ we find
\begin{equation*}
x_\e = \e t_\e \leq \frac{C_0 \e}{\sqrt{1-s_\e}} \xupref{def:seps}{\sim} \frac{C_0\e}{\e^\frac14\sqrt{\stress}} \longrightarrow 0 \qquad\text{as }\e\to0,
\end{equation*}
whereas if $m_0=1$
\begin{equation*}
x_\e = \e t_\e \leq \frac{\e\,\omega(1-m_\e)}{(1-s_\e)^2} \xupref{def:seps}{\sim} \frac{\omega(1-m_\e)}{\stress^2} \longrightarrow 0 \qquad\text{as }\e\to0,
\end{equation*}
proving that $x_\e\to0$ in any case.

Consider now the case $c_\eps\geq\frac{\stress}{2}$. Notice that in this case it must be $m_\e\to1$, or otherwise by Lemma~\ref{lemma:c0} $c_\e\to c_0=\frac12(1-m_0)f(m_0)<\frac{\stress}{2}$, which is not possible. We can then apply Proposition~\ref{prop:ODE2} to deduce that the estimate \eqref{proof:est-teps-2} continues to hold, and therefore $x_\e\to0$, as before.

Assume now that $m_0<1$ and let us show that $\frac{x_\e}{\e}\to+\infty$. As already observed it must be $c_\e<\frac{\stress}{2}$, and we can apply again Proposition~\ref{prop:ODE} with $\eta=s_\e$ to deduce by \eqref{eq:est-teta-3} that 
\begin{equation} \label{proof:est-teps3}
\frac{x_\e}{\e} =  t_\e \geq \log\biggl( \frac{1-z_{c_\e} +k_\e}{1-s_\e+k_\e}\biggr) 
\qquad 
\quad k_\e\defeq \bigg|1-\biggl(\frac{2c_\e}{(1-m_\e)f(m_\e)}\biggr)^2\bigg|^\frac{1}{2}.
\end{equation}
By \eqref{eq:c0-2} we have $k_\e\to0$ as $\e\to0$. Moreover $z_{c_\e}\to z_{c_0}$, where $c_0=\frac12(1-m_0)f(m_0)$ by \eqref{eq:c0-2}. Since $m_0<1$, we then have $c_0\in[0,\frac{\stress}{2})$ and in turn $z_{c_0}\in[0,1)$ by the properties of $\fbar$ in Proposition~\ref{prop:salpha}. Therefore by passing to the limit as $\e\to0$ in \eqref{proof:est-teps3} we obtain $\frac{x_\e}{\e}\to+\infty$ as $\e\to0$, which completes the proof of the first part of the statement.

For every fixed $\delta>0$ it holds $A_\e\subset(\frac{L}{2}-\delta,\frac{L}{2}+\delta)$ for all $\e$ sufficiently small. In particular for all $x\in(\frac{L}{2}-\delta,\frac{L}{2}+\delta)^c\defeq (0,L)\setminus(\frac{L}{2}-\delta,\frac{L}{2}+\delta)$ we have $v_\e\geq s_\e$ and, in turn, $f_\e(v_\e(x))\geq f_\e(s_\e)\to1$, that is, $f_\e(v_\e)$ converges uniformly to 1 on compact sets not containing $\frac{L}{2}$. Hence by \eqref{eq:ceps}
\begin{equation} \label{eq:conv-ueps}
u_\e' = \frac{c_\e}{f_\e^2(v_\e)} \longrightarrow c_0 \qquad\text{uniformly on }(\midp-\delta,\midp+\delta)^c, \text{ for all }\delta>0.
\end{equation}
We also have
\begin{equation*}
\Feps(u_\eps,v_\eps) \geq f_\e^2(s_\e)\int_{(\frac{L}{2}-\delta,\frac{L}{2}+\delta)^c} |u_\e'|^2\dd x, 
\end{equation*}
so that from the uniform bound on the energies \eqref{eq:bound-energy} and the convergence $f_\e(s_\e)\to1$, we have that $u_\e$ is uniformly bounded in $H^1((\frac{L}{2}-\delta,\frac{L}{2}+\delta)^c)$ for all $\delta>0$. We can conclude that $u\in H^1((\frac{L}{2}-\delta,\frac{L}{2}+\delta)^c)$ for all $\delta>0$. The properties in the statement are then immediate consequences of the previous facts.
\end{proof}


\subsection{Case I: pre-fractured regime} \label{subsec:cohesive}
We show that if $m_0\in(0,1)$ then the limit function $u$ is a piecewise affine critical point of the cohesive energy \eqref{def:F} with a single jump at $\frac{L}{2}$ and constant slope. This is summarized in the following proposition, which is the main result of this subsection.

\begin{proposition} \label{prop:case1b}
Assume that $m_0\in(0,1)$. Then $u\in\SBV(0,L)$, $u'=c_0\in(0,\frac{\stress}{2})$ almost everywhere in $(0,L)$, $J_u=\{\frac{L}{2}\}$, $[u](\frac{L}{2})\in(0,\sfrac)$, and
\begin{equation} \label{eq:case1b}
g'([u](\textstyle\frac{L}{2})) = 2c_0.
\end{equation}
Moreover $u$ attains the limit boundary conditions, that is, $|Du|(0,L)=c_0L+[u](\frac{L}{2})=a$. Finally, the convergence of the energies \eqref{eq:conv-energy} holds.
\end{proposition}

Notice that, under the assumptions of the proposition, we can apply Lemma~\ref{lemma:stima-tempi} since $m_\e<s_\e$ for $\e$ small enough, as $m_0\in(0,1)$ and $s_\e\to1$. We premise a lemma to the proof of the proposition.

\begin{lemma} \label{lemma:case1b}
Under the assumptions of Proposition~\ref{prop:case1b}, let $A_\e$ be as in Lemma~\ref{lemma:stima-tempi}. Then we have as $\e\to0$
\begin{equation*}
\Feps(u_\e,v_\e; A_\e^c) \defeq \int_{(0,L)\setminus A_\e}\Bigl( f_\e^2(v_\e) |u_\e'|^2 + \frac{(1-v_\e)^2}{4\e} + \e|v_\e'|^2 \Bigr) \dd x \to \int_{0}^L|u'|^2\dd x.
\end{equation*}
Moreover $c_0^2+d_0=0$.
\end{lemma}

\begin{proof}
By multiplying \eqref{eq:cp1} by the function $v_\e-1$ and integrating in $(0,\frac{L}{2}-x_\e)$ we have, after integration by parts,
\begin{align*}
\int_{0}^{\frac{L}{2}-x_\e} \biggl( \e|v_\e'|^2 + \frac{(1-v_\e)^2}{4\e}\biggr)\dd x 
& = \e (s_\e-1) v_\e'(\midp-x_\e) + \int_0^{\frac{L}{2}-x_\e} \frac{c_\e^2 f_\e'(v_\e)}{f_\e^3(v_\e)}(1-v_\e)\dd x \\
& \leq \e (1-s_\e)\| v_\e'\|_\infty + c_\e^2 \int_0^{\frac{L}{2}-x_\e} \frac{ \psi_\e'(v_\e)}{(1-v_\e)\psi_\e^3(v_\e)}(1-v_\e)^2\dd x 
\end{align*}
where in the second equality we used the fact that $f_\e(v_\e)=\psi_\e(v_\e)$ in $(0,\frac{L}{2}-x_\e)$ since $v_\e\geq s_\e$ in this interval. In view of the monotonicity properties of $\psi_\e$ in assumptions \ref{ass:psi1} and \ref{ass:psi3}, the previous estimate yields
\begin{align*}
\int_{0}^{\frac{L}{2}-x_\e} \biggl( \e|v_\e'|^2 + \frac{(1-v_\e)^2}{4\e}\biggr)\dd x 
& \leq \e (1-s_\e)\| v_\e'\|_\infty + \frac{ c_\e^2\psi_\e'(s_\e)}{(1-s_\e)\psi_\e^3(s_\e)}\int_0^{\frac{L}{2}-x_\e} (1-v_\e)^2\dd x.
\end{align*}
Now recalling \ref{ass:f2}, \eqref{def:seps}, and that $c_0<\frac{\stress}{2}$ (by \eqref{eq:c0-2} and the assumption $m_0\in(0,1)$)
\begin{equation*}
\lim_{\e\to0}\frac{4\e c_\e^2 \psi_\e'(s_\e)}{(1-s_\e)\psi_\e^3(s_\e)}
= \lim_{\e\to0} \frac{4\e\sqrt{\e}c_\e^2}{(1-s_\e)^3}\cdot\frac{(1-s_\e)^2 f'(s_\e)}{f_\e^3(s_\e)} = \frac{4c_0^2}{\stress^2}<1,
\end{equation*}
hence we have that there exists a constant $C>0$, independent of $\e$, such that for all $\e$ sufficiently small
$$
\frac{1}{4\e} - \frac{ c_\e^2\psi_\e'(s_\e)}{(1-s_\e)\psi_\e^3(s_\e)} \geq \frac{C}{4\e}\,.
$$
In turn we find
\begin{equation*}
\int_{0}^{\frac{L}{2}-x_\e} \biggl( \e|v_\e'|^2 + \frac{C(1-v_\e)^2}{4\e}\biggr)\dd x \leq \e (1-s_\e)\| v_\e'\|_\infty \to 0 \qquad\text{as }\e\to0
\end{equation*}
in view of the bound $\e\|v_\e'\|_\infty\leq 1$ for $\e$ small in Lemma~\ref{lemma:compactness}. By symmetry of $v_\e$ with respect to the midpoint $\frac{L}{2}$ we can conclude that
\begin{equation} \label{proof:c0d0}
\lim_{\e\to0}\int_{(0,L)\setminus A_\e} \biggl( \e|v_\e'|^2 + \frac{(1-v_\e)^2}{4\e}\biggr)\dd x = 0.
\end{equation}
In turn using \eqref{eq:conv-ueps}
\begin{equation*}
\lim_{\e\to0}\Feps(u_\e,v_\e; A_\e^c) = \lim_{\e\to0}\int_{(0,L)\setminus A_\e} f_\e^2(v_\e) |u_\e'|^2 \dd x = c_0^2L = \int_{0}^L |u'|^2\dd x
\end{equation*}
proving the first part of the statement.

To show that $c_0^2+d_0=0$, fix any $x_0\in(0,\frac{L}{2})$ and notice that $(0,x_0)\subset A_\e^c$ for $\e$ small, since $x_\e\to0$ by Lemma~\ref{lemma:stima-tempi}. Then, using the uniform convergence of $u_\e'$ to $c_0$ (see \eqref{eq:conv-ueps}) we find
\begin{align*}
|c_0^2 + d_0|x_0
= \lim_{\e\to0} \int_0^{x_0}| c_\e u_\e' + d_\e |\dd x
\xupref{eq:first-int}{=} \lim_{\e\to0} \int_0^{x_0}\bigg| \frac{(1-v_\e)^2}{4\e} -\e (v_\e')^2 \bigg|\dd x
\xupref{proof:c0d0}{=}0,
\end{align*}
which completes the proof of the lemma.
\end{proof}

We are now ready to give the proof of Proposition~\ref{prop:case1b}.

\begin{proof}[Proof of Proposition~\ref{prop:case1b}.]
We first prove that the limit function $u$ satisfies $|Du|(0,L)=a$. Fix $\delta>0$ such that $u_\e(\frac{L}{2}\pm\delta)\to u(\frac{L}{2}\pm\delta)$ and denote by $I_\delta\defeq(0,L)\setminus(\frac{L}{2}-\delta,\frac{L}{2}+\delta)$. We have $I_\delta\subset A_\e^c$ for $\e$ small enough, since $x_\e\to0$ by Lemma~\ref{lemma:stima-tempi}; hence by \eqref{proof:c0d0}
\begin{align*}
0
& = \lim_{\e\to0}\int_{I_\delta} \Bigl( \frac{(1-v_\e)^2}{4\e} - \e|v_\e'|^2 \Bigr) \dd x \\
& = \lim_{\e\to0} \int_{I_\delta}\bigl(c_\e u_\e'+d_\e\bigr)\dd x \\
& = c_0 \lim_{\e\to0}\bigl(u_\e(L)-u_\e(\midp+\delta) + u_\e(\midp-\delta)-u_\e(0)\bigr) + d_0(L-2\delta) \\
& = c_0 \bigl(a-u(\midp+\delta) + u(\midp-\delta)\bigr) - c_0^2\bigl(L-2\delta\bigr),
\end{align*}
where the second equality follows by \eqref{eq:first-int} and \eqref{eq:ceps}, and the last one by Lemma~\ref{lemma:case1b}. Hence by letting $\delta\to0$  we find $a = c_0L + [u](\frac{L}{2})=\int_0^Lu'\dd x + [u](\frac{L}{2})=|Du|(0,L)$.

To conclude the proof, it remains to show the criticality identity \eqref{eq:case1b} and the convergence of the energies \eqref{eq:conv-energy}.
We consider a blow-up of the functions $u_\e$ and $v_\e$ around the midpoint: let $\tilde{u}_\e(t)\defeq u_\e(\frac{L}{2}+\e t)$, $\tilde{v}_\e(t)\defeq v_\e(\frac{L}{2}+\e t)$ for $t\in(-\frac{L}{2\e},\frac{L}{2\e})$. The idea of the proof is to show that the pair $(\tilde{u}_\e,\tilde{v}_\e)$ converges to an optimal pair $(\alpha_{s_0},\beta_{s_0})$ for the minimum problem \eqref{def:g} which defines $g(s_0)$, with $s_0=[u](\frac{L}{2})$. We refer to Section~\ref{sect:g}, and in particular to Proposition~\ref{prop:g2a} and Proposition~\ref{prop:g2b}, for the existence and the main properties of optimal pairs for $g$.

We first remark that for all $T>0$ and all $\e$ sufficiently small
\begin{equation} \label{eq:bound-blowup}
\int_{-T}^T \Bigl( \frac{(1-\tilde{v}_\e)^2}{4} + |\tilde{v}_\e'|^2 \Bigr) \dd t = \int_{\frac{L}{2}-\e T}^{\frac{L}{2}+\e T} \Bigl(\frac{(1-v_\e)^2}{4\e} + \e|v_\e'|^2 \Bigr) \dd x \leq C,
\end{equation}
for a constant $C>0$ independent of $\e$ and $T$, by the uniform bound \eqref{eq:bound-energy}. Therefore, up to extracting a subsequence, we have that $\tilde{v}_\e\wto \tilde{v}$ weakly in $H^1_\loc(\R)$, for some function $\tilde{v}$ with $1-\tilde{v}\in H^1(\R)$, and the convergence is also uniform on compact sets.

Furthermore, the function $\tilde{v}_\e$ solves the initial value problem \eqref{eq:blowup} in $(-\frac{x_\e}{\e},\frac{x_\e}{\e})$, where $t_\e\defeq\frac{x_\e}{\e}\to+\infty$ as $\e\to0$ by Lemma~\ref{lemma:stima-tempi}. For every $T>0$ and every test function $\vphi\in\mathrm{C}^\infty_{\mathrm c}(-T,T)$, since $(-T,T)\subset(-\frac{x_\e}{\e},\frac{x_\e}{\e})$ for all $\e$ small enough we can pass to the limit in the weak formulation of \eqref{eq:blowup}:
\begin{equation*}
0 = \int_{-T}^T \biggl( \tilde{v}_\e'\vphi' + \frac{c_\e^2 f'(\tilde{v}_\e)}{f^3(\tilde{v}_\e)}\vphi + \frac{\tilde{v}_\e-1}{4}\vphi\biggr)\dd t
\longrightarrow
\int_{-T}^T \biggl( \tilde{v}'\vphi' + \frac{c_0^2 f'(\tilde{v})}{f^3(\tilde{v})}\vphi + \frac{\tilde{v}-1}{4}\vphi\biggr)\dd t
\end{equation*}
where the convergence is justified since $\tilde{v}_\e\wto \tilde{v}$ weakly in $H^1(-T,T)$ and uniformly on $[-T,T]$, and $\tilde{v}\geq\tilde{v}(0)=m_0>0$. In conclusion, we obtained that the limit function $\tilde{v}$ is a weak solution to the equation
\begin{equation} \label{proof:vtilde}
\begin{split}
& \tilde{v}'' = \frac{c_0^2f'(\tilde{v})}{f^3(\tilde{v})}+\frac{\tilde{v}-1}{4} \qquad\text{in }\R, \\[5pt]
& \tilde{v}(0) = m_0, \quad \tilde{v}'(0) = 0.
\end{split}
\end{equation}
Notice that the right-hand side of \eqref{proof:vtilde} is a continuous function, and therefore $\tilde{v}\in\mathrm{C}^2(\R)$ and the equation holds in the classical sense; moreover $\tilde{v}'(0)=0$ since the origin is a minimum point of $\tilde{v}$.

Consider now the map $s\mapsto \mbar_s$ defined in Proposition~\ref{prop:g2b}, which associates to every $s\in[0,+\infty)$ the minimum value of the optimal profile $\beta_s$ for $g(s)$. This map is a continuous bijection between $(0,\sfrac)$ and $(0,1)$: in particular, since $m_0\in(0,1)$, we have that there exists $s_0\in(0,\sfrac)$ such that $m_0=\mbar_{s_0}$. By \eqref{eq:opt-prof-1} and \eqref{eq:opt-prof-2} the optimal profile $\beta_{s_0}$ for $g(s_0)$ solves
\begin{equation} \label{proof:betas}
\begin{split}
& \beta_{s_0}'' = \frac{\sigma_{s_0}^2f'(\beta_{s_0})}{f^3(\beta_{s_0})}+\frac{\beta_{s_0}-1}{4} \qquad\text{in }\R, \\[5pt]
& \beta_{s_0}(0) = \mbar_{s_0}, \quad \beta_{s_0}'(0) = 0,
\end{split}
\end{equation}
where by \eqref{eq:ms-1} and \eqref{eq:c0-2} the constant $\sigma_{s_0}$ is given by
\begin{equation} \label{proof:c0}
\sigma_{s_0}=\frac12 (1-\mbar_{s_0}) f(\mbar_{s_0}) = \frac12(1-m_0)f(m_0) = c_0.
\end{equation}
Therefore by comparing \eqref{proof:vtilde} and \eqref{proof:betas} we conclude, by uniqueness, that it must be
\begin{equation} \label{proof:vtilde-betas}
\tilde{v}=\beta_{s_0}.
\end{equation}

In order to obtain the criticality condition \eqref{eq:case1b}, it is now sufficient to show that $s_0=[u](\frac{L}{2})$: indeed in this case we would have by Proposition~\ref{prop:g3}
$$
g'([u]({\textstyle\frac{L}{2}})) = g'(s_0) = (1-\mbar_{s_0})f(\mbar_{s_0}) \xupref{proof:c0}{=} 2c_0.
$$
Therefore we now prove that $s_0=[u](\frac{L}{2})$.

By using the properties of the optimal pair $(\alpha_{s_0},\beta_{s_0})$ in Proposition~\ref{prop:g2a}, we have
\begin{align*}
s_0
& = \int_{\R}\alpha_{s_0}'\dd t
\xupref{eq:opt-prof-2}{=} \int_{\R}\frac{\sigma_{s_0}}{f^2(\beta_{s_0})}\dd t
\xupref{proof:vtilde-betas}{=} \int_{\R}\frac{c_0 }{f^2(\tilde{v})} \dd t \\
& = \sup_{T>0} \int_{-T}^T \frac{c_0}{f^2(\tilde{v})}\dd t
= \sup_{T>0} \, \lim_{\e\to0}\int_{-T}^{T} \frac{c_\e}{f^2(\tilde{v}_\e)}\dd t
= \sup_{T>0} \, \lim_{\e\to0}\int_{\frac{L}{2}-\e T}^{\frac{L}{2}+\e T} \frac{c_\e}{\e f^2(v_\e)}\dd x \\
& \xupref{eq:ceps}{=} \sup_{T>0} \, \lim_{\e\to0}\int_{\frac{L}{2}-\e T}^{\frac{L}{2}+\e T} u_\e' \dd x
= \sup_{T>0} \, \lim_{\e\to0} \bigl( u_\e(\midp+\e T) - u_\e(\midp-\e T) \bigr)
\leq u(\midp+\delta) - u(\midp-\delta)
\end{align*}
for every $\delta>0$ such that $u_\e({\textstyle\frac{L}{2}}\pm\delta)\to u({\textstyle\frac{L}{2}}\pm\delta)$, since $u_\e$ is monotone nondecreasing. By letting $\delta\to0$ we obtain $s_0\leq [u](\frac{L}{2})$.

We next show the opposite inequality $s_0\geq [u](\frac{L}{2})$. For $\delta>0$ as above we have
\begin{align}\label{[u]<s0}	
u(\midp+\delta) - u(\midp-\delta)&=\lim_{\e\to0}(u_\e(\midp+\delta)-u_\e(\midp-\delta))=	
\lim_{\e\to0}\int_{\midp-\delta}^{\midp+\delta}u'_\e \dd x\nonumber\\
&=	\lim_{\e\to0}\int_{\midp-\delta}^{\midp+\delta}\frac{c_\e}{f^2_\e(v_\e)} \dd x\leq\lim_{\e\to0}\int_{\midp-x_\e}^{\midp+x_\e}\frac{c_\e}{f^2_\e(v_\e)} \dd x+2\delta c_0\nonumber\\
&= \lim_{\e\to0}2\int_{0}^{\frac{x_\e}{\e}}\frac{c_\e}{f^2(\tilde v_\e)} \dd t+2\delta c_0=\lim_{\e\to0}2c_0\int_{m_\e}^{s_\e}\frac{\dd s}{f^2(s)\sqrt{\Psi_\e(s)}} +2\delta c_0\nonumber\\
&\leq 2c_0\int_{m_0}^{1}\frac{\dd s}{f^2(s)\sqrt{\Psi_0(s)}} +2\delta c_0
\xupref{eq:ms-3}{=}s_0+2\delta c_0,
\end{align}	
where in the second line we have used \eqref{eq:ceps} and $f_\e(v_\e)\geq f_\e(s_\e)$ in $(\midp-x_\e,\midp+x_\e)^c$, with $f_\e(s_\e)\to1$ and $x_\e\to0$ by Lemma~\ref{lemma:stima-tempi}; in the third line we have used
\[(\tilde v'_\e)^2=\Psi_\e(\tilde v_\e)\]
with
\[\Psi_\e(s)\coloneqq\frac{1}{4f^2(s)}\biggl((1-s)^2f^2(s)-(1-m_\e)^2f^2(m_\e)+4\e |d_\e|(f^2(s)-f^2(m_\e))\biggr),\quad s\in(m_0,1),\]
which comes from \eqref{eq:first-integral2}, \eqref{eq:first-int-min}, and \eqref{eq:deps-negative}; and in the fourth line we have set
\[\Psi_0(s)\coloneqq\frac{1}{4f^2(s)}\biggl((1-s)^2f^2(s)-(1-m_0)^2f^2(m_0))\biggr),\quad s\in(m_0,1).\]
As $\delta\to 0$, we get $[u](\midp)\leq s_0$.

Hence $s_0=[u](\frac{L}{2})$ which in turn yields, as we have seen before, that \eqref{eq:case1b} holds. The only missing point to complete the proof of Proposition~\ref{prop:case1b} is the convergence of the energies \eqref{eq:conv-energy}. However, this follows immediately by combining Lemma~\ref{lemma:case1b} with the computation below, which is based on the same arguments used in \eqref{[u]<s0}:
\begin{align*}
& \lim_{\e\to 0}\int_{A_\e} \Bigl( f_\e^2(v_\e) |u_\e'|^2 + \frac{(1-v_\e)^2}{4\e} + \e|v_\e'|^2 \Bigr) \dd x
 \xupref{eq:first-int}{=} \lim_{\e\to0} \biggl( 2\int_{A_\e}\frac{(1-v_\e)^2}{4\e}\dd x - \int_{A_\e}d_\e\dd x \biggr) \nonumber\\
& = \lim_{\e\to0} \; 2\int_{-\frac{x_\e}{\e}}^{\frac{x_\e}{\e}}\frac{(1-\tilde{v}_\e)^2}{4}\dd t
 = \lim_{\e\to0} \; \int_{m_\e}^{s_\e}\frac{(1-s)^2}{\sqrt{\Psi_\e(s)}}\dd s
 = \int_{m_0}^{1}\frac{(1-s)^2}{\sqrt{\Psi_0(s)}}\dd s 
 \xupref{eq:ms-2}{=} g(s_0)=g([u](\midp)),
\end{align*}
where in the fourth equality we have used dominated convergence. This is allowed since, setting $\tilde f(s)\defeq(1-s)f(s)$, which is monotone increasing by assumption \ref{ass:f3}, we have
$$
\frac{(1-s)^2}{\sqrt{\Psi_\e(s)}}
\leq \frac{2\tilde{f}(s)(1-s)}{\sqrt{\tilde{f}^2(s)-\tilde{f}^2(m_\e)}}
\leq \frac{2\tilde{f}(s)(1-s)}{\sqrt{2\tilde{f}(m_\e)}\sqrt{\tilde{f}'(\zeta_\e(s))(s-m_\e)}} \,,
$$
for all $s\in(m_\e,s_\e)$ and some $\zeta_\e(s)\in(m_\e,s)$. Since $\zeta_\e(s)\geq m_\e$ and $\inf_{\e}m_\e>0$, in view of assumptions \ref{ass:f3} and \ref{ass:f5} we have that $\tilde{f}'(\zeta_\e(s)) \geq C(1-\zeta_\e(s))^3 \geq C(1-s)^3$ for all $s\in(m_\e,s_\e)$ and for a constant $C>0$ independent of $\e$. Therefore we find for another constant $C_1>0$ independent of $\e$
$$
\frac{(1-s)^2}{\sqrt{\Psi_\e(s)}} \leq \frac{C_1}{\sqrt{1-s}\sqrt{s-m_\e}}
$$
which allows to apply the (generalized) dominated convergence theorem, as required. This concludes the proof.
\end{proof}


\subsection{Case II: complete fracture} \label{subsec:fracture}
We next show that if $m_0=0$ then the limit function $u$ is a critical point of the cohesive energy \eqref{def:F} describing a completely fractured state, namely $u$ has a single jump at $\frac{L}{2}$ and is constant elsewhere.

\begin{proposition} \label{prop:case1a}
Assume that $m_0=0$. Then necessarily $\sfrac\in\R$ and $a=\sfrac$, and $u(x)=a\chi_{(\frac{L}{2},L)}(x)$. Furthermore the convergence of the energies \eqref{eq:conv-energy} holds.
\end{proposition}

\begin{proof}
The proof is in part similar to that of Proposition~\ref{prop:case1b}. By Lemma~\ref{lemma:c0} we have $c_0=0$. We can also apply Lemma~\ref{lemma:stima-tempi} to deduce that $A_\e\defeq\{v_\e<s_\e\}=(\midp-x_\e,\midp+x_\e)$ with $x_\e\to0$ and $\frac{x_\e}{\e}\to+\infty$. Moreover, by the same lemma we have $u\in\SBV(0,L)$ with $J_u\subset\{\midp\}$ and $u'=c_0=0$ almost everywhere in $(0,L)$, with uniform convergence $u_\e'\to 0$ on compact sets not containing $\midp$ by \eqref{eq:conv-ueps}. We can also repeat word by word the proof of Lemma~\ref{lemma:case1b}, so that
\begin{equation} \label{proof:case1a-1}
\lim_{\e\to0}\Feps(u_\e,v_\e;A_\e^c) = 0, \qquad c_0^2+d_0=0
\end{equation}
(hence $d_0=0$).

For any $x\in(0,\midp)$ we have $u_\e(x)=\int_0^x u_\e'\dd t \to c_0x=0$, whereas for $x\in(\midp,L)$ we have $u_\e(x)=u_\e(L)-\int_x^L u_\e'\dd t \to a - c_0(L-x)=a$. Therefore the limit function is $u(x)=a\chi_{(\midp,L)}(x)$ with a jump at the midpoint of amplitude $[u](\midp)=a$.

We now claim that 
\begin{equation} \label{proof:case1a-2}
\lim_{\e\to0}\Feps(u_\e,v_\e) = 1.
\end{equation}
We first observe that by \eqref{eq:ceps} and the Dirichlet boundary condition \eqref{def:cp3}
\begin{equation} \label{proof:case1a-3}
\lim_{\e\to0}\int_0^L f_\e^2(v_\e)(u_\e')^2\dd x = \lim_{\e\to0} c_\e\int_0^L u_\e'\dd x = \lim_{\e\to0} c_\e a_\e = 0.
\end{equation}
By \eqref{eq:first-int} we have
\begin{equation*}
\lim_{\e\to0}\int_0^L \bigg| \frac{(1-v_\e)^2}{4\e} - \e(v_\e')^2\bigg| \dd x
= \lim_{\e\to0}\int_0^L \big| c_\e u_\e' +d_\e\big|\dd x
\leq \lim_{\e\to0} \bigl( c_\e a_\e + |d_\e|L \bigr) = 0
\end{equation*}
(where we used in particular that $c_0=0$, $d_0=0$), and in turn it follows that
\begin{align*}
\lim_{\e\to0}\int_0^L \biggl( \frac{(1-v_\e)^2}{4\e} +\e (v_\e')^2 - (1-v_\e)|v_\e'|\biggr)\dd x 
& = \lim_{\e\to0} \int_0^L \biggl( \frac{1-v_\e}{2\sqrt{\e}} - \sqrt{\e}|v_\e'| \biggr)^2 \dd x \\
& \leq \lim_{\e\to0}\int_0^L \bigg| \frac{(1-v_\e)^2}{4\e} -\e (v_\e')^2\bigg|\dd x = 0.
\end{align*}
Therefore, combining this equation and \eqref{proof:case1a-3}, we find
\begin{equation*}
\begin{split}
\lim_{\e\to0}\Feps(u_\e,v_\e)
& = \lim_{\e\to 0} \int_0^L f_\e^2(v_\e)(u_\e')^2 \dd x + \int_0^L \bigg( \frac{(1-v_\e)^2}{4\e} +\e (v_\e')^2\bigg) \dd x \\
& = \lim_{\e\to0} \int_0^L (1-v_\e)|v_\e|'\dd x 
= \lim_{\e\to0} 2 \int_0^{\midp} (1-v_\e)(-v_\e')\dd x \\
& = \lim_{\e\to0} (1-v_\e(\midp))^2 = \lim_{\e\to0} (1-m_\e)^2 = 1,
\end{split}
\end{equation*}
which proves the claim \eqref{proof:case1a-2}.

To conclude the proof, we need to show that
\begin{equation} \label{proof:case1a-4}
[u](\midp) = \sfrac
\end{equation}
(and, in particular, that $\sfrac$ is finite).
Indeed, recalling that $u$ is piecewise constant, in this case we have $\F(u,1)=g([u](\midp))=1$, and therefore \eqref{proof:case1a-2} gives also the convergence of the energy. The rest of the proof is therefore devoted to showing \eqref{proof:case1a-4}. 

Let $\tilde u_\e(t)\defeq u_\e(\midp+\e t)$, $\tilde v_\e(t) \defeq v_\e(\midp+\e t)$, for $t\in (-\frac{x_\e}{\e},\frac{x_\e}{\e})$. We first check that $\tilde v_\e(t)\rightharpoonup 1-e^{-|t|/2}$ weakly in $H^1_{\loc}(\R)$ and uniformly on compact sets. Indeed, as in \eqref{eq:bound-blowup} we have that $(1-\tilde{v}_\e)$ is uniformly bounded in $H^1(-T,T)$ for all fixed $T>0$, so that $\tilde v_\e$ converges weakly in $H^1_{\loc}(\R)$ and uniformly on compact sets to some function $\tilde v$ as $\e\to0$, with $1-\tilde v\in H^1(\R)$. Also, $\tilde v(0)=\lim_\e \tilde v_\e(0)=\lim_\e m_\e=0$.
We now check that $\{\tilde v=0\}=\{0\}$. By the properties of $v_\e$, we have $\{\tilde v=0\}=[-\tilde x,\tilde x]$, for some $\tilde x\geq 0$. Recalling \eqref{proof:case1a-2} and that $\frac{x_\e}{\e}\to+\infty$ by Lemma~\ref{lemma:stima-tempi}, we have
\begin{equation*}
1=\lim_{\e\to0}\Feps(u_\e,v_\e) \geq
\liminf_{\e\to0}\int_{-\frac{x_\e}{\e}}^{\frac{x_\e}{\e}}\biggl(\frac{(1-\tilde v_\e)^2}{4} + (\tilde v_\e')^2\biggr) \dd t
\geq \int_{-T}^T\biggl( \frac{(1-\tilde v)^2}{4} +(\tilde v')^2\biggr) \dd t
\end{equation*}
for all $T>0$, so that
\begin{equation*}
1 \geq \int_{\R}\bigg( \frac{(1-\tilde v)^2}{4} +(\tilde v')^2\bigg) \dd t
\geq \frac{\tilde x}{2} + \int_{\R\setminus (-\tilde x,\tilde x)} \bigg(\frac{(1-\tilde v)^2}{4}+(\tilde v')^2\bigg) \dd t\geq \frac{\tilde x}{2}+1.
\end{equation*}
This implies $\tilde x=0$ and in turn $\{\tilde v=0\}=\{0\}$.

By writing in weak form the equation \eqref{eq:blowup} satisfied by $\tilde v_\e$ and passing to the limit as $\e\to0$ we get for all $\varphi\in C^\infty_{\mathrm{c}}(\R\setminus \{0\})$
\begin{equation*}
0=\int_{\mathrm{supp}(\varphi)}\bigg(\tilde v_\e'\varphi' + \frac{c_\e^2f'(\tilde v_\e)}{f^3(\tilde v_\e)}\varphi+\frac{\tilde v_\e-1}{4}\varphi\bigg)\dd t
\to\int_{\text{supp}(\varphi)}\bigg(\tilde v'\varphi'+\frac{\tilde v-1}{4}\varphi\bigg)\dd t,
\end{equation*}
where we used the weak $L^2$-convergence of $\tilde{v}_\e'$ for the first term and the uniform convergence for the second and the third term.
Together with $\tilde v(0)=0$ and $\tilde v(+\infty)=1$, this implies $\tilde v(t)=1-e^{-|t|/2}$.

We further recall that by changing variables in \eqref{eq:first-integral2}, using \eqref{eq:first-int-min} and since $d_\e<0$ by \eqref{eq:deps-negative}, we have in $(-\frac{x_\e}{\e},\frac{x_\e}{\e})$
\begin{equation}\label{eqnvPsi}
\begin{split}
(\tilde v_\e')^2
& = \frac{1}{4f^2(\tilde v_\e)} \biggl((1-\tilde v_\e)^2f^2(\tilde v_\e)-4c_\e^2-4\e d_\e  f^2(\tilde v_\e) \biggr) \\
& =\frac{1}{4f^2(\tilde v_\e)}\biggl((1-\tilde v_\e)^2f^2(\tilde v_\e)-(1-m_\e)^2f^2(m_\e)+4\e |d_\e|(f^2(\tilde v_\e)-f^2(m_\e))\biggr)\eqqcolon\Psi_\e(\tilde v_\e).
\end{split}
\end{equation}

Let us now compute $[u](\midp)$. We have
\begin{equation} \label{eq:jump}
\begin{split}
[u](\midp)
& = a =\lim_{\e\to0}\int_0^L u_\e'\dd x
\xupref{eq:ceps}{=} \lim_{\e\to0} 2\int_{\frac{L}{2}}^L \frac{c_\e}{f_\e^2(v_\e)}\dd x
= \lim_{\e\to0} 2\int^{\frac{L}{2}+x_\e}_{\frac{L}{2}} \frac{c_\e}{f_\e^2(v_\e)}\dd x\\
&= \lim_{\e\to0} 2\int_0^{\frac{x_\e}{\e}} \frac{c_\e}{f^2(\tilde v_\e)}\dd t
\xupref{eqnvPsi}{=} \lim_{\e\to0} 2c_\e\int_0^{\frac{x_\e}{\e}} \frac{\tilde v'_\e}{f^2(\tilde v_\e)\sqrt{\Psi_\e(\tilde v_\e)}}\dd t\\
&  = \lim_{\e\to0}2c_\e\int_{m_\e}^{s_\e} \frac{1}{f^2(s)\sqrt{\Psi_\e(s)}}\dd s,	
\end{split}
\end{equation}	
where in the fourth equality we have used that $f_\e(v_\e)\geq f_\e(v_\e(L/2+x_\e))\to 1$ in $A_\e^c$ and that $c_\e\to0$, and the last passage follows by a change of variables. Fix now any $\delta\in(0,1)$. Since $m_\e < \delta < s_\e$ for $\e$ small enough, we have by the definition of $\Psi_\e$ that
\begin{equation*}
\begin{split}
0 \leq \lim_{\e\to0} 2c_\e \int_{\delta}^{s_\e} \frac{\dd s}{f^2(s)\sqrt{\Psi_\e(s)}}
& \leq \lim_{\e\to0} 4c_\e \int_{\delta}^{s_\e} \frac{\dd s}{f(s)\sqrt{(1-s)^2f^2(s)-(1-m_\e)^2f^2(m_\e)}} \\
& \leq \lim_{\e\to0} \frac{4c_\e(s_\e-\delta)}{f(\delta)\sqrt{(1-\delta)^2f^2(\delta)-(1-m_\e)^2f^2(m_\e)}} = 0
\end{split}
\end{equation*}
where we used the monotonicity of the map $s\mapsto(1-s)f(s)$ given by assumption \ref{ass:f3}.	
Therefore by \eqref{eq:jump} we see that for all $\delta\in(0,1)$
\begin{equation} \label{eq:jump2}
[u](\midp) = a = \lim_{\e\to0}2c_\e\int_{m_\e}^{\delta} \frac{1}{f^2(s)\sqrt{\Psi_\e(s)}}\dd s.
\end{equation}
By inserting the definition of $\Psi_\e$, we find
\begin{align*}
[u](\midp)
&= \lim_{\e\to0}4c_\e\int_{m_\e}^{\delta} \frac{\dd s}{f(s)\sqrt{(1-s)^2f^2(s)-(1-m_\e)^2f^2(m_\e)+4\e |d_\e|(f^2(s)-f^2(m_\e))}}\\
&\geq \lim_{\e\to0}\frac{4c_\e}{\sqrt{(1-m_\e)^2+4\e |d_\e|}}\int_{m_\e}^{\delta} \frac{\dd s}{f(s)\sqrt{f^2(s)-f^2(m_\e)}}\\
&\geq \lim_{\e\to0}2f(m_\e)\inf_{s\in(m_\e,\delta)}\bigg(\frac{1}{f'(s)}\bigg) \int_{m_\e}^{\delta} \frac{f'(s)}{f(s)\sqrt{f^2(s)-f^2(m_\e)}}\dd s\\ 
&= \lim_{\e\to0}2f(m_\e)\inf_{s\in(m_\e,\delta)}\bigg(\frac{1}{f'(s)}\bigg) \int_{f(m_\e)}^{f(\delta)} \frac{\dd t}{t\sqrt{t^2-f^2(m_\e)}} \\ 
&= \lim_{\e\to0}2\inf_{s\in(m_\e,\delta)}\bigg(\frac{1}{f'(s)}\bigg) \arctan\bigg(\frac{\sqrt{f^2(\delta)-f^2(m_\e)}}{f(m_\e)}\bigg)
=\pi\inf_{s\in(0,\delta)}\bigg(\frac{1}{f'(s)}\bigg).
\end{align*}
Similarly, again by \eqref{eq:jump2} and using the definition of $\Psi_\e$, and denoting by $\tilde{f}(s)\defeq(1-s)f(s)$, we have
\begin{align*}
[u](\midp)
 &= \lim_{\e\to0}4c_\e\int_{m_\e}^{\delta} \frac{\dd s}{f(s)\sqrt{(1-s)^2f^2(s)-(1-m_\e)^2f^2(m_\e)+4\e |d_\e|(f^2(s)-f^2(m_\e))}}\\
 &\leq \lim_{\e\to0}4c_\e \int_{m_\e}^{\delta} \frac{\dd s}{f(s)\sqrt{(1-s)^2f^2(s)-(1-m_\e)^2f^2(m_\e)}}\\
 & = \lim_{\e\to0} 4c_\e \int_{m_\e}^{\delta} \frac{\tilde{f}'(s)}{\tilde{f}(s)\sqrt{\tilde{f}^2(s)-\tilde{f}^2(m_\e)}}\cdot\frac{(1-s)}{\tilde{f}'(s)}\dd s\\
 & \leq \lim_{\e\to0} 4c_\e \sup_{s\in(m_\e,\delta)}\bigg(\frac{1-s}{\tilde{f}'(s)}\bigg) \int_{\tilde{f}(m_\e)}^{\tilde{f}(\delta)} \frac{\dd t}{t\sqrt{t^2-\tilde{f}^2(m_\e)}} \\ 
 & = \lim_{\e\to0} 4c_\e \sup_{s\in(m_\e,\delta)}\bigg(\frac{1-s}{\tilde{f}'(s)}\bigg) \frac{1}{\tilde{f}(m_\e)}\arctan\bigg(\frac{\sqrt{\tilde{f}^2(\delta)-\tilde{f}^2(m_\e)}}{\tilde{f}(m_\e)}\bigg)
= \pi \sup_{s\in(0,\delta)}\bigg(\frac{1-s}{\tilde{f}'(s)}\bigg)
\end{align*}
where we used the fact that $\frac{4c_\e}{\tilde{f}(m_\e)}\to 2$ by \eqref{eq:first-int-min}.
By collecting the previous inequalities we conclude that for all $\delta\in(0,1)$
$$
\pi\inf_{s\in(0,\delta)}\bigg(\frac{1}{f'(s)}\bigg) \leq [u](\midp) \leq \pi \sup_{s\in(0,\delta)}\bigg(\frac{1-s}{\tilde{f}'(s)}\bigg)
$$
so that by letting $\delta\to0$ and recalling Proposition~\ref{prop:sfrac}, we get $[u](\midp)=\frac{\pi}{f'(0)}=\sfrac$.
With a small abuse of notation, the previous computation says that if $f'(0)=0$, then $[u](\midp)=+\infty$, which is a contradiction with $[u](\midp)=a<+\infty$. Hence, necessarily $\sfrac$ is finite and \eqref{proof:case1a-4} holds.
\end{proof}


\subsection{Case III: elastic regime} \label{subsec:elastic}
We eventually consider the case $m_0=1$. The limit behaviour of the family $(u_\e,v_\e)$ is an elastic critical point, as described by the following proposition.

\begin{proposition} \label{prop:case1c}
Assume that $m_0=1$. Then $u(x)=\frac{a}{L}x$ and $c_0=\frac{a}{L}$. Moreover the convergence of the energies \eqref{eq:conv-energy} holds if and only if $\frac{a}{L}\leq\frac{\stress}{2}$.
\end{proposition}

\begin{proof}
We distinguish three cases, depending on whether the minimum value $m_\e$ of $v_\e$ is above or below the threshold $s_\e$ (see \eqref{def:feps}) and whether $f_\e(m_\e)$ converges to zero or not.

\medskip\noindent\textit{Step 1: $m_\e\geq s_\e$.}
In this case $f_\e(v_\e(x))\geq f_\e(m_\e)=\psi_\e(m_\e)\to1$, that is, $f_\e(v_\e)$ converges to 1 uniformly in $[0,L]$. In turn $u_\e'\to c_0$ uniformly in $[0,L]$ by \eqref{eq:ceps} and $u'\equiv c_0$. In view of the boundary conditions \eqref{def:cp3}
\begin{equation*}
a = \lim_{\e\to0} a_\e = \lim_{\e\to0}\int_0^L u_\e'\dd x = c_0L,
\end{equation*}
therefore $c_0=\frac{a}{L}$ and $u(x)=\frac{a}{L}x$. To complete the proof in this case, it only remains to show that the convergence of the energy holds if and only if $c_0\leq\frac{\stress}{2}$.

If $c_0>\frac{\stress}{2}$, then
\begin{equation*}
\Feps(u_\e,v_\e) \geq \int_0^L f_\e^2(v_\e)|u_\e'|^2\dd x \stackrel{\e\to0}{\longrightarrow} c_0^2 L > \Bigl(\stress c_0 - \frac{\stress^2}{4}\Bigr)L = \int_0^L \Bigl( \stress u' - \frac{\stress^2}{4}\Bigr)\dd x = \F(u,1).
\end{equation*}
Conversely, assume that $c_0\leq\frac{\stress}{2}$. Recalling the definition of the discrepancy $\xi_\e$ in \eqref{def:discrepancy} and evaluating it at the point $y_\e$ where $\xi_\e(y_\e)=0$, by the uniform convergence $f_\e(v_\e)\to1$ we have
\begin{equation*}
c_0^2 + d_0 = \lim_{\e\to0} \Bigl(\frac{c_\e^2}{f_\e^2(v_\e(y_\e))} + d_\e\Bigr) = \lim_{\e\to0}\xi_\e(y_\e)= 0,
\end{equation*}
and therefore $c_0^2+d_0=0$. Moreover, we compute
\begin{align*}
\int_0^L \bigg| \frac{(1-v_\e)^2}{4\e} -\e (v_\e')^2\bigg|\dd x
& = 2\int_0^{\midp}|\xi_\e(x)|\dd x 
= - 2\int_0^{y_\e} \xi_\e(x)\dd x + 2\int_{y_\e}^{\midp}\xi_\e(x)\dd x \\
& = - 2\int_0^{y_\e} (c_\e u_\e'+d_\e)\dd x + 2\int_{y_\e}^{\midp} (c_\e u_\e'+d_\e) \dd x \\
& = 2d_\e\bigl(\midp-2y_\e\bigr) + 2c_\e\bigl(u_\e(\midp)-2u_\e(y_\e)\bigr) \\
& = d_\e L +2c_\e u_\e(\midp) - 4\Bigl( d_\e y_\e + c_\e\int_0^{y_\e}u_\e'\dd x \Bigr),
\end{align*}
so that by passing to the limit, and assuming up to subsequences $y_\e\to y_0\in[0,L]$,
\begin{equation} \label{proof:case1c-1}
\lim_{\e\to0}\int_0^L \bigg| \frac{(1-v_\e)^2}{4\e} -\e (v_\e')^2\bigg|\dd x = d_0L + c_0 a -4\bigl(d_0y_0 + c_0^2y_0\bigr) = 0,
\end{equation}
where the last equality follows by the identities $c_0^2+d_0=0$ and $c_0=\frac{a}{L}$. We deduce that
\begin{align*}
\int_0^L \biggl( \frac{(1-v_\e)^2}{4\e} +\e (v_\e')^2 - (1-v_\e)|v_\e'|\biggr)\dd x 
& = \int_0^L \biggl( \frac{1-v_\e}{2\sqrt{\e}} - \sqrt{\e}|v_\e'| \biggr)^2 \dd x \\
& \leq \int_0^L \bigg| \frac{(1-v_\e)^2}{4\e} -\e (v_\e')^2\bigg|\dd x \to 0,
\end{align*}
and in turn
\begin{equation} \label{proof:case1c-2}
\begin{split}
\lim_{\e\to0} \int_0^L \bigg( \frac{(1-v_\e)^2}{4\e} +\e (v_\e')^2\bigg) \dd x
& = \lim_{\e\to0} \int_0^L (1-v_\e)|v_\e|'\dd x 
= \lim_{\e\to0} -2 \int_0^{\midp} (1-v_\e)v_\e'\dd x \\
& = \lim_{\e\to0} (1-v_\e(\midp))^2 = \lim_{\e\to0} (1-m_\e)^2 = 0.
\end{split}
\end{equation}
In view of \eqref{proof:case1c-2}, we conclude that the convergence of the energy holds:
\begin{equation*}
\lim_{\e\to0}\Feps(u_\e,v_\e) \xupref{proof:case1c-2}{=} \lim_{\e\to0}\int_0^L f_\e^2(v_\e)|u_\e'|^2\dd x = c_0^2L = \int_0^L (u')^2\dd x = \F(u,1),
\end{equation*}
where we used the uniform convergences $f_\e(v_\e)\to1$ and $u_\e'\to c_0$, and $u'\equiv c_0\leq\frac{\stress}{2}$.

\medskip\noindent\textit{Step 2: $m_\e< s_\e$.}
In this case we can apply Lemma~\ref{lemma:stima-tempi} to deduce that $|A_\e|\to0$, where $A_\e=\{v_\e<s_\e\}$, and that $u_\e'\to c_0$ uniformly on compact subsets of $[0,L]$ not containing $\midp$, see in particular \eqref{eq:conv-ueps}. For any $x\in(0,\midp)$ we have $u_\e(x)=\int_0^x u_\e'\dd t \to c_0x$, whereas for $x\in(\midp,L)$ we have $u_\e(x)=u_\e(L)-\int_x^L u_\e'\dd t \to a - c_0(L-x)$. Therefore the limit function is
\begin{equation} \label{proof:case1c-4}
u(x)=
\begin{cases}
c_0 x & \text{if }x\in(0,\midp),\\
c_0 x + a-c_0L & \text{if }x\in(\midp,L), 
\end{cases}
\end{equation}
with a possible jump at $\midp$ with amplitude $[u](\midp)=a-c_0L$. Notice that $c_0\leq\frac{a}{L}$ by \eqref{eq:bound-c0}. We next distinguish two further subcases depending on the limit value of $f_\e(m_\e)$.

\medskip\noindent\textit{Step 2a: $f_\e(m_\e)\to0$.}
In this case by Lemma~\ref{lemma:c0} we have $c_0=\frac{\stress}{2}$.
Let us first show that $[u](\midp)=0$. We consider once more the blow-up $\tilde{v}_\e(t)\defeq v_\e(\midp+\e t)$, which obeys the equation \eqref{eqnvPsi} in $(-\frac{x_\e}{\e},\frac{x_\e}{\e})$. Denote by $\tilde{f}(s)\defeq(1-s)f(s)$ and recall that $\tilde{f}$ is strictly increasing by assumption \ref{ass:f3}. By monotonicity of $f$ and $\tilde{f}$, the function $\Psi_\e$ defined in \eqref{eqnvPsi} satisfies
\begin{equation*}
\Psi_\e(s) \geq \frac{\tilde{f}^2(s)-\tilde{f}^2(m_\e)}{4f^2(s)} = \frac{\tilde{f}(s)-\tilde{f}(m_\e)}{4f^2(s)} \cdot \bigl(\tilde{f}(s)+\tilde{f}(m_\e)\bigr) \geq \frac{\tilde{f}(s)-\tilde{f}(m_\e)}{2f^2(s)} \cdot \tilde{f}(m_\e)
\end{equation*}
 for $s\in(m_\e,1)$. Then by arguing as in \eqref{eq:jump} we have, recalling that $c_\e\to\frac{\stress}{2}$, $\tilde{v}_\e(\frac{x_\e}{\e})=s_\e$, and $(\tilde{v}_\e')^2=\Psi_\e(\tilde{v}_\e)$ by \eqref{eqnvPsi},
\begin{align*}
[u](\midp)
& = a-c_0L
= \lim_{\e\to0} 2\int_{\frac{L}{2}}^{L}\frac{c_\e}{f_\e^2(v_\e)}\dd x -c_0L
= \lim_{\e\to0} 2c_\e \int_0^{\frac{x_\e}{\e}} \frac{\dd t}{f^2(\tilde{v_\e})} +c_0L - c_0L\\
& = \lim_{\e\to0} \stress \int_0^{\frac{x_\e}{\e}} \frac{\tilde{v}_\e'}{f^2(\tilde{v_\e})\sqrt{\Psi_\e(\tilde{v}_\e)}}\dd t
= \lim_{\e\to0} \stress \int_{m_\e}^{s_\e} \frac{\dd s}{f^2(s)\sqrt{\Psi_\e(s)}} \\
& \leq \lim_{\e\to0} \frac{\sqrt{2}\stress}{\bigl(\tilde{f}(m_\e)\bigr)^{1/2}} \int_{m_\e}^{s_\e} \frac{\dd s}{f(s)\bigl(\tilde{f}(s)-\tilde{f}(m_\e)\bigr)^{1/2}}\\
& = \lim_{\e\to0}\sqrt{2\stress}\int_{m_\e}^{s_\e} \frac{(1-s)\dd s }{\tilde{f}(s)\bigl(\tilde{f}(s)-\tilde{f}(m_\e)\bigr)^{1/2}}
 \leq \lim_{\e\to0} \frac{\sqrt{2\stress}}{\tilde{f}(m_\e)}\int_{m_\e}^{s_\e} \frac{(1-s)\dd s }{\bigl(\tilde{f}(s)-\tilde{f}(m_\e)\bigr)^{1/2}}\\
& = \lim_{\e\to0} \sqrt{\frac{2}{\stress}} \int_{m_\e}^{s_\e} \frac{(1-s)\dd s }{\bigl(\tilde{f}(s)-\tilde{f}(m_\e)\bigr)^{1/2}} \,.
\end{align*}
We can now write, for $s\in(m_\e,s_\e)$, $\tilde{f}(s)-\tilde{f}(m_\e)=\tilde{f}'(\zeta_\e(s))(s-m_\e)$ for some point $\zeta_\e(s)\in(m_\e,s)$. Since $\zeta_\e(s)\geq m_\e\to1$ and in view of assumption \ref{ass:f5}, given any $M>0$ we have that for all $\e$ small enough
\begin{equation*}
\tilde{f}'(\zeta_\e(s)) \geq M(1-\zeta_\e(s))^3 \geq M(1-s)^3 \qquad\text{for all }s\in(m_\e,s_\e),
\end{equation*} 
hence
\begin{align*}
[u](\midp)
& \leq \lim_{\e\to0} \sqrt{\frac{2}{\stress}} \cdot \frac{1}{M}\int_{m_\e}^{s_\e} \frac{(1-s)\dd s }{\sqrt{s-m_\e}(1-s)^{3/2}} \\
& \leq \lim_{\e\to0} \sqrt{\frac{2}{\stress}} \cdot \frac{1}{M}\int_{m_\e}^{s_\e} \frac{\dd s }{\sqrt{s-m_\e}\sqrt{s_\e-s}} 
= \sqrt{\frac{2}{\stress}} \cdot \frac{\pi}{M}\,.
\end{align*}
Since $M$ is arbitrarily large, we conclude that $[u](\midp)=0$, as claimed. In turn $u(x)=\frac{a}{L}x$ by \eqref{proof:case1c-4}. In particular we also have $\frac{a}{L}=c_0=\frac{\stress}{2}$.

To conclude the proof in this case, we need to show that the convergence of the energy holds. We preliminary show that 
\begin{equation} \label{proof:case1c-5}
\lim_{\e\to0} \int_0^L \e(v_\e'(x))^2\dd x = 0 .
\end{equation}
To this aim, we multiply \eqref{eq:cp1} by the function $v_\e-1$ and integrate in $(0,L)$: we have, after integration by parts,
\begin{align*}
\int_{0}^{L} \e|v_\e'|^2 \dd x
& = \int_0^{L} \frac{c_\e^2 f_\e'(v_\e)}{f_\e^3(v_\e)}(1-v_\e)\dd x - \int_0^L \frac{(1-v_\e)^2}{4\e}\dd x \\
& = \int_0^L \biggl( \frac{4\e c_\e^2 f_\e'(v_\e)}{(1-v_\e)f_\e^3(v_\e)} - 1 \biggr)\frac{(1-v_\e)^2}{4\e}\dd x .
\end{align*}
Now, for $x\in A_\e$ we have $v_\e(x)\leq s_\e$ and by definition of $f_\e$ (see \eqref{def:feps}) and $\fbar$ (see \eqref{def:fbar})
\begin{equation*}
\frac{4\e c_\e^2 f_\e'(v_\e)}{(1-v_\e)f_\e^3(v_\e)} - 1
= (2c_\e)^2\fbar(v_\e)-1
\leq (2c_\e)^2\fbar(m_\e)-1,
\end{equation*}
where we used the monotonicity of $\fbar$ (see Proposition~\ref{prop:salpha}). Similarly, if $x\in A_\e^c$ we have $v_\e(x)\geq s_\e$ and by the monotonicity properties of $\psi_\e$ in assumptions \ref{ass:psi1} and \ref{ass:psi3}
\begin{equation*}
\frac{4\e c_\e^2 f_\e'(v_\e)}{(1-v_\e)f_\e^3(v_\e)} - 1
= \frac{(2c_\e)^2 \e\psi_\e'(v_\e)}{(1-v_\e)\psi_\e^3(v_\e)} - 1
\leq \frac{(2c_\e)^2 \e\psi_\e'(s_\e)}{(1-s_\e)\psi_\e^3(s_\e)} - 1
= (2c_\e)^2\fbar(s_\e)-1.
\end{equation*}
Hence 
\begin{align*}
\int_{0}^{L} \e|v_\e'|^2 \dd x \leq \Bigl((2c_\e)^2\fbar(m_\e)-1\Bigr)\int_0^L \frac{(1-v_\e)^2}{4\e}\dd x \to 0
\end{align*}
since the integral on the right-hand side is uniformly bounded by \eqref{eq:bound-energy}, and $(2c_\e)^2\fbar(m_\e)\to (\frac{2c_0}{\stress})^2=1$ again by Proposition~\ref{prop:salpha}. Hence \eqref{proof:case1c-5} follows.

We next show $c_0^2+d_0=0$. By evaluating \eqref{eq:first-integral2} at $x=0$ we have
\begin{equation*}
c_0^2 + d_0 = \lim_{\e\to0} (c_\e^2 + d_\e ) = \lim_{\e\to0} -\e(v_\e'(0))^2 \leq 0,
\end{equation*}
hence $c_0^2+d_0\leq0$. On the other hand, for every fixed $x_0\in(0,\midp)$ we have that $f_\e(v_\e)\to1$ uniformly in $(0,x_0)$, and therefore
\begin{align*}
(c_0^2+d_0)x_0
& = \lim_{\e\to0} \int_0^{x_0} \biggl(\frac{c_\e^2}{f_\e^2(v_\e)} + d_\e\biggr)\dd x
\xupref{eq:first-integral2}{=} \lim_{\e\to0}\int_0^{x_0} \biggl(\frac{(1-v_\e)^2}{4\e} - \e(v_\e')^2\biggr)\dd x \\
& \geq - \lim_{\e\to0}\int_0^{x_0} \e(v_\e')^2 \dd x \xupref{proof:case1c-5}{=} 0,
\end{align*}
which combined with the inequality obtained before gives $c_0^2+d_0=0$, as desired.

By using \eqref{eq:first-int}, \eqref{eq:ceps}, \eqref{proof:case1c-5}, and the identities $c_0^2+d_0=0$, $c_0=\frac{a}{L}$, we have
\begin{equation} \label{proof:case1c-6}
\begin{split}
\lim_{\e\to0}\int_{0}^{L} \frac{(1-v_\e)^2}{4\e} \dd x 
& = \lim_{\e\to0} \int_{0}^{L} \bigl(\e(v_\e')^2 + c_\e u_\e'\bigr)\dd x + d_0 L\\
& = c_0 a + d_0L = c_0(a-c_0L) = 0,
\end{split}
\end{equation}
and similarly
\begin{equation} \label{proof:case1c-7}
\lim_{\e\to0}\int_{0}^{L} \biggl(\frac{(1-v_\e)^2}{4\e} - f_\e^2(v_\e)(u_\e')^2 \biggr)\dd x
= \lim_{\e\to0} \int_{0}^{L} \e(v_\e')^2 \dd x + d_0 L = -c_0^2L.
\end{equation}
Therefore combining \eqref{proof:case1c-5}, \eqref{proof:case1c-6}, and \eqref{proof:case1c-7} we find
\begin{equation}  \label{proof:case1c-8}
\lim_{\e\to0}\Feps(u_\e,v_\e)
= c_0^2L + 2\lim_{\e\to0}\int_0^L \frac{(1-v_\e)^2}{4\e}\dd x 
= \int_0^L (u')^2\dd x = \F(u,1),
\end{equation}
that is, the convergence of the energy holds.

\medskip\noindent\textit{Step 2b: $f_\e(m_\e)\to\alpha>0$.}
By monotonicity of $f_\e$
\begin{equation*}
\Feps(u_\e,v_\e)\geq\int_0^L f_\e^2(v_\e)|u_\e'|^2\dd x \geq\frac{\alpha^2}{2}\int_0^L |u_\e'|^2\dd x
\end{equation*}
for all $\e$ sufficiently small; by the uniform bound \eqref{eq:bound-energy} we then deduce that $\|u_\e\|_{H^1(0,L)}$ is uniformly bounded, and therefore that the limit $u$ belongs to $H^1(0,L)$. In particular $u$ cannot jump at $\midp$ and by \eqref{proof:case1c-4} we conclude that $u(x)=\frac{a}{L}x$ and $c_0=\frac{a}{L}$.

To conclude the proof, we have to show that also in this case the convergence of the energy holds if and only if $c_0\leq\frac{\stress}{2}$. Assume first that $c_0>\frac{\stress}{2}$: then for every $\delta>0$, by the uniform convergences $u_\e'\to c_0$ and $f_\e(v_\e)\to1$ in $(\midp-\delta,\midp+\delta)^c$ we find
\begin{equation*}
\liminf_{\e\to0}\Feps(u_\e,v_\e)
\geq \liminf_{\e\to0}\int_{(\frac{L}{2}-\delta,\frac{L}{2}+\delta)^c} f_\e^2(v_\e)|u_\e'|^2\dd x 
= c_0^2 (L-2\delta),
\end{equation*}
so that by letting $\delta\to0$ we find
\begin{equation*}
\liminf_{\e\to0}\Feps(u_\e,v_\e) \geq c_0^2L  > \Bigl(\stress c_0 - \frac{\stress^2}{4}\Bigr)L = \int_0^L \Bigl( \stress u' - \frac{\stress^2}{4}\Bigr)\dd x = \F(u,1),
\end{equation*}
that is, the convergence of the energy does not hold. If, instead, $c_0\leq\frac{\stress}{2}$, then one can prove that the convergence of the energy holds just by repeating the argument in Step~2a leading to \eqref{proof:case1c-8}.
\end{proof}


\begin{proof}[Proof of Theorem~\ref{thm:main1}]
The result follows by combining Proposition~\ref{prop:case1b}, Proposition~\ref{prop:case1a} and Proposition~\ref{prop:case1c}.
\end{proof}


\section{Proof of the approximation of critical points} \label{sect:proof2}

In this section we give the proof of Theorem~\ref{thm:main2}. We premise a technical lemma to the proof, which shows that it is possible to construct a solution $v_\e$ to the ODE \eqref{eq:cp1} which attains the boundary conditions $v_\e(0)=v_\e(L)=1$.

\begin{lemma} \label{lemma:construction-ODE}
Let $c_\eps\in(0,\frac{\stress}{2})$ be such that $\sup_\e c_\e<\frac{\stress}{2}$. Then there exists $\e_0>0$ with the following property: for every $\e\in(0,\e_0)$ there exists $m_\e\in(0,z_{c_\e})$, with $(1-m_\e)f(m_\e)\leq 2c_\e$, such that the unique solution to the initial value problem
\begin{equation} \label{ODE:construction}
\begin{cases}
\ds\e v_\e'' = \frac{c_\e^2 f_\e'(v_\e)}{f_\e^3(v_\e)} + \frac{v_\e-1}{4\e} \\[7pt]
v_\e(\midp)=m_\e \\[3pt]
v_\e'(\midp)=0
\end{cases}
\end{equation}
satisfies $v_\e(0)=v_\e(L)=1$. Moreover:
\begin{itemize}
\item if $\inf_\e c_\e>0$ then $\inf_\e m_\e >0$;
\item if $c_\e\to0$ then $m_\e\to0$ and $\frac{2c_\e}{f(m_\e)}\to1$.
\end{itemize}
\end{lemma}

\begin{proof}
Recall that the value $z_{c_\e}\in(0,1)$ appearing in the statement is defined by the relation \eqref{eq:salpha}. Since $\sup_\e c_\e<\frac{\stress}{2}$ we have that $\sup_\e z_{c_\e}<1$ and therefore by choosing $\e_0$ small enough we can guarantee that $z_{c_\e}<s_\e$ for all $\e\in(0,\e_0)$.

For all $m\in(0,z_{c_\e})$ we consider the solution $v_\e(\cdot\,;m)$ of the initial value problem \eqref{ODE:construction} with $v_\e(\midp;m)=m$, which exists and is unique by Cauchy-Lipschitz Theorem. The solution is also symmetric with respect to the point $\midp$. The proof of the lemma amounts to show that we can choose a value $m_\e$ such that $v_\e(L;m_\e)=1$.

Fix any $m\in(0,z_{c_\e})$ such that $(1-m)f(m)\leq 2c_\e$. We first observe that, in the region $\{v_\e<s_\e\}$ (which contains an interval centered at the point $\midp$, since $m<s_\e$ and the solution is symmetric), the rescaled function $\tilde{v}_\e(t)\defeq v_\e(\midp+\e t;m)$ solves the initial value problem \eqref{eq:ODE1}--\eqref{eq:ODE3} studied in Section~\ref{sect:ODE}, for $\alpha=c_\e$. In view of Theorem~\ref{thm:ODE}, since we are assuming $(1-m)f(m)\leq 2c_\e$, the solution reaches the value $s_\e$ in finite time, namely
\begin{equation*}
\forall\,m\in(0,z_{c_\e})\text{ with }(1-m)f(m)\leq 2c_\e \quad \exists\, x_1=x_1(m,\e)>\midp \text{ such that } v_\e(x_1;m)=s_\e.
\end{equation*}
Furthermore, we can estimate $x_1$ by applying Proposition~\ref{prop:ODE} with $\eta=s_\e$: after a rescaling we find
\begin{equation*}
x_1(m,\e) \leq \frac{L}{2} + \frac{C\e}{\sqrt{1-s_\e}}
\end{equation*}
where the constant $C>0$ is independent of $\e$ and $m$, since $\sup_{\e}c_\e<\frac{\stress}{2}$ (see \eqref{eq:est-teta-1}). By \eqref{def:seps}, up to reducing the value of $\e_0$ if necessary, we can therefore guarantee that 
\begin{equation} \label{proof:approx-1}
\frac{L}{2} \leq x_1(m,\e) \leq \frac{3}{4}L
\end{equation}
for all $\e\in(0,\e_0)$ and for all $m\in(0,z_{c_\e})$ such that $(1-m)f(m)\leq 2c_\e$.

In the following argument we work with a fixed $\e\in(0,\e_0)$ and we study the family of solutions $\{v_\e(\cdot\,;m)\}$ depending on the parameter $m$.
By Theorem~\ref{thm:ODE} it also follows that $v_\e$ is strictly increasing in $(\midp,x_1)$ with $v_\e'(x_1;m)>0$. We let
$$
 x_2=x_2(m,\e)\defeq\sup\bigl\{x>x_1(m,\e)\,:\, v_\e(\cdot\,;m)<1 \text{ and } v_\e'(\cdot\,;m)>0 \text{ in } (x_1,x)\bigr\}
 $$
so that $v_\e$ is strictly increasing in $(\midp,x_2)$. By multiplying \eqref{ODE:construction} by $v_\e'$ and integrating in $(\midp,x)$, for $x<x_2$, we find with a change of variables
 \begin{align*}
 \frac{\e}{2}(v_\e'(x))^2 
 & = \int_{\frac{L}{2}}^x \e v_\e' v_\e'' \dd t
 = \int_{\frac{L}{2}}^x \biggl(\frac{c_\e^2 f_\e'(v_\e)}{f_\e^3(v_\e)} + \frac{v_\e-1}{4\e}\biggr) v_\e' \dd t \\
 & = \int_m^{v_\e(x)} \biggl(\frac{c_\e^2 f_\e'(s)}{f_\e^3(s)} + \frac{s-1}{4\e}\biggr) \dd s \\
 & = \frac{c_\e^2}{2}\biggl(\frac{1}{f_\e^2(m)} - \frac{1}{f_\e^2(v_\e(x))}\biggr) + \frac{1}{8\e}\bigl((1-v_\e(x))^2 - (1-m)^2\bigr),
 \end{align*}
 whence
\begin{equation} \label{proof:approx-2}
(v_\e')^2 = \frac{1}{\e^2}\biggl[ \biggl(\frac{(1-v_\e)^2}{4}-\frac{\e c_\e^2}{f_\e^2(v_\e)}\biggr) - \biggl(\frac{(1-m)^2}{4}-\frac{c_\e^2}{f^2(m)}\biggr) \biggr]\eqqcolon H(v_\e;m,\e)
\end{equation}
in $(\midp,x_2(m,\e))$. Let
\begin{equation*}
i(\e) \defeq \inf_{s\in(s_\e,1)} \biggl(\frac{(1-s)^2}{4} - \frac{\e c_\e^2}{f_\e^2(s)} \biggr)
\end{equation*}
and notice for later use that 
\begin{equation} \label{proof:approx-6}
-\frac{\e c_\e^2}{f^2_\e(s_\e)}\leq i(\e) \leq - \e c_\e^2 <0.
\end{equation}
The map $m\mapsto \frac{(1-m)^2}{4}-\frac{c_\e^2}{f^2(m)}$ is strictly increasing for $m\in(0,z_{c_\e})$ (by Proposition~\ref{prop:salpha}), tends to $-\infty$ as $m\to0^+$ and vanishes if $(1-m)f(m)=2c_\e$. Hence there exists a unique $\hat{m}=\hat{m}(\e)\in(0,z_{c_\e})$ such that $(1-\hat{m})f(\hat{m})<2c_\e$ and 
\begin{equation*}
i(\e) - \biggl(\frac{(1-\hat{m}(\e))^2}{4}-\frac{c_\e^2}{f^2(\hat{m}(\e))}\biggr) = 0.
\end{equation*}
By the definition \eqref{proof:approx-2} of the function $H$ it follows that
\begin{equation*}
\inf_{s\in(s_\e,1)} H(s;\hat{m}(\e),\e)=0, \qquad \inf_{s\in(s_\e,1)} H(s;m,\e)>0 \quad\text{for all }m\in(0,\hat{m}(\e)).
\end{equation*}
In turn, since $v_\e(x;m)\in(s_\e,1)$ for $x\in(x_1(m,\e),x_2(m,\e))$, by \eqref{proof:approx-2} we have that
\begin{equation} \label{proof:approx-3}
\inf_{x\in(x_1(m,\e),x_2(m,\e))} v_\e'(x;m) = \inf_{s\in(s_\e,1)} \sqrt{H(s;m,\e)}>0 \qquad\text{ for all } m\in (0,\hat{m}(\e)).
\end{equation}
Then for $m\in (0,\hat{m}(\e))$ the solution $v_\e(\cdot\,;m)$ reaches the value 1 at the finite point $x_2(m,\e)\in(x_1(m,\e),+\infty)$.

Summing up, we have proved so far that for all $m\in(0,\hat{m}(\e))$ there exist two points $x_1(m,\e)\in(\midp,\frac34L)$ and $x_2(m,\e)\in(x_1(m,\e),+\infty)$ such that
$$
v_\e(x_1(m,\e);m) = s_\e, \qquad v_\e(x_2(m,\e);m) = 1.
$$
The goal is now to show the existence of a value $m_\e\in(0,\hat{m}(\e))$ such that $x_2(m(\e),\e)=L$.

By the continuous dependence of the solution to \eqref{ODE:construction} on the initial value $m$, the point $x_2$ is a continuous function of $m$. We can write by \eqref{proof:approx-2}
\begin{equation} \label{proof:approx-4}
\begin{split}
x_2(m,\e)
& = x_1(m,\e) + \int_{x_1(m,\e)}^{x_2(m,\e)} \frac{v_\e'(x;m)}{\sqrt{H(v_\e(x;m);m,\e)}}\dd x \\
& = x_1(m,\e) + \int_{s_\e}^1 \frac{\dd s}{\sqrt{H(s;m,\e)}} \\
& \leq \frac{3}{4}L + (1-s_\e)\biggl(\inf_{s\in(s_\e,1)}H(s;m,\e)\biggr)^{-\frac12}
\end{split}
\end{equation}
and since $\lim_{m\to0^+}\inf_{s\in(s_\e,1)}H(s;m,\e)=+\infty$, we see that $x_2(m,\e)<L$ for all $m$ sufficiently small. On the other hand, for $m=\hat{m}(\e)$ we have 
$$
\inf_{x\in(x_1(\hat{m}(\e),\e),x_2(\hat{m}(\e),\e))} v_\e'(x;\hat{m}(\e)) = \inf_{s\in(s_\e,1)} \sqrt{H(s;\hat{m}(\e),\e)}=0
$$
and therefore $x_2(\hat{m}(\e),\e)=+\infty$, which implies $\lim_{m\to\hat{m}(\e)^-}x_2(m,\e)=+\infty$. By continuity of $m\mapsto x_2(m,\e)$, we conclude that there exists $m_\e\in(0,\hat{m}(\e))$ such that $x_2(m_\e,\e)=L$ and therefore $v_\e(L;m_\e)=1$, as claimed.

We eventually prove the second part of the statement.
By \eqref{proof:approx-4} it also follows that
\begin{equation*} 
\frac{1}{\e^2}\biggl[i(\e)-\biggl(\frac{(1-m_\e)^2}{4}-\frac{c_\e^2}{f^2(m_\e)}\biggr)\biggr] = \inf_{s\in(s_\e,1)} H(s;m_\e,\e) \leq \frac{16(1-s_\e)^2}{L^2}\,.
\end{equation*}
By elementary manipulations in the previous inequality, and recalling that by construction $(1-m_\e)f(m_\e)<2c_\e$, we find
\begin{equation} \label{proof:approx-5}
\frac{(1-m_\e)^2}{4} \leq \frac{c_\e^2}{f^2(m_\e)} \leq \frac{16\e^2(1-s_\e)^2}{L^2} - i(\e) + \frac{(1-m_\e)^2}{4}\,.
\end{equation}
Suppose that $\inf_\e c_\e>0$. If by contradiction $m_\e\to0$, then by passing to the limit as $\e\to0$ in \eqref{proof:approx-5} we would have that the middle term would tend to $+\infty$, whereas the right-hand side would tend to $\frac14$ (since $i(\e)\to0$ by \eqref{proof:approx-6}). This contradiction proves that if $\inf_\e c_\e>0$ then $\inf_\e m_\e >0$.

Similarly, if  $c_\e\to0$ then by $(1-m_\e)f(m_\e)<2c_\e$ we must have $m_\e\to0$. Again by passing to the limit as $\e\to0$ in \eqref{proof:approx-5} we easily deduce that $\frac{2c_\e}{f(m_\e)}\to1$.
\end{proof}

\begin{proof}[Proof of Theorem~\ref{thm:main2}]
We divide the proof into three cases according to the form of the critical point $\bar{u}$, as in the statement of the theorem.

\medskip\noindent\textit{Case }(i). Assume $\bar{u}(x)=\frac{a}{L}x$ for some $a>0$. In this case it is sufficient to take $u_\e(x)=\frac{a}{L}x$ and $v_\e(x)\equiv1$. Since $f_\eps'(1)=\psi_\e'(1)=0$, it is immediately checked that the pair $(u_\e,v_\e)$ is indeed a solution to \eqref{def:cp1}--\eqref{def:cp4}. 

\medskip\noindent\textit{Case }(ii). Assume $\bar{u}(x)=c_0x + (a-c_0L)\chi_{(\frac{L}{2},L)}(x)$ with $c_0\in(0,\frac{\stress}{2})$ and $g'(a-c_0L)=2c_0$.

We apply Lemma~\ref{lemma:construction-ODE} with $c_\e=c_0$ for all $\e$, to find values $m_\e$ and functions $v_\e$ solving \eqref{ODE:construction} such that $v_\e(0)=v_\e(L)=1$. Notice also that $m_0\defeq\inf_\e m_\e>0$. We define
$$
u_\e(x)\defeq \int_0^x \frac{c_0}{f_\e^2(v_\e(x))}\dd x 
$$
and we obtain that $(u_\e,v_\e)$ is a family of critical points for $\Feps$, i.e.\ they solve the system of equations \eqref{def:cp1}--\eqref{def:cp4} for $a_\e\defeq \int_0^L \frac{c_0}{f_\e^2(v_\e)}\dd x$. To conclude, we need to show that $a_\e\to a$ and that $u_\e\to \bar{u}$ in $L^1([0,L])$.

To this aim, we first show that the equiboundedness of the energy \eqref{eq:bound-energy} holds for the family $(u_\e,v_\e)$. By the construction in Lemma~\ref{lemma:construction-ODE} the function $v_\e$ obeys the equation \eqref{proof:approx-2} (with $c_\e=c_0$ and $m=m_\e$). Denoting by $H_\e(s)\defeq H(s; m_\e,\e)$ the function appearing in \eqref{proof:approx-2}, we have for all $s\in(m_\e,1)$ after some elementary manipulations
\begin{align*}
H_\e(s) & = \frac{1}{4\e^2 f_\e^2(s)}\bigl[ (1-s)^2f_\e^2(s) - (1-m_\e)^2f_\e^2(m_\e) \bigr] + \frac{f_\e^2(s)-f_\e^2(m_\e)}{4\e^2f_\e^2(s)}\Bigl( \frac{4c_0^2}{f^2(m_\e)}-(1-m_\e)^2\Bigr) \\
& \geq  \frac{1}{4\e^2 f_\e^2(s)}\Bigl( (1-s)^2f_\e^2(s) - (1-m_\e)^2f_\e^2(s) \Bigr)
\end{align*}
where we used the monotonicity of the map $f_\e$ and the fact that $(1-m_\e)f(m_\e)\leq 2c_0$ by the construction in Lemma~\ref{lemma:construction-ODE}.
Then, denoting by $\tilde{f}(s)\defeq(1-s)f(s)$ (which is strictly increasing by assumption \ref{ass:f3}), we have
\begin{equation}\label{proof:approx-3b}
\begin{split}
\int_0^L f_\e^2(v_\e)(u_\e')^2\dd x
& = c_0^2\int_{0}^{L}\frac{\dd x}{f_\e^2(v_\e(x))}
\leq c_0^2 \int_{\{v_\e\leq s_\e\}} \frac{|v_\e'|}{f_\e^2(v_\e)\sqrt{H_\e(v_\e)}}\dd x + \frac{c_0^2\,|\{v_\e>s_\e\}|}{f_\e^2(s_\e)} \\
& \leq 2c_0^2 \int_{m_\e}^{s_\e} \frac{\dd s}{f_\e^2(s)\sqrt{H_\e(s)}} + \frac{Lc_0^2}{f_\e^2(s_\e)} \\
& \leq 4c_0^2\int_{m_\e}^{s_\e} \frac{\dd s}{f(s)\bigl(\tilde{f}^2(s)-\tilde{f}^2(m_\e)\bigr)^{1/2}} + \frac{Lc_0^2}{f_\e^2(s_\e)}\,.
\end{split}
\end{equation}
For $s\in(m_\e,s_\e)$ we write $\tilde{f}(s)-\tilde{f}(m_\e)=\tilde{f}'(\zeta_\e(s))(s-m_\e)$ for some point $\zeta_\e(s)\in(m_\e,s)$. Since $\zeta_\e(s)\geq m_\e\geq m_0>0$, in view of assumptions \ref{ass:f3} and \ref{ass:f5} we have that $\tilde{f}'(\zeta_\e(s)) \geq C(1-\zeta_\e(s))^3 \geq C(1-s)^3$ for all $s\in(m_\e,s_\e)$ and for a constant $C$ independent of $\e$. Therefore
\begin{equation} \label{proof:approx-1b}
\begin{split}
\int_0^L f_\e^2(v_\e)(u_\e')^2\dd x
& \leq \frac{4c_0^2}{\sqrt{2C\tilde{f}(m_\e)}}\int_{m_\e}^{s_\e} \frac{\dd s}{f(s)\sqrt{s-m_\e}(1-s)^{3/2}} + \frac{Lc_0^2}{f_\e^2(s_\e)} \\
& \leq \frac{4c_0^2}{\sqrt{2C\tilde{f}(m_\e)}} \frac{1}{\tilde{f}(m_\e)}\int_{m_\e}^{s_\e} \frac{\dd s}{\sqrt{s-m_\e}\sqrt{1-s}} + \frac{Lc_0^2}{f_\e^2(s_\e)}
\leq C'
\end{split}
\end{equation}
for another constant $C'$ uniform in $\e$. By multiplying \eqref{ODE:construction} by $(v_\e-1)$ and integrating, after integration by parts we have
\begin{equation}\label{proof:approx-2b}
\begin{split}
\int_0^L \biggl(\e(v_\e')^2 + \frac{(1-v_\e)^2}{4\e}\biggr)\dd x
& = c_0^2\int_0^L \frac{(1-v_\e)f_\e'(v_\e)}{f_\e^3(v_\e)}\dd x \\
& \leq \sup_{s\in(m_\e,1)}\biggl(\frac{(1-s)f_\e'(s)}{f_\e(s)}\biggr) \int_0^L f_\e^2(v_\e)(u_\e')^2\dd x \leq C''
\end{split}
\end{equation}
for a constant $C''$ independent of $\e$, where the last estimate follows by \eqref{proof:approx-1b} and from the assumptions \ref{ass:f4}, \ref{ass:psi1}, \ref{ass:psi2}, \ref{ass:psi3} and by \eqref{def:seps} (recalling that $m_\e\geq m_0>0$).

Combining \eqref{proof:approx-1b} and \eqref{proof:approx-2b} we obtain that $\sup_{\e}\Feps(u_\e,v_\e)<+\infty$. The critical points $(u_\e,v_\e)$ then satisfy the assumption of Theorem~\ref{thm:main1}. Since $\lim_{\e\to 0}m_\e\in(0,1)$, we are in case~\ref{thm:main-case2} and we can conclude that up to extraction of a subsequence $u_{\e_k}\to u_0$ in $L^1([0,L])$, where $u_0(x)=c_0x+(a_0-c_0L)\chi_{(\frac{L}{2},L)}(x)$, $a_0=\lim_{k}a_{\e_k}$, and $g'(a_0-c_0L)=2c_0$. Since we also have $g'(a-c_0L)=2c_0$ and $g'$ is injective in $(0,\sfrac)$ by Proposition~\ref{prop:g3}, we conclude that $a_0=a$ and $u_0=\bar{u}$. Hence, as the limit of any subsequence of $u_\e$ converges to $\bar{u}$, we conclude that $u_\e\to\bar{u}$ in $L^1([0,L])$.

\medskip\noindent\textit{Case }(iii). Assume that $\sfrac$ is finite (i.e.\ $f'(0)>0$ by Proposition~\ref{prop:sfrac}) and that $\bar{u}(x)=a\chi_{(\frac{L}{2},L)}(x)$ with $a=\sfrac$.

Take any sequence $c_\e\to0$, with $c_\e>0$, and apply Lemma~\ref{lemma:construction-ODE} to find values $m_\e$ and functions $v_\e$ solving \eqref{ODE:construction} such that $v_\e(0)=v_\e(L)=1$. Notice also that $m_\e\to0$ and $\frac{2c_\e}{f(m_\e)}\to1$. We define
$$
u_\e(x)\defeq \int_0^x \frac{c_\e}{f_\e^2(v_\e(x))}\dd x 
$$
and we obtain that $(u_\e,v_\e)$ is a family of critical points for $\Feps$, i.e.\ they solve the system of equations \eqref{def:cp1}--\eqref{def:cp4} for $a_\e\defeq \int_0^L \frac{c_\e}{f_\e^2(v_\e)}\dd x$. To conclude, we need to show that $a_\e\to \sfrac$ and that $u_\e\to \bar{u}$ in $L^1([0,L])$.

As in the previous step, we first show that the equiboundedness of the energy \eqref{eq:bound-energy} holds for the family $(u_\e,v_\e)$. We indeed have, similarly to \eqref{proof:approx-3b}, for $\delta>0$,
\begin{equation*}
\begin{split}
a_\e
& = \int_0^L \frac{c_\e}{f_\e^2(v_\e)}\dd x
\leq 4c_\e\int_{m_\e}^{\delta} \frac{\dd s}{f(s)\bigl(\tilde{f}^2(s)-\tilde{f}^2(m_\e)\bigr)^{1/2}}
+ \frac{4c_\e(1-\delta)}{f(\delta)\bigl(\tilde{f}^2(\delta)-\tilde{f}^2(m_\e)\bigr)^{1/2}} + \frac{Lc_\e}{f_\e^2(s_\e)} \\
& \leq 4c_\e \sup_{s\in(m_\e,\delta)}\biggl(\frac{1-s}{\tilde{f}'(s)}\biggr)\int_{m_\e}^{\delta} \frac{\tilde{f}'(s)\dd s}{\tilde{f}(s)\bigl(\tilde{f}^2(s)-\tilde{f}^2(m_\e)\bigr)^{1/2}} +C_\delta c_\e \\
& = 4c_\e \sup_{s\in(m_\e,\delta)}\biggl(\frac{1-s}{\tilde{f}'(s)}\biggr)\int_{\tilde{f}(m_\e)}^{\tilde{f}(\delta)} \frac{\dd t}{t\bigl(t^2-\tilde{f}^2(m_\e)\bigr)^{1/2}} + C_\delta c_\e \\
& = \frac{4c_\e}{\tilde{f}(m_\e)} \sup_{s\in(m_\e,\delta)}\biggl(\frac{1-s}{\tilde{f}'(s)}\biggr)\arctan\Biggl(\frac{\bigl(\tilde{f}^2(\delta)-\tilde{f}^2(m_\e)\bigr)^{1/2}}{\tilde{f}(m_\e)}\Biggr) + C_\delta c_\e,
\end{split}
\end{equation*}
where $C_\delta$ is a constant depending on $\delta$, for $\e$ is small enough.
Now, using the fact that $c_\e\to0$, $m_\e\to0$, $f'(0)>0$ and that $\frac{2c_\e}{f(m_\e)}\to1$, one can see that the right-hand side in the previous chain of inequalities is uniformly bounded. Therefore
\begin{equation} \label{proof:approx-4b}
\sup_{\e}a_\e <+\infty, 
\qquad
\lim_{\e\to0}\int_0^L f_\e^2(v_\e)(u_\e')^2\dd x = \lim_{\e\to0}c_\e a_\e = 0.
\end{equation}
Coming to the energy of $v_\e$, we fix $\delta\in(0,1)$ and as in \eqref{proof:approx-2b} we have
\begin{equation*}
\begin{split}
\int_0^L \biggl(\e(v_\e')^2 + \frac{(1-v_\e)^2}{4\e}\biggr)\dd x
& = c_\e^2\int_0^L \frac{(1-v_\e)f_\e'(v_\e)}{f_\e^3(v_\e)}\dd x \\
& \leq \frac{c_\e}{f(m_\e)}\biggl(\sup_{s\in(m_\e,\delta)}f'(s)\biggr)\int_{\{v_\e\leq\delta\}} \frac{ c_\e }{f_\e^2(v_\e)}\dd x \\
& \qquad + \sup_{s\in(\delta,1)}\biggl(\frac{(1-s)f_\e'(s)}{f_\e(s)}\biggr) \int_{\{v_\e>\delta\}} f_\e^2(v_\e)(u_\e')^2\dd x.
\end{split}
\end{equation*}
Again by \eqref{proof:approx-4b}, by $\frac{2c_\e}{f(m_\e)}\to1$, and by all the assumptions on $f_\e$, it is possible to check that the previous quantities are uniformly bounded with respect to $\e$. Hence $\sup_{\e}\Feps(u_\e,v_\e)<+\infty$.

The critical points $(u_\e,v_\e)$ then satisfy the assumption of Theorem~\ref{thm:main1}. Since $\lim_{\e\to 0}m_\e=0$, we are in case~\ref{thm:main-case3} and, as at the end of the previous step, we can conclude that $u_\e\to\bar{u}$ in $L^1([0,L])$ and $a=\sfrac$.
\end{proof}


\bigskip
\subsection*{Acknowledgments}
The authors are thankful to Cinzia Soresina for fruitful conversations and to Matteo Focardi for valuable suggestions. MB and FI are members of the GNAMPA group of the Istituto Nazionale di Alta Matematica (INdAM).


\bibliographystyle{siam}
\def\url#1{}
\bibliography{references}

\begin{thebibliography}{10}

\bibitem{AleMarVid14}
{\sc R.~Alessi, J.-J. Marigo, and S.~Vidoli}, {\em Gradient damage models
  coupled with plasticity and nucleation of cohesive cracks}, Arch. Ration.
  Mech. Anal., 214 (2014), pp.~575--615.

\bibitem{ABS}
{\sc R.~Alicandro, A.~Braides, and J.~Shah}, {\em Free-discontinuity problems
  via functionals involving the {$L^1$}-norm of the gradient and their
  approximations}, Interfaces Free Bound., 1 (1999), pp.~17--37.

\bibitem{AF}
{\sc R.~Alicandro and M.~Focardi}, {\em Variational approximation of
  free-discontinuity energies with linear growth}, Commun. Contemp. Math., 4
  (2002), pp.~685--723.

\bibitem{Alm17}
{\sc S.~Almi}, {\em Energy release rate and quasi-static evolution {\it via}
  vanishing viscosity in a fracture model depending on the crack opening},
  ESAIM Control Optim. Calc. Var., 23 (2017), pp.~791--826.

\bibitem{AlmBelNeg19}
{\sc S.~Almi, S.~Belz, and M.~Negri}, {\em Convergence of discrete and
  continuous unilateral flows for {A}mbrosio-{T}ortorelli energies and
  application to mechanics}, ESAIM Math. Model. Numer. Anal., 53 (2019),
  pp.~659--699.

\bibitem{AmbTor90}
{\sc L.~Ambrosio and V.~M. Tortorelli}, {\em Approximation of functionals
  depending on jumps by elliptic functionals via {$\Gamma$}-convergence}, Comm.
  Pure Appl. Math., 43 (1990), pp.~999--1036.

\bibitem{AmbTor92}
{\sc L.~Ambrosio and V.~M. Tortorelli}, {\em On the approximation of free
  discontinuity problems}, Boll. Un. Mat. Ital. B (7), 6 (1992), pp.~105--123.

\bibitem{ACFS}
{\sc M.~Artina, F.~Cagnetti, M.~Fornasier, and F.~Solombrino}, {\em Linearly
  constrained evolutions of critical points and an application to cohesive
  fractures}, Math. Models Methods Appl. Sci., 27 (2017), pp.~231--290.

\bibitem{BabMill2014}
{\sc J.-F. Babadjian and V.~Millot}, {\em Unilateral gradient flow of the
  {A}mbrosio-{T}ortorelli functional by minimizing movements}, Ann. Inst. H.
  Poincar\'{e} C Anal. Non Lin\'{e}aire, 31 (2014), pp.~779--822.

\bibitem{BabMilRodb}
{\sc J.-F. Babadjian, V.~Millot, and R.~Rodiac}, {\em A note on the
  one-dimensional critical points of the {A}mbrosio-{T}ortorelli functional},
  Asymptot. Anal., 135 (2023), pp.~349--362.

\bibitem{BabMilRoda}
\leavevmode\vrule height 2pt depth -1.6pt width 23pt, {\em On the convergence
  of critical points of the {A}mbrosio-{T}ortorelli functional}, Ann. Inst. H.
  Poincar\'{e} C Anal. Non Lin\'{e}aire,  (2023).

\bibitem{BonConIur21}
{\sc M.~Bonacini, S.~Conti, and F.~Iurlano}, {\em Cohesive fracture in 1{D}:
  quasi-static evolution and derivation from static phase-field models}, Arch.
  Ration. Mech. Anal., 239 (2021), pp.~1501--1576.

\bibitem{BCFR2022}
{\sc E.~Bonetti, C.~Cavaterra, F.~Freddi, and F.~Riva}, {\em On a phase-field
  model of damage for hybrid laminates with cohesive interface}, Math. Methods
  Appl. Sci., 45 (2022), pp.~3520--3553.

\bibitem{BouBraBut95}
{\sc G.~Bouchitt\'{e}, A.~Braides, and G.~Buttazzo}, {\em Relaxation results
  for some free discontinuity problems}, J. Reine Angew. Math., 458 (1995),
  pp.~1--18.

\bibitem{Bourdin}
{\sc B.~Bourdin}, {\em Image segmentation with a finite element method}, M2AN
  Math. Model. Numer. Anal., 33 (1999), pp.~229--244.

\bibitem{BFM}
{\sc B.~Bourdin, G.~A. Francfort, and J.-J. Marigo}, {\em The variational
  approach to fracture}, J. Elasticity, 91 (2008), pp.~5--148.

\bibitem{BraDMGar99}
{\sc A.~Braides, G.~Dal~Maso, and A.~Garroni}, {\em Variational formulation of
  softening phenomena in fracture mechanics: the one-dimensional case}, Arch.
  Ration. Mech. Anal., 146 (1999), pp.~23--58.

\bibitem{CCF}
{\sc L.~Caffarelli, F.~Cagnetti, and A.~Figalli}, {\em Optimal regularity and
  structure of the free boundary for minimizers in cohesive zone models}, Arch.
  Ration. Mech. Anal., 237 (2020), pp.~299--345.

\bibitem{C}
{\sc F.~Cagnetti}, {\em A vanishing viscosity approach to fracture growth in a
  cohesive zone model with prescribed crack path}, Math. Models Methods Appl.
  Sci., 18 (2008), pp.~1027--1071.

\bibitem{CT}
{\sc F.~Cagnetti and R.~Toader}, {\em Quasistatic crack evolution for a
  cohesive zone model with different response to loading and unloading: a
  {Y}oung measures approach}, ESAIM Control Optim. Calc. Var., 17 (2011),
  pp.~1--27.

\bibitem{ConFocIur16}
{\sc S.~Conti, M.~Focardi, and F.~Iurlano}, {\em Phase field approximation of
  cohesive fracture models}, Ann. Inst. H. Poincar\'{e} C Anal. Non
  Lin\'{e}aire, 33 (2016), pp.~1033--1067.

\bibitem{ConFocIur22}
{\sc S.~Conti, M.~Focardi, and F.~Iurlano}, {\em Phase-field approximation of a
  vectorial, geometrically nonlinear cohesive fracture energy}, Arch. Ration.
  Mech. Anal., 248 (2024), pp.~Paper No. 21, 60.

\bibitem{CLO}
{\sc V.~Crismale, G.~Lazzaroni, and G.~Orlando}, {\em Cohesive fracture with
  irreversibility: quasistatic evolution for a model subject to fatigue}, Math.
  Models Methods Appl. Sci., 28 (2018), pp.~1371--1412.

\bibitem{DMG}
{\sc G.~Dal~Maso and A.~Garroni}, {\em Gradient bounds for minimizers of free
  discontinuity problems related to cohesive zone models in fracture
  mechanics}, Calc. Var. Partial Differential Equations, 31 (2008),
  pp.~137--145.

\bibitem{DMOT}
{\sc G.~Dal~Maso, G.~Orlando, and R.~Toader}, {\em Fracture models for
  elasto-plastic materials as limits of gradient damage models coupled with
  plasticity: the antiplane case}, Calc. Var. Partial Differential Equations,
  55 (2016), pp.~Art. 45, 39.

\bibitem{DMZ}
{\sc G.~Dal~Maso and C.~Zanini}, {\em Quasi-static crack growth for a cohesive
  zone model with prescribed crack path}, Proc. Roy. Soc. Edinburgh Sect. A,
  137 (2007), pp.~253--279.

\bibitem{DelPie13}
{\sc G.~Del~Piero}, {\em A variational approach to fracture and other inelastic
  phenomena}, J. Elasticity, 112 (2013), pp.~3--77.

\bibitem{DPT1998}
{\sc G.~Del~Piero and L.~Truskinovsky}, {\em A one-dimensional model for
  localized and distributed failure}, J. Phys. IV France, 8 (1998).

\bibitem{DelPieTru09}
\leavevmode\vrule height 2pt depth -1.6pt width 23pt, {\em Elastic bars with
  cohesive energy}, Contin. Mech. Thermodyn., 21 (2009), pp.~141--171.

\bibitem{FraLeSer09}
{\sc G.~A. Francfort, N.~Q. Le, and S.~Serfaty}, {\em Critical points of
  {A}mbrosio-{T}ortorelli converge to critical points of {M}umford-{S}hah in
  the one-dimensional {D}irichlet case}, ESAIM Control Optim. Calc. Var., 15
  (2009), pp.~576--598.

\bibitem{FM1998}
{\sc G.~A. Francfort and J.-J. Marigo}, {\em Revisiting brittle fracture as an
  energy minimization problem}, J. Mech. Phys. Solids, 46 (1998),
  pp.~1319--1342.

\bibitem{FreIur}
{\sc F.~Freddi and F.~Iurlano}, {\em Numerical insight of a variational smeared
  approach to cohesive fracture}, J. Mech. Phys. Solids, 98 (2017),
  pp.~156--171.

\bibitem{G2005}
{\sc A.~Giacomini}, {\em Ambrosio-{T}ortorelli approximation of quasi-static
  evolution of brittle fractures}, Calc. Var. Partial Differential Equations,
  22 (2005), pp.~129--172.

\bibitem{HutTon00}
{\sc J.~E. Hutchinson and Y.~Tonegawa}, {\em Convergence of phase interfaces in
  the van der {W}aals-{C}ahn-{H}illiard theory}, Calc. Var. Partial
  Differential Equations, 10 (2000), pp.~49--84.

\bibitem{LCM}
{\sc H.~Lammen, S.~Conti, and J.~Mosler}, {\em A finite deformation phase field
  model suitable for cohesive fracture}, J. Mech. Phys. Solids, 178 (2023),
  pp.~Paper No. 105349, 17.

\bibitem{LS}
{\sc C.~J. Larsen and V.~Slastikov}, {\em Dynamic cohesive fracture: models and
  analysis}, Math. Models Methods Appl. Sci., 24 (2014), pp.~1857--1875.

\bibitem{Le10}
{\sc N.~Q. Le}, {\em Convergence results for critical points of the
  one-dimensional {A}mbrosio-{T}ortorelli functional with fidelity term}, Adv.
  Differential Equations, 15 (2010), pp.~255--282.

\bibitem{LeSte19}
{\sc N.~Q. Le and P.~J. Sternberg}, {\em Asymptotic behavior of
  {A}llen-{C}ahn-type energies and {N}eumann eigenvalues via inner variations},
  Ann. Mat. Pura Appl. (4), 198 (2019), pp.~1257--1293.

\bibitem{LucMod89}
{\sc S.~Luckhaus and L.~Modica}, {\em The {G}ibbs-{T}hompson relation within
  the gradient theory of phase transitions}, Arch. Rational Mech. Anal., 107
  (1989), pp.~71--83.

\bibitem{M1987}
{\sc L.~Modica}, {\em The gradient theory of phase transitions and the minimal
  interface criterion}, Arch. Rational Mech. Anal., 98 (1987), pp.~123--142.

\bibitem{NS2017}
{\sc M.~Negri and R.~Scala}, {\em A quasi-static evolution generated by local
  energy minimizers for an elastic material with a cohesive interface},
  Nonlinear Anal. Real World Appl., 38 (2017), pp.~271--305.

\bibitem{NS2021}
{\sc M.~Negri and R.~Scala}, {\em Existence, energy identity, and higher time
  regularity of solutions to a dynamic viscoelastic cohesive interface model},
  SIAM J. Math. Anal., 53 (2021), pp.~5682--5730.

\bibitem{PadTon98}
{\sc P.~Padilla and Y.~Tonegawa}, {\em On the convergence of stable phase
  transitions}, Comm. Pure Appl. Math., 51 (1998), pp.~551--579.

\bibitem{R2023}
{\sc F.~Riva}, {\em Energetic evolutions for linearly elastic plates with
  cohesive slip}, Nonlinear Anal. Real World Appl., 74 (2023), pp.~Paper No.
  103934, 27.

\bibitem{RoeTon08}
{\sc M.~R\"{o}ger and Y.~Tonegawa}, {\em Convergence of phase-field
  approximations to the {G}ibbs-{T}homson law}, Calc. Var. Partial Differential
  Equations, 32 (2008), pp.~111--136.

\bibitem{R2019}
{\sc T.~Roub\'{\i}\v{c}ek}, {\em A general thermodynamical model for adhesive
  frictional contacts between viscoelastic or poro-viscoelastic bodies at small
  strains}, Interfaces Free Bound., 21 (2019), pp.~169--198.

\bibitem{T2018}
{\sc M.~Thomas}, {\em A comparison of delamination models: modeling,
  properties, and applications}, in Mathematical analysis of continuum
  mechanics and industrial applications. {II}, vol.~30 of Math. Ind. (Tokyo),
  Springer, Singapore, 2018, pp.~27--38.

\bibitem{TZ2017}
{\sc M.~Thomas and C.~Zanini}, {\em Cohesive zone-type delamination in
  visco-elasticity}, Discrete Contin. Dyn. Syst. Ser. S, 10 (2017),
  pp.~1487--1517.

\bibitem{Ton02}
{\sc Y.~Tonegawa}, {\em Phase field model with a variable chemical potential},
  Proc. Roy. Soc. Edinburgh Sect. A, 132 (2002), pp.~993--1019.

\bibitem{Ton05}
\leavevmode\vrule height 2pt depth -1.6pt width 23pt, {\em A diffused interface
  whose chemical potential lies in a {S}obolev space}, Ann. Sc. Norm. Super.
  Pisa Cl. Sci. (5), 4 (2005), pp.~487--510.

\bibitem{Wu}
{\sc J.-Y. Wu}, {\em A unified phase-field theory for the mechanics of damage
  and quasi-brittle failure}, J. Mech. Phys. Solids, 103 (2017), pp.~72--99.

\bibitem{WuNgu}
{\sc J.-Y. Wu and V.~P. Nguyen}, {\em A length scale insensitive phase-field
  damage model for brittle fracture}, J. Mech. Phys. Solids, 119 (2018),
  pp.~20--42.

\end{thebibliography}

\end{document}